\documentclass[reqno, intlimits, sumlimits]{amsart} 

\usepackage{amsmath, amstext, amssymb, amsthm, bbm, lmodern, enumerate, esint, lipsum,color,todonotes}
\usepackage{mathrsfs}
\usepackage[left=3cm,right=3cm, top=4cm, bottom=3cm]{geometry}

\usepackage[hidelinks]{hyperref}
\usepackage[symbols,nogroupskip,sort=def]{glossaries-extra}

\allowdisplaybreaks

	\renewcommand{\div}{\textup{div}\,}
    \renewcommand{\phi}{\varphi}
    \renewcommand{\epsilon}{\varepsilon}
    \renewcommand{\P}{\mathsf{P}}
 
    \newcommand	{\eins}	{\mathbbm{1}}   
	\newcommand	{\norm}[1]	{\left\lVert#1\right\rVert}
	\newcommand	{\abs}[1]		{\left\lvert#1\right\rvert}
	\newcommand {\ent}	{\mathsf{Ent}}

\newcommand{\Cpl}{ \mathsf{Cpl}}

\newcommand{\Q}{ \mathsf{Q}}

\newcommand{\q}{ \mathsf{q}}
\newcommand{\sq}{ \mathsf{sq}}

\newcommand{\Poi}{\mathsf{Poi}}
\newcommand{\Leb}{\mathsf{Leb}}
\newcommand{\pr}{\mathrm{pr}}

\newcommand{\Semi}{\mathsf{S}}

\newcommand{\U}{\mathsf U}

\newcommand{\s}{ \mathsf{s}}
\newcommand{\SQ}{ \mathsf{S}}
\newcommand{\x}{\bar\tau }
\newcommand{\wint}{\mathcal{W}^{int}}

\newcommand{\PR}{\mathrm{Pr}^{\Gamma}}
\renewcommand{\j}{ \mathsf{gap}}
\newcommand{\g}{ \mathsf{gain}}

	\DeclareMathOperator	{\IE}			{\mathbb{E}} 
	\DeclareMathOperator	{\IN}			{\mathbb{N}}
	\DeclareMathOperator	{\IP}			{\mathbb{P}}
	
	\DeclareMathOperator	{\IR}			{\mathbb{R}}
	\DeclareMathOperator	{\IZ}			{\mathbb{Z}}
	
	\DeclareMathOperator	{\supp}			{supp}

\DeclareMathOperator	{\paste}			{\mathsf{paste}}

	\theoremstyle{plain}
\newtheorem{thm}			{Theorem}[section]
\newtheorem{lem}	[thm]	{Lemma}
\newtheorem{cor}	[thm]	{Corollary}
\newtheorem{prop}	[thm]	{Proposition}

\theoremstyle{definition}
\newtheorem{defi}	[thm]	{Definition}
\newtheorem{ex}	    [thm]	{Example}
\newtheorem{rem}	[thm]	{Remark}
\newtheorem{assumption}[thm]{Assumption}

\numberwithin{equation}{section}

\begin{document}
\title[Gap metrics for stationary point processes]{Gap metrics for stationary point processes and quantitative convexity of the free energy}

\author[M. Huesmann]{Martin Huesmann}
\address{M.H.: Universit\"at M\"unster, Germany}
\email{martin.huesmann@uni-muenster.de}
\author[B. Müller]{Bastian Müller}
\address{B.M.: Universit\"at M\"unster, Germany   }

\thanks{MH and BM are funded by the Deutsche Forschungsgemeinschaft (DFG, German Research Foundation) under Germany's Excellence Strategy EXC 2044 -390685587, Mathematics M\"unster: Dynamics--Geometry--Structure and   by the DFG through the SPP 2265 {\it Random Geometric Systems. }}

\begin{abstract}
In this article, we are interested in convexity properties of the free energy for stationary point processes on $\IR$ w.r.t.\ a new geometry inspired by optimal transport. We will show for a rich class of pairwise interaction energies
\begin{itemize}
    \item[A] quantified strict convexity of the free energy implying uniqueness of minimizers
    \item[B] existence of a gradient flow curve of the free energy w.r.t.\ the new metric converging exponentially fast to the unique minimizer.
\end{itemize}
Examples for energies for which A holds include logarithmic or Riesz interactions with parameter $0<s<1$, examples for which A and B hold are hypersingular Riesz or Yukawa interactions.
\end{abstract}

\date{\today}
\maketitle

\maketitle

\section{Introduction}
Let $\xi$ be a stationary point process on $\IR$ with distribution $\P.$ For an even interaction potential $\varphi:\IR\to\IR$ and $\Lambda_n=[-n/2,n/2]$ we put 
$$H_n(\xi)=\frac12 \sum_{x,y\in \Lambda_n\cap \xi, x\neq y} \varphi(x-y).$$
If $\varphi$ is stable (see \eqref{eq:stable}) and $\int \xi(\Lambda_1)^2 \P(d\xi)<\infty$ the internal energy
$$\mathcal W^{int}(\P)=\lim_n \frac1n \int H_n(\xi) \P(d\xi)$$
exists encoding the interaction energy between the points of the point process $\xi$. Let $\Poi$ denote the law of a unit intensity homogeneous Poisson point process on $\IR$. Denote the restriction of $\xi$ to a set $B$ by $\xi_B$ and the  law of $\xi_B$ by $\P_B$. The specific entropy of $\P$ is defined as
$$\mathcal E(P)=\sup_{n\geq 1}\frac1n \ent(\P_{\Lambda_n}|\Poi_{\Lambda_n}),$$
where $\ent(\mu|\nu)$ is the relative entropy of $\mu$ w.r.t.\ $\nu$. Denoting the inverse temperature by $\beta$ we define the free energy of $\P$ by 
$$\mathcal F_\beta(\P)=\beta \mathcal W^{int}(\P)+ \mathcal E(\P).$$
The free energy is an important quantity in statistical mechanics carrying precious information about the system. For instance, the free energy appears as a rate function in large deviation principles in \cite{georgii2} (resp.\ \cite{LeSe17}) for Gibbs measures with superstable and regular interaction (resp. for log- and Riesz gases). As a consequence, Gibbs measures are minimizer of the free energy, an instance of the Gibbs variational principle.

In this article we are interested in analytical properties of the free energy w.r.t.\ a well chosen geometry that is able to pick up convexity properties of $\mathcal F_\beta$ that cannot be seen using the induced vector space  structure of the space of probability measures.
To this end, we assume that the even potential $\varphi$ is twice continuously differentiable on $(0,\infty)$, superstable and regular (see Section \ref{sec:Interaction_energy} for a precise definition). Moreover, put $h_n(x)=\mathsf{Leb}(\Lambda_n\cap(\Lambda_n-x))$ and $g_n=\varphi h_n.$ Assume that there is a continuous $f:(0,\infty)\to (0,\infty)$ such that 
$$g_n''(x)\geq n f(x) \text{ for all } n\in \Lambda_n\setminus \{0\}.$$
A guiding example satisfying these assumptions is the Riesz interaction potential $\varphi(x)=|x|^{-s}$ for $s>1$, see Example \ref{example:riesz}.
We call a continuous curve of probability measures $(\P_t)_{t\in[0,1]}$ an interpolation between $\P_0$ and $\P_1$.
Then, our first main result states
\begin{thm}\label{thm:intro1}
Let $\P_0,\P_1$ be the laws of two stationary point processes with unit intensity and $\mathcal F_\beta(\P_i)<\infty$ for $i=0,1.$ Then, there is an interpolation $(\P_t)_{t\in[0,1]}$ between $\P_0$ and $\P_1$ and a coupling $\Q^0$ between the Palm measures of $\P_0$ and $\P_1$ such that 
\begin{align}\label{eq:intro1}  
    \mathcal F_{\beta}(\P_{t})\leq  (1-t)\mathcal F_{\beta}(\P_{0})+t\mathcal F_{\beta}(\P_{1})-\frac{(1-t)t\beta}{2}\int \sum_{(0,0)\neq (x,y)\in \bar\xi} (x-y)^2\inf_{z\in [x,y]}f(\abs{z})\Q^0(d\bar\xi),
    \end{align}
\end{thm}
Note that unless $\P_0=\P_1$ the last term in \eqref{eq:intro1} is strictly positive such that we immediately obtain the uniqueness of minimizers of the free energy. Going back to the example of the Riesz interaction, the last theorem together with \cite{georgii2} imply a variational characterisation of the hypersingular Riesz gas, the Gibbsian point process with pairwise interaction induced by $\varphi(x)=|x|^{-s}$ for $s>1$. 

Importantly, the measure $\Q^0$ and the interpolation $(\P_t)_{t\in [0,1]}$ are tightly linked. In fact, the interpolation is induced by the coupling $\Q^0$ between the Palm measures $\P_0^0$ and $\P_1^0$ of $\P_0$ and $\P_1$. Let us explain this in some detail. We use a particular structure of one-dimensional simple point configurations $\xi$ with $0\in \xi$. Firstly, we can uniquely number the points of $\xi$ as $(\tau_i(\xi))_{i\in\mathbb Z}$ with $\tau_i(\xi)\leq \tau_{i+1}(\xi)$ for all $i$ and $\tau_0(\xi)=0$. Then, we can make a bijective change of coordinates going from an increasing sequence of points $(\tau_i(\xi))_i$ with $\tau_0(\xi)=0$ to the sequence of gaps $\s=(\s_i)_i$ of $\xi$ with $\s_i=\tau_i(\xi)-\tau_{i-1}(\xi)$. This allows us to switch between the Palm measure $\P^0$ of a stationary point process $\P$ and the distribution of its gaps, denoted by $\Pi=\j(\P)$ (see Theorem \ref{thm:correspondence} for details). For example the gap process of a Poisson point process is the infinite product of the distribution of exponential random variables. 

Lifting this correspondence between gaps and point configurations to the level of couplings, the coupling $\Q^0$ between $\P_0^0$ and $\P_1^0$ induces a coupling $\mathsf U$ between its gap distributions $\Pi_0$ and $\Pi_1$. 
For $(\s,\s')\in \supp \mathsf U$ and $t\in [0,1]$ we define the map $T_t:(\s,\s')\mapsto \s^t$ with $\s^t_i:=(1-t)\s_i + t\s'_i$ and the interpolating gap distribution $\Pi_t=(T_t)_\#\mathsf U$ as the pushforward of $\U$ by $T_t$. The point process  $\P_t$ corresponding to $\Pi_t$ is precisely the interpolating point process $\P_t$ from \eqref{eq:intro1}. In words, $\Q^0$ induces a coupling of gap distributions. If two sequences of gaps are coupled, we obtain an interpolation by linearly interpolating between the i-th gaps for each $i\in\IZ.$

Clearly, not every interpolation of this form will satisfy \eqref{eq:intro1}. We have to choose a particular $\Q^0$. To understand how to do that, let us consider the special case $\phi\equiv 0$, i.e.\ $\mathcal F_\beta=\mathcal E$ the specific entropy.

The specific entropy has a beautiful representation in the gap coordinates. Recall that the distribution of gaps of an intensity one Poisson point process is given by the product of rate one exponential distributions (denoted by $\gamma$). Then, we have
$$\mathcal E(\P)=\mathcal E^*(\Pi):=\sup_{n\geq 1}\frac1n\ent(\Pi^{1,n}|\gamma^{\otimes_{i=1}^n}),$$
where $\Pi=\j(\P)$ and  $\Pi^{1,n}$ denotes the pushforward of $\Pi$ under the map $\s\mapsto (\s_1,\ldots,\s_n)$. It is well known that the relative entropy $\ent$ is convex w.r.t.\ the displacement interpolation from optimal transport, i.e.\ it is convex along $\ell^p$-Wasserstein geodesics, see Section \ref{sec:definitions}. Following the strategy of \cite{erbar2023optimal}, we define a metric between stationary point processes as follows. Let $\Pi_0, \Pi_1$ be two distributions of gaps, i.e.\ probability measures on the space of non-negative bi-infinite sequences $\SQ=\{\s=(\s_i)_{i\in\IZ}, \s_i\geq 0\}$ which are stationary w.r.t.\ the natural shift. Recall that they correspond to laws of stationary point processes. Then, we put
\begin{align}\label{eq:gapmetric intro}
   \mathscr W_{gap,p}^p(\Pi_0,\Pi_1):=\inf_{\mathsf U\in \Cpl_s(\Pi_0,\Pi_1)}\int \abs{\s_1-\s'_1}^p\U(d\s,d\s'),
\end{align}
where $\Cpl_s(\Pi_0,\Pi_1)$ denotes the set of couplings of $\Pi_0$ and $\Pi_1$ which are stationary under diagonal shifts (cf.\ Definition \ref{def:cpls}). It is not hard to see that $\mathscr W_{gap,p}$ is indeed a geodesic metric on the space of stationary gap distributions inducing a corresponding metric on the space of stationary point processes, see Theorem \ref{thm:equivalent_cost}.
Similar to classical Wasserstein metrics, finiteness of $\mathscr W_{gap,p}$ is induced by finite $p$-th moments of the first gap, i.e.\ $\s_1$ under $\Pi_0$ and $\Pi_1$.

With this language we can now give another interpretation of the interpolation in equation \eqref{eq:intro1}. In fact, under the additional assumption that $\mathscr W_{gap,p}(\j(\P_0),\j(\P_1))<\infty$ the interpolation $(\P_t)_t$ is precisely given by a geodesic w.r.t.\ the gap metric $\mathscr W_{gap,p}$, i.e.\ the coupling $\Q^0$ is induced by an optimal $\U$ in the r.h.s.\ of \eqref{eq:gapmetric intro} via the change of coordinates between gaps and ordered points described above.
Hence, Theorem \ref{thm:intro1} says that on the space of stationary point processes equipped with the gap metric, the free energy is strictly geodesically convex with an explicit positive gain in convexity.

To  exploit this strict convexity it is desirable  to relate the gain 
$$\int \sum_{(0,0)\neq (x,y)\in \bar\xi} (x-y)^2\inf_{z\in [x,y]}f(\abs{z})\Q^0(d\bar\xi) $$
to the distance of $\P_0$ and $\P_1$ in the gap metric $\mathscr W_{gap,p}$. In fact, the contribution $(x-y)^2$ of the first positive pair $(x,y)$ precisely corresponds to the difference of the first gaps to the power of $2$. However, it is weighted by $\inf_{z\in [x,y]}f(\abs{z})$ which for instance in the case of Riesz interaction is not bounded from below. 
Hence, it seems to be difficult to relate the gain to $\mathscr W_{gap,2}(\Pi_0,\Pi_1)$. However, using Hölder's inequality, we can relate it to $\mathscr W_{gap,p}(\Pi_0,\Pi_1)$ for $1<p<2$. To obtain a uniform control we have to restrict the space of stationary probability measures we consider slightly further. For the case of the Riesz interaction we can for instance consider
$$X^a_p=\left\{\P \text{ stationary, intensity one} : \int\s_1^{\frac{p(2+s)}{2-p}} \j(\P)(d\s)\leq a \left(\frac{s(s+1)}{2}\right)^{p/(2-p)} \right\},$$
which gives a uniform control on certain moments of the first gap. Then we have the following result (for the general version see Proposition \ref{prop:lambda_conv}):
\begin{prop}\label{prop:intro}
    Let $1<p<2$ and $\varphi(x)=|x|^{-s}$. Then, $(X^a_p,\mathscr W_{gap,p})$ is a complete geodesic metric space. The free energy $\mathcal F_\beta$ is weakly $\lambda$-geodesically convex on $(X^a_p,\mathscr W_{gap,p})$ with $\lambda=\beta a^{-\frac{2-p}{p}}>0$, i.e.\ for $\P_0,\P_1\in X^a_p$ there is a constant speed geodesic $(\P_t)_{t\in[0,1]}$ such that for any $0\leq t \leq 1$
    \begin{align*}
        \mathcal F_\beta(\P_t)\leq (1-t) \mathcal F_\beta(\P_0)+t \mathcal F_\beta(\P_1)-\lambda \frac{t(1-t)}{2} \mathscr W_p^2(\P_0,\P_1).
    \end{align*}
\end{prop}

As a consequence of $\lambda$-convexity we can use the theory of gradient flows in metric spaces (cf.\ Sections \ref{sec:evihwi} and \ref{sec:lambda} for the necessary definitions) to obtain the following result stated for the special case of Riesz interaction (see Theorem \ref{thm:exist max slope} for the general version)

\begin{thm}\label{thm:intro2}
  Consider $\varphi(x)=|x|^{-s}$ for $s>1$ and the metric space $( X^{a}_{p},\mathscr W_{gap,p})$.
Every $\P\in X^{a}_{p}$ with  $\mathcal F_{\beta}(\P)<\infty$    is the starting point of a curve of maximal slope
for $\mathcal F_{\beta}$ with respect to the local slope $\abs{\partial \mathcal F_{\beta}}$, given by  \[
    \abs{\partial \mathcal F_{\beta}}(\P):=\limsup_{X^{a}_{p}\ni \P'\to\P}\frac{(\mathcal F_{\beta}(\P)-\mathcal F_{\beta}(\P'))^+}{\mathscr W_{gap,p}(\j(\P'),\j(\P))}.
    \] 
    Moreover, such a curve $(\P_t)_{t>0}$ satisfies the energy identity\begin{align}
    \frac{1}{2}\int_0^T\abs{\P'}^2(t)+\frac{1}{2}\int_0^T\abs{\partial \mathcal F_{\beta}}^2(\P(t))dt+\mathcal F_{\beta}(\P(T))=\mathcal F_{\beta}(\P(0)),\quad T>0.
\end{align}  
\end{thm}
Having this abstract result at our disposal we obtain the following corollary (see Section \ref{sec:lambda} for the general versions).

\begin{cor}
   Assume that $D(\mathcal F_{\beta})\cap X^{a}_{p}\neq \emptyset$ and  let $\P_\beta\in X^{a}_{p}$ be the unique minimizer of $\mathcal F_{\beta}$ on $X^{a}_{p}$. Put $\Pi_\beta:=\j(\P_\beta)$.
   Then for any $\P\in X^{a}_{p}\cap D(\mathcal F_{\beta})$
   \begin{align*}
       \frac{\beta }{2a^{\frac{2-p}{p}}} \mathscr W_{gap,p}^2(\j(\P),\Pi_\beta)\leq \mathcal F_{\beta}(\P)-\mathcal F_{\beta}(\P_\beta)\leq \frac{a^{\frac{2-p}{p}}}{2\beta } \abs{\partial \mathcal F_{\beta}}^2(\P).
    \end{align*}
    Moreover, any curve of maximal slope $(\P_t)_{t>0}$ w.r.t. $\abs{\partial \mathcal F_{\beta}}$ in $X^{a}_{p}$ satisfies for every $t\geq t_0>0$
    \begin{align*}
        \frac{\beta }{2a^{\frac{2-p}{p}}}\mathscr W_{gap,p}^2(\j(\P_t),\Pi_\beta)\leq \mathcal F_{\beta}(\P_t)-\mathcal F_{\beta}(\P_\beta)\leq (\mathcal F_{\beta}(\P_{t_0})-\mathcal F_{\beta}(\P_\beta))e^{-2\beta a^{\frac{p-2}{p}}(t-t_0)}.
    \end{align*}
Moreover, for $0<\beta<\beta'$ 
\begin{align*}
       \mathscr W_{gap,p}^2(\Pi_{\beta},\Pi_{\beta'})\leq \frac{2(\beta'-\beta)}{\beta}\max\left(\int f(\s_0)^{\frac{p}{p-2}}\Pi_\beta(d\s) ,\int f(\s_0)^{\frac{p}{p-2}}\Pi_{\beta'}(d\s) \right)^{\frac{2-p}{p}}\mathcal W^{int}(\P_{\beta})
    \end{align*}    
\end{cor}
The first displayed inequality can be interpreted as an instance of Talagrand and log-Sobolev type inequalities, e.g.\ by comparing to classical Wasserstein gradient flows in $\IR^d$ \cite[Section 4.4]{figalli2021invitation}. The second displayed inequality shows that there is a curve of laws of point processes that converges exponentially fast to the minimizer of the free energy by following a steepest descent route (it is a gradient flow curve). The last displayed equation shows a local $\frac{1}{2}$-Hölder continuity of the minimizer of $\mathcal F_\beta$ in the inverse temperature for $0<\beta<\infty.$

In view of our guiding example of the Riesz potential $\varphi(x)=|x|^{-s}$ it is natural to ask how much of our analysis rests on integrability assumptions on $\varphi$, i.e.\ the choice $s>1$. It turns out that for particular long-range interaction models, concretely Riesz potentials for $0<s<1$ and $\varphi(x)=-\log|x|$ there are strong approximation results  obtained in \cite{Le17}, that allow us to obtain the following result. Note that we have to adapt the definition of $\mathcal F_\beta$ to account for renormalisation due to lack of integrability. We denote the "corrected" version of the free energy by $\mathcal F_\beta^{elec}$ and refer to Section \ref{sec:longrange} for precise definitions.

\begin{thm}\label{thm:intro3}
Let $\varphi(x)=|x|^{-s}$ for $0<s<1$ or $\varphi(x)=-\log|x|$.    Let $\P_0,\P_1$ be the laws of two stationary point processes with unit intensity with $\mathcal F^{elec}_{\beta}(\P_i)<\infty$, $i=0,1$. Then there exists a coupling $\Q^0$ of $\P^0_0$ and $\P^0_1$ such that for $0\leq t\leq 1$ ($s=0$ corresponding to the $\log$ case)
    \begin{align*}
        \mathcal F^{elec}_{\beta}(\P_{t})&\leq  (1-t) \mathcal F^{elec}_{\beta}(\P_{0})+t\mathcal F^{elec}_{\beta}(\P_{1})\\
        &-\frac{(\eins_0(s)+s(s+1))(1-t)t\beta}{4}\int \sum_{(0,0)\neq (x,y)\in \bar\xi} (x-y)^2\inf_{z\in [x,y]}\abs{z}^{-s-2}\Q^0(d\bar\xi).
    \end{align*}
In particular, $\mathcal F^{elec}_{\beta}$ has a unique minimizer.
\end{thm}
This result recovers the main result of \cite{EHL21} and extends it to long-range Riesz interactions. Combining Theorem \ref{thm:intro3} with \cite{LeSe17} it follows that there is a unique stationary point processes minimizing $\mathcal F_\beta^{elec}$ implying a variational characterisation of Riesz gases in $d=1$. It would be interesting to understand whether this point process coincides with the circular Riesz gas constructed in \cite{Bo22}. 

While the interpolation $(\P_t)_{t\in[0,1]}$ and the coupling $\Q^0$ are related exactly as in Theorem \ref{thm:intro1}, unfortunately, so far we do not know, even assuming finite distance of $\P_0$ and $\P_1$, whether the coupling $\Q^0$ can be chosen to be optimal and, therefore, whether the interpolation is a geodesic curve or not. Without this knowledge we cannot discuss $\lambda$-geodesic convexity and, hence, neither gradient flows of $\mathcal F_\beta^{elec}$.
The main problem to overcome is the question of stability of optimal couplings in our setup. Up to now, we do not have anything close to classical stability results such as \cite[Theorem 5.20]{Villani}.

\subsection{Related literature}
Strict convexity of the free energy is closely related to the Gibbs variational principle in statistical mechanics. For lattice systems the link between Gibbs measures in the sense of DLR equations and minimizer of the free energy, a thermodynamic notion, has been established for various models under quite general assumptions, e.g.\ see \cite{Georgii, FrVe18} and references therein.
In the case of point processes we mention the seminal works \cite{georgii1, georgii2} and the more recent \cite{DeGe09, De16, JaKoStZa24}.  The case of long-range interactions has been treated in \cite{LeSe17} using renormalisation techniques. For an introduction to Gibbs point processes we refer to \cite{Der19, Ja18} and for general point processes to \cite{Daley, Last, bremaud}.
The question of uniqueness of minimizers of the rate functional in \cite{LeSe17}, i.e.\ $\mathcal F_\beta^{elec}$, was answered for the case of $\varphi(x)=-\log |x|$ in \cite{EHL21} using very careful approximation techniques. In fact this article was one key inspiration for the current work.
The second comes from the recent \cite{erbar2023optimal} where a Wasserstein type metric between stationary point processes on $\IR^d$ was constructed (see also \cite{huesmann2024benamoubrenier} for a dynamic formulation). The authors  identified the gradient flow of the specific entropy as infinitely many non-interacting Brownian motions. 
The main difference between $\mathscr W_{gap,p}$ and the metric constructed in \cite{erbar2023optimal} is that the former compares the size of gaps, in a certain sense relative positions of particles, whereas the latter compares the absolute positions of particles.

Interestingly, the metric from \cite{erbar2023optimal} is closely related to hyperuniform point processes, see \cite{DeFlHuLe25, HuLe25} and also \cite{LRY24,BuDaGaZe24}, as well as to the Coulomb energy \cite{HuLe25}. Non-local versions of Wasserstein metrics for point processes have been constructed in \cite{DSHeSu24,HuSt25} leading to birth-death processes as gradient flows of the (specific) entropy. So far this has not been extended to interacting systems.

More generally, stationary transports or couplings of point processes or random measures have received some attention in the literature due to their connection with shift couplings of Palm measures, e.g.\ \cite{HoPe05, LaTh09}. This leads to the question of constructing factor allocations or matchings, e.g.\ \cite{ChPePeRo10, HoJaWa22, HS13, HPPS09}. A finite version of this question is the classical optimal matching problem \cite{AKT84, CaLuPaSi14, AmStTr16, GoHu22, GoTr21, ClMa23}, which received some  renewed attention in the literature in the last years.

The theory of gradient flows in metric spaces and in particular its combination with the theory of optimal transport is a long and fascinating story. We refer to the books \cite{AGS08, Villani, villani2003topics, figalli2021invitation, Santa} for detailed references. 

There are several further approaches to construct and analyse dynamics associated to Gibbs measures. A powerful approach is via the theory of Dirichlet forms, e.g.\ \cite{AKR98, AKR98b, Os12, Os13, Su23}. Another via solving systems of infinitely many SDEs as e.g.\ in \cite{Ts16}.

\subsection{Overview of the article}
In Section \ref{sec:definitions}, we recall concepts and results from point processes and derive the key representation of point processes via their gap distribution. In Section \ref{sec:metric}, we construct the gap metric $\mathscr W_{gap,p}$ and its isometric sibling $\mathscr W_p$ living on the space of distributions of point processes. We show that these metrics are geodesic and that the specific entropy is geodesically convex. In Section \ref{sec:evihwi}, we investigate the free case of no interaction. We identify the gradient flow of the specific entropy w.r.t.\ $\mathscr W_p$ and show an HWI inequality. In Section \ref{sec:Interaction_energy} we show Theorem \ref{thm:intro1}. The $\lambda$-convexity and gradient flow results of Proposition \ref{prop:intro} and Theorem \ref{thm:intro2} are proven in Section \ref{sec:lambda}. Finally, Section \ref{sec:longrange} establishes Theorem \ref{thm:intro3}, the convexity result for the long-range interaction for the log-gas and Riesz-gas interaction in the parameter regime $0<s<1$.

\section{Preliminaries}\label{sec:definitions}
For a Polish space $X$ we denote its Borel sets by $\mathcal B(X)$ and set of probability measures over $X$ by $\mathcal P(X)$.
 For $\P_0,\P_1\in\mathcal P(X)$ we denote by $\Cpl(\P_0,\P_1)$ the set of all couplings between $\P_0$ and $\P_1$, i.e.\ the set of probability measures on $X\times X$ with marginals $\P_0$ and $\P_1$. For a map $T:X\to Y$ and $\mu\in\mathcal P(X)$ we denote its push forward or image measure by $T_\#\mu=\mu\circ T^{-1}\in\mathcal P(Y).$
 For $n\in\IN$ we put $\Lambda_n=[-n/2,n/2] \subset \IR$.
The configuration space $\Gamma_{\IR^d}$ is the space of all locally finite counting measures on $\IR^d.$ We equip $\Gamma_{\IR^d}$ with the topology of vague convergence, i.e.\ $\xi_n\to\xi$ iff $\int f\ d\xi_n\to \int f\ d\xi$ for all compactly supported and continuous functions $f\in C_c(\IR^d)$. This turns $\Gamma_{\IR^d}$ into a Polish space, see e.g.\ \cite[Theorem A2.3]{Kallenberg}.
As usual, we shall identify the configuration space $\Gamma_{\IR^d}$ with the set of all locally finite subsets of $\IR^d$. We write $\Gamma_{\IR}=\Gamma.$
 We label the points of  $\xi\in \Gamma$ in the following generic way 
 \begin{align}\label{eq:numbering}
    \xi=\{\ldots, \tau_{-1}(\xi),\tau_0(\xi),\tau_1(\xi),\tau_2(\xi),\dots\}
\end{align}
with $\tau_{0}(\xi)\leq 0$, $\tau_1(\xi)>0$ and $\tau_{i}(\xi)\leq \tau_{i+1}(\xi)$ for all $i\in \IZ$. We will be interested in configurations $\xi$ with $\xi(\{x\})\in\{0,1\}$ for all $x\in\IR^d$ such that there will be no ambiguity in the numbering.
We will denote by $(\theta_x)_{x\in \IR}$  the shift on the configuration space $\Gamma$, defined for $\xi\in \Gamma$ and $A \in \mathcal B(\IR)$  by  
\begin{align}\label{eq:shift_M}
    \theta_x\xi(A)=\xi(A+x).
\end{align}
A point process is a random variable $\xi$ with values in $\Gamma_{\IR^d}$. It is called simple if 
$$\IP(\xi(\{x\})\in\{0,1\},\; \forall x\in\IR^d)=1.$$ 
Its distribution $\P\in\mathcal P(\Gamma_{\IR^d})$ is said to be stationary if $(\theta_x)_\#\P=\P$ for all $x\in \IR^d$. A point process is stationary if its distribution is stationary. The intensity of a stationary point process $\xi$ resp.\ its distribution $\P$ is given by $\int \xi([0,1]^d)\P(d\xi).$ Throughout this article we will fix this intensity to be one. We will only work with simple point processes. For $B\in \mathcal B(\IR^d)$ we put $r_B:\Gamma_{\IR^d}\to [0,\infty]$, $\xi\mapsto\xi_B$ with $\xi_B(A)=\xi(A\cap B)$; for $\P\in\mathcal P(\Gamma_{\IR^d})$ we write $\P_B:=(r_B)_\#\P$.

To simplify several formulas we sometimes use the non standard convention to number tuples $x=(x_0,x_1)$.

\subsubsection*{Palm measures, gaps, and inversion.}
Let $\Q$ be a $\sigma$ finite stationary measure on $\Gamma_{\IR^d}$. The Palm measure $\Q^0$ of $\Q$ is defined by 
\begin{align}
        \Q^0(A):=\frac{1}{\Leb(B)}\int \int_{B} \eins_A(\theta_{x}\bar \xi)\bar\xi(dx)\Q(d\bar \xi),\quad \forall A\in \mathcal B(\Gamma_{\IR^d}),
    \end{align}
for any $B$ with positive Lebesgue measure. By stationarity of $\Q$ this definition is independent of the particular choice of $B$, e.g.\ \cite[equation (9.3)]{Last}. If $\Q$ is a probability measure with intensity one, then $\Q^0$ is a probability measure as well.
By a monotone class argument, the definition of the Palm measure implies the refined Campbell theorem \cite[Theorem 9.1]{Last}
\begin{thm}[Refined Campbell Theorem]\label{thm:refinedCampbell}
Let $\xi$ be a stationary point process on $\IR^d$ of intensity one with distribution $\P$. Then for all non-negative measurable functions $f$
$$\int\int f(x,\theta_x\xi) \xi(dx)\P(d\xi) = \int\int f(x,\xi) \P^0(d\xi)dx$$
or equivalently
$$\int\int f(x,\xi) \xi(dx)\P(d\xi) = \int\int f(x,\theta_{-x}\xi) \P^0(d\xi)dx$$
\end{thm}
We note that the Palm measure characterises the point processes in huge generality, see Remark \ref{rem:inversion}.
We will mostly be interested in stationary point processes on $\IR$ with intensity one which we denote by
\begin{align}\label{eq:Ps1}
    \mathcal P_{s,1}(\Gamma):=\{\P\in \mathcal P(\Gamma) \mid \forall x\in \IR: (\theta_x)_\#\P= \P  \text{ and }\int \xi(\Lambda_1)\P(d\xi)=1 \}.
    \end{align}

Let $\SQ =[0,\infty)^{\IZ}$ be the space of bi-infinite sequences with non-negative entries. We equip $\SQ$ with the infinite product topology. 
Define the shifts $\sigma^k$ on $\SQ$, $k\in \IZ$,  by $\sigma^k(\s_i)_{i\in \IZ}:=(\s_{i+k})_{i\in \IZ}$. To be able to switch between $\xi\in\Gamma$ and the sequence of its gaps we define the map 
 \begin{align}\label{eq:def_sq}
    \sq:\Gamma\to [0,\infty)^{\IZ}, \sq(\xi):=(\sq_i(\xi))_{i\in \IZ}:=(\tau_i(\xi)-\tau_{i-1}(\xi))_{i\in \IZ}.
\end{align}
Denote by $\Gamma_0\subset \Gamma$ the subset of all configurations $\xi$ with $0\in \xi$. We note that $\sq:\Gamma_0\to \SQ$ is a bijection and for $\xi\in \Gamma_0$ we have  
\begin{align}\label{eq:compatible_shifts}
    \sigma^k(\sq(\xi))=\sq(\theta_{\tau_{-k}(\xi)}\xi).
\end{align}
Moreover, since $\xi_n\to\xi$ in $\Gamma_0$ iff $\tau_i(\xi_n)\to\tau_i(\xi)$ for all $i\in\mathbb Z$ it follows that $\sq:\Gamma_0\to \SQ$ is a homeomorphism.

It will be beneficial to consider the gap distributions of stationary point processes. These gap distributions are stationary probability measures on $\SQ$. The following definition makes this precise.

\begin{defi}\label{def:statgapdistribution}
    We let $\mathcal P_{s,1}(\SQ)$ be the set of all stationary (w.r.t.\ the shifts $(\sigma^k)_{k\in \IZ}$) probability measures $\Pi$ on $\SQ$ with $\int \s_1\Pi(d\s)=1$.
\end{defi}

We will be interested in couplings of stationary gap distributions respecting the stationarity constraint.

\begin{defi}\label{def:cpls}
    For two stationary probability measures $\Pi_i\in \mathcal P_{s,1}(\SQ)$, $i=0,1$ we denote by $\Cpl_s(\Pi_0,\Pi_1)$ the set of all couplings of $\Pi_0$ and $\Pi_1$, which are invariant under the diagonal shifts 
    \begin{align*}
        \SQ^2 \ni (\s,\s')\mapsto (\sigma^ks,\sigma^ks')\in 
        \SQ^2,\quad k\in \IZ.
    \end{align*}
\end{defi}
We note that the set $\Cpl_s(\Pi_0,\Pi_1)$ is never empty since $\Pi_0\otimes \Pi_1\in \Cpl_s(\Pi_0,\Pi_1)$. We will show that couplings of gap distributions are closely related to monotone configurations in $\Gamma_{\IR^2}.$

\begin{defi}
    A configuration $\bar \xi\in\Gamma_{\IR^2}$ is monotone iff the following holds:   
    If $(x_1,y_1),(x_2,y_2)\in \bar \xi $ with $x_1\leq x_2$, then $y_1\leq y_2$ (we write $(x_1,y_1)\leq (x_2,y_2)$). We denote the subset of all monotone configurations by $\Gamma_{m}\subset \Gamma_{\IR^2}$.
\end{defi}
For $i=0,1$ define 
\begin{align}\label{eq:Proj}
     \PR_i:\Gamma_{\IR^2}\to \Gamma, \bar \xi\mapsto \{ x_i\mid (x_0,x_1)\in \bar\xi\}
 \end{align}

Finally, for two stationary point processes $\P_0$ and $\P_1$ we introduce a class of stationary and monotone couplings $\Cpl_{s,m}(\P_0,\P_1)$. 

\begin{defi}
    For  $\P_i\in \mathcal P_{s,1}(\Gamma)$, $i=0,1$ we denote by $\Cpl_{s,m}(\P_0,\P_1)$ the set of all $\sigma$-finite stationary  (with respect to the shifts $(\theta_x)_{x\in \IR^2}$) measures  $\Q$ on $\Gamma_{\IR^2}$ such that:
    \begin{itemize}
        \item[i)] The Palm measure $\Q^0$ of $\Q$  is a probability measure and $(\PR_i)_{\#}\Q^0=\P_i^0$ for $i=0,1$.
        \item[ii)]  The measure $\Q$ is concentrated on the set of monotone configurations $\bar\xi\in \Gamma_{\IR^2}$.
    \end{itemize}
\end{defi}
We note that the set $\Cpl_{s,m}(\P_0,\P_1)$ is never empty. This follows from the fact that the sets $\Cpl_s(\Pi_0,\Pi_1)$ are never empty, combined with Theorem \ref{thm:equivalent_cost} below.

For $i=0,1$ let $\pr_i:\IR^2\to \IR, (x_0,x_1)\mapsto x_i$ be the projection onto the $i$-th coordinate.
For a monotone  configuration $\bar \xi \in \Gamma_{\IR^2}$  we use an enumeration similar to \eqref{eq:numbering} 
\begin{align}\label{eq:bartau}
    \bar\xi=\{\ldots,\bar \tau_{-1}(\bar \xi),\bar \tau_0(\bar \xi),\bar \tau_1(\bar \xi),\dots\},
\end{align}
with $\pr_0(\bar \tau_0(\bar \xi))\leq 0<\pr_0(\bar \tau_1(\bar \xi))$ and $\bar \tau_{i}(\bar\xi)\leq \bar \tau_{i+1}(\bar \xi)$ (coordinate-wise).

We now collect some important results linking stationary point processes with their gap distributions.
The following theorem shows that there is a unique correspondence between  stationary point processes and stationary measures on $\SQ$, cf.\ \cite[Theorem 13.3.I]{Daley}.
\begin{thm}\label{thm:correspondence}
The map \begin{align}
    \j:\mathcal P_{s,1}(\Gamma)\to \mathcal P_{s,1}(\SQ), \P\mapsto 
    \sq_{\#}\P^0
\end{align}
is a bijection and for $\P\in \mathcal P_{s,1}(\Gamma)$ it holds 
 \begin{align}\label{eq:inversion_formula}
\int fd\P=\int \int_0^{\tau_1(\xi)}f(\theta_x\xi)dx\P^0(d\xi),
    \end{align}
    where $f:\Gamma\to \IR$ is bounded and measurable. We call $\j(\P)$ the gap distribution of $\P$ resp.\ of the point process $\xi$ with distribution $\P.$
\end{thm}

\begin{rem}\label{rem:inversion}
    Equation \eqref{eq:inversion_formula} is a particular instance of the inversion formula which implies that the point process is completely characterised by the Palm measure. This is true in broad generality and also beyond the setting of probability spaces. In fact, the setting of $\sigma$ finite measures is sometimes even more natural. We refer to \cite[Satz 2.4, Satz 2.5]{Mecke1967StationreZM} and \cite[Example 2.5]{LaTh09} for a general version of the inversion formula. This will be important in Theorem \ref{thm:equivalent_cost}.
\end{rem}

\textbf{Notational convention:} We will denote elements of $\mathcal P_{s,1}(\SQ)$ usually by $\Pi,\Pi_i$ and elements of $\mathcal P_{s,1}(\Gamma)$ by $\P,\P_i$.
\medskip\\
Theorem \ref{thm:correspondence} implies that the gap distributions of stationary point processes are stationary w.r.t.\ the shift $\sigma^k$, $k\in \IZ$, on $\SQ$.
Also, their Palm measures have a  stationary property, which is called \emph{point-stationarity}, see 
\cite[Theorem 7.3.1]{bremaud}.

\begin{thm}\label{thm:pointstationary}
    Let $\P\in \mathcal P_{s,1}(\Gamma)$. Then its Palm measure $\P^0$ is stationary w.r.t.\ the shifts $\theta_{\tau_i}$, $i\in \IZ$, defined  for $\xi\in \Gamma_0$ by 
    \begin{align*}
        \theta_{\tau_i}(\xi):=\theta_{\tau_i(\xi)}\xi,
    \end{align*}
    where $\xi=\{\cdots,\tau_{-1}(\xi),0=\tau_0(\xi),\tau_1(\xi),\cdots\}$ is the  enumeration of $\xi$ defined in \eqref{eq:numbering}.
\end{thm}

Let $\Poi\in \mathcal P_{s,1}(\Gamma)$ denote the distribution of a Poisson point process. The distribution of its corresponding gap process is well known:
\begin{ex}[{\cite[Theorem 7.2]{Last}}]
It holds $\j(\Poi)=\gamma^{\otimes\IZ}$, where $\gamma$ is the exponential distribution with parameter one.
\end{ex}

We will often switch between gaps and point processes. It will be important for us to transfer convergence statements from gap distributions to convergence statements about distributions of point processes and vice versa. Since the map $\sq:\Gamma_0\to \SQ$ (cf.\ \eqref{eq:def_sq}) is a bijection this is equivalent to understand the relation of convergence of Palm measures and the corresponding distributions of point processes. This is the content of the following lemma.

\begin{lem}\label{lem:conv-palm}
  \begin{enumerate}[a)]
      \item Let $\Pi_n\in \mathcal P_{s,1}(\SQ)$ with $\Pi_n\xrightarrow{n\to \infty}\Pi\in \mathcal P_{s,1}(\SQ)$ weakly, and let $(X_n)_{n\geq 1}$ be a sequence of random variables with $\mathsf{law}(X_n)=(\s_1)_{\#}\Pi_n$.  Assume the sequence $(X_n)_{n\geq 1}$ to be uniformly integrable.
    Let $\P_n:=\j^{-1}(\Pi_n)$ and $\P:=\j^{-1}(\Pi)$, 
    then also $\P_n\xrightarrow{n\to \infty}\P$ weakly.
   
    \item Conversely, let $\P_n\xrightarrow{n\to \infty}\P$ weakly in $\mathcal P_{s,1}(\Gamma)$ and  let $(X_n)_{n\geq 1}$ be a sequence of random variables with $\mathsf{law}(X_n)=(\xi(\Lambda_1))_{\#}\P_n$.  Assume the sequence $(X_n)_{n\geq 1}$ to be uniformly integrable. Let $\Pi_n:=\j(\P_n)$ and $\Pi:=\j(\P)$.
    Then $\Pi_n\xrightarrow{n\to \infty}\Pi$ weakly in $\mathcal P_{s,1}(\SQ)$.
      \end{enumerate}
\end{lem}
\begin{proof}
\textit{Item a)}    Let $\Pi_n\xrightarrow{n\to \infty}\Pi\in \mathcal P_{s,1}(\SQ)$ weakly and let the sequence $(X_n)_{n\geq 1}$ with $\mathsf{law}(X_n)=(\s_1)_{\#}\Pi_n$ be uniformly integrable. 
Fix $f\in C_c(\IR)$.
For $F\in C_b(\IR)$  continuous and bounded, we have to show that
\begin{align*}
    \int F\left(\int fd\xi\right)\P_n(d\xi)\xrightarrow{n\to \infty}\int F\left(\int fd\xi\right)\P(d\xi).
\end{align*}
From the inversion formula \eqref{eq:inversion_formula} it follows that 
\begin{align}\label{eq:inv_form}
\int F\left(\int fd\xi\right)\P_n(d\xi)=\int \int_0^{\tau_1(\xi)}F\left(\int fd\theta_z\xi\right)dz\P^0_n(d\xi).
\end{align}
Note that the weak convergence $\Pi_n\xrightarrow{n\to \infty}\Pi$ implies the weak convergence $\P^0_n\xrightarrow{n\to \infty}\P^0$ (when $\Gamma$ is equipped with the vague topology).

Moreover, the function $\Gamma_0\ni\xi\mapsto \int_0^{\tau_1(\xi)}F\left(\int fd\theta_z\xi\right)dz$ is continuous with respect to the topology of vague convergence on $\Gamma_0$.
To see this, let $\Gamma_0\ni\xi^n\xrightarrow{n\to \infty}\xi\in \Gamma_0$ vaguely. In particular, we have $\tau_0(\xi^n)=\tau_0(\xi)=0$ and the convergence $\tau_1(\xi^n)\xrightarrow{n\to \infty}\tau_1(\xi)$.
Hence, \begin{align*}
    &\abs{\int_0^{\tau_1(\xi^n)}F\left(\int fd\theta_z\xi^n\right)dz- \int_0^{\tau_1(\xi)}F\left(\int fd\theta_z\xi\right)dz}\\
    &\leq  \int_0^{\tau_1(\xi)}\abs{F\left(\int fd\theta_z\xi^n\right)-F\left(\int fd\theta_z\xi\right)}dz+\norm{F}_{\infty}\abs{\tau_1(\xi^n)-\tau_1(\xi)}. 
\end{align*}
The integral converges to zero by the dominated convergence Theorem and the fact that for fixed $z\in \IR$ we have 
$F\left(\int fd\theta_z\xi^n\right)\xrightarrow{n\to \infty}F\left(\int fd\theta_z\xi\right)$.
The second term also converges to zero. This proves the continuity.
Moreover, we have the following upper bound
\begin{align}\label{eq:bound}
    \abs{\int_0^{\tau_1(\xi)}F\left(\int fd\theta_z\xi\right)dz}\leq \norm{F}_{\infty} \tau_1(\xi).
\end{align}
Since  $\tau_1(\cdot)_{\#}\P^0_n=(\s_1)_{\#}\Pi^n$, the uniform integrability assumption  implies that a sequence of random variables $(Y_n)_{n\geq 1}$ with $\mathsf{law}(Y_n)=\tau_1(\cdot)_{\#}\P^0_n$ is uniformly integrable.

Combining   the continuity of the function  $\Gamma_0\ni\xi\mapsto \int_0^{\tau_1(\xi)}F\left(\int fd\theta_z\xi\right)dz$, the  weak convergence $\P^0_n\xrightarrow{n\to \infty}\P^0$, the bound \eqref{eq:bound} and the preceding uniform integrability statement implies
the convergence 
\begin{align*}
    \int \int_0^{\tau_1(\xi)}F\left(\int fd\theta_z\xi\right)dz\P^0_n(d\xi)\xrightarrow{n\to\infty}\int\int_0^{\tau_1(\xi)}F\left(\int fd\theta_z\xi\right)dz\P^0(d\xi).
\end{align*}
By \eqref{eq:inv_form} this proves the weak convergence $\P_n\to \P$.

\medskip

\textit{Item b)}  This part follows similarly. Choose $f$ and $F$ as above. The definition of the Palm measure yields 
 \[
\int F\left(\int fd\xi\right)\P^0_n(d\xi)=\int\sum_{x\in \xi\cap\Lambda_1}F\left(\int fd\theta_x\xi\right)\P_n(d\xi).
\]
Let $\Gamma^*:=\{\xi\in \Gamma \mid \xi(\partial \Lambda_1)=0\}$. Note that by stationarity we have $\P_n(\Gamma^*)=\P(\Gamma^*)=1$.
The function $\Gamma^*\ni\xi\mapsto \sum_{x\in \xi\cap \Lambda_1}F\left(\int fd\theta_x\xi\right)$ is continuous with respect to the topology of vague convergence.
To see this, let $\Gamma^*\ni\xi^n\xrightarrow{n\to \infty}\xi\in \Gamma^*$ vaguely. 
Let $\{\tau_j(\xi),\dots,\tau_l(\xi)\}=\xi\cap \Lambda_1$, $j\leq l$, where we use the enumeration introduced in \eqref{eq:numbering}.
For $n$ sufficiently large, we can assume that 
$\{\tau_j(\xi^n),\dots,\tau_l(\xi^n)\}=\xi^n\cap  \Lambda_1$. Moreover, we have the convergence $\tau_k(\xi^n)\xrightarrow{n\to \infty}\tau_k(\xi)$ for $j\leq k\leq l$. Hence, 
\begin{align*}
  \sum_{x\in \xi^n\cap \Lambda_1}F\left(\int fd\theta_x\xi^n\right)=
  \sum_{k=j}^lF\left(\int fd\theta_{\tau_k(\xi^n)}\xi^n\right)
  \xrightarrow{n\to \infty }\sum_{k=j}^lF\left(\int fd\theta_{\tau_k(\xi)}\xi\right),
\end{align*}
since for any $k=j,\dots,l$ we  have the convergence 
\[F\left(\int fd\theta_{\tau_k(\xi^n)}\xi^n\right)\xrightarrow{n\to\infty}F\left(\int fd\theta_{\tau_k(\xi^)}\xi\right).\]
This proves the continuity.
Furthermore, we have the following upper bound
\begin{align}\label{eq:bound2}
    \abs{\sum_{x\in \xi\cap\mathring \Lambda_1}F\left(\int fd\theta_x\xi\right)}\leq \norm{F}_{\infty}\xi(\Lambda_1).
\end{align}
Again, combining  the continuity of the map $\Gamma^*\ni\xi\mapsto \sum_{x\in \xi\cap \Lambda_1}F\left(\int fd\theta_x\xi\right)$, the  weak convergence $\P_n\xrightarrow{n\to \infty}\P$, the bound \eqref{eq:bound2} and the  uniform integrability assumption from the second part of this lemma yield
the convergence
\begin{align*}
    \int \sum_{x\in \xi\cap \Lambda_1}F\left(\int fd\theta_x\xi\right)\P_n(d\xi)\xrightarrow{n\to\infty}\int\sum_{x\in \xi\cap \Lambda_1}F\left(\int fd\theta_x\xi\right)\P(d\xi).
\end{align*}
Hence, 
 the weak convergence $\P_n^0\xrightarrow{n\to \infty}\P^0$ holds which yields the  weak convergence $\Pi_n\xrightarrow{n\to \infty}\Pi$.
\end{proof}

\subsubsection*{Specific entropy.} We recall the definition of the specific relative entropy of a stationary point processes. For $\P\in \mathcal P_{s,1}(\Gamma)$, the specific relative entropy $ \mathcal E( \P)$ with respect to the Poisson point process, see \cite{Serfaty,RAS,Georgii}, is defined as
\begin{align}\label{eq:specific_rel_ent}
    \mathcal E(\P):=\lim_{n\to\infty} \frac 1{n^d}\mathsf{Ent}(\P_{\Lambda_n}|\Poi_{\Lambda_n}):=\lim_{n\to\infty} \frac 1{n^d} \int \log\Big(\frac{d\P_{\Lambda_n}}{d\Poi_{\Lambda_n}}(\xi)\Big)d\mathsf{P}_{\Lambda_n}(\xi).
\end{align}

 It has the following basic properties, see \cite{Serfaty,RAS,Georgii}.
 
\begin{lem}\label{lem:Georgii} Let $\P\in \mathcal P_{s,1}(\Gamma)$. Then
\begin{enumerate}[a)]
    \item The limit  in \eqref{eq:specific_rel_ent} exists in $[0,\infty]$ and is equal to the supremum, i.e. \begin{align}\label{eq:def_entrop_sup}
        \mathcal E( \P)=\sup_{n\in \IN} \frac 1{n^d}\mathsf{Ent}(\P_{\Lambda_n}|\Poi_{\Lambda_n})
    \end{align}
    \item The map $\P\mapsto\mathcal E(\P)$ is affine and lower semi-continuous.
    \item The specific relative entropy vanishes iff $\P=\Poi$.

\end{enumerate}
\end{lem}

There exists a similar notion for measures $\Pi\in \mathcal P_{s,1}(\SQ)$. The reference measure is given by the gap distribution of the Poisson point process, which  is $\gamma^{\IZ}$.

\begin{defi}
For $\Pi\in \mathcal P_{s,1}(\SQ)$ and $m<n$    denote by $\Pi^{m,n}$ the pushforward of $\Pi$ under the map $
        (\s_i)_{i\in \IZ}\mapsto (\s_i)_{i=m,\dots,n}$. We set  \begin{align*}
    \mathcal E^*(\Pi):=\lim_{n\to \infty}n^{-1}\ent(\Pi^{1,n}\mid \gamma^{\otimes_{i=1}^n})=\sup_{n\in \IN}n^{-1}\ent(\Pi^{1,n}\mid \gamma^{\otimes_{i=1}^n}).
\end{align*}
\end{defi}
The existence of the limit follows by a standard subadditivity argument (cf. \cite[Theorem 6.7]{RAS}).
These two notions of entropy are linked by the bijection $\j$, i.e.\
the specific relative entropy of a point process and the entropy of its gap distribution coincide, see \cite[Proposition 3.8]{georgii1}.

\begin{thm}\label{thm:equiv_entropies}
    For a stationary point process $\P\in \mathcal P_{s,1}(\Gamma)$ \begin{align*}
        \mathcal{E}(\P)=\mathcal{E}^*(\j(\P)).
    \end{align*}
\end{thm}

We finish this section by recalling the concept of hyperuniformity. This definition will become relevant in Section \ref{sec:longrange}, where we deal with the one-dimensional log-gas.
\begin{defi}\label{def:hyperunif}
    A point process $\P\in \mathcal P_{s,1}(\Gamma)$ is  hyperuniform if \begin{align*}
       \frac{1}{n} \int (\xi(\Lambda_n)-n)^2\P(d\xi)\xrightarrow{n\to \infty}0.
    \end{align*}
    The process $\P$ is class-1-hyperuniform if  $ \int (\xi(\Lambda_n)-n)^2\P(d\xi)=O(1)$.
\end{defi}

\section{Metric structure and convexity of specific entropy}\label{sec:metric}
In this section, we construct two isometric geodesic extended metrics, one between distributions of point processes and one between the corresponding gap distributions. We will show that the specific entropy is convex along geodesics.
We will make extensive use of the bijection $\j$, which allows us to conduct calculations and estimations in the  space of sequences $\SQ$ which are usually much easier.

We start by defining the extended gap metric $\mathscr W_{gap,p}$. Recall $\Cpl_s(\Pi_0,\Pi_1)$, the set of stationary couplings between gap distributions from Definition \ref{def:cpls}.

\begin{defi}
For $p\geq 1$  and  $\Pi_i\in \mathcal P_{s,1}(\SQ)$, $i=0,1$, 
define 
\begin{align}\label{eq:w gap}
   \mathscr W_{gap,p}^p(\Pi_0,\Pi_1):=\inf_{\mathsf U\in \Cpl_s(\Pi_0,\Pi_1)}\int \abs{\s_1-\s'_1}^p\U(d\s,d\s').
\end{align}
\end{defi}
It follows from the proof of  Theorem \ref{thm: ext metric gap} that $\mathscr W_{gap,p}$ is indeed a metric that might attain the value $\infty$, in this sense, it is an extended metric.
The existence of an $\mathscr W_{gap,p}-$optimal coupling, i.e., a coupling attaining the  infimum in \eqref{eq:w gap}, follows from a standard argument.

\begin{lem}
    For $\Pi_i\in \mathcal P_{s,1}(\SQ)$, $i=0,1$, there exists a $\mathscr W_{gap,p}-$optimal coupling $\U\in \Cpl_s(\Pi_0,\Pi_1)$.
\end{lem}
\begin{proof} If the cost is infinite, any coupling is optimal, hence we can assume that 
\[\mathscr W_{gap,p}(\Pi_0,\Pi_1)<\infty.\]
    Let $(\U_n)_{n\in \IN}$ be a minimizing sequence in $\Cpl_s(\Pi_0,\Pi_1)$. 
    It follows from classical theory that $\Cpl(\Pi_0,\Pi_1)$ is compact (cf. \cite[Proof of Theorem 4.1]{Villani}) with respect to the weak topology. Hence, there exists a subsequence, still denoted by $(\U_n)_n$ converging weakly to some $\U\in \Cpl(\Pi_0,\Pi_1)$. 
    Let $F:\SQ^2\to \IR$ be continuous and bounded, and let $k\in \IZ$. Then the function $\tilde F=F\circ (\sigma^k,\sigma^k)$ is also continuous and bounded. Hence,
    \begin{align*}
        \int Fd\U=\lim_{n\to \infty}\int Fd\U_n=\lim_{n\to \infty}\int \tilde Fd\U_n= \int \tilde Fd\U,
    \end{align*}
    which proves stationarity of $\U$ (and compactness of $\Cpl_s(\Pi_0,\Pi_1)$).
    Hence, it remains to show that $\U$ is optimal.  By continuity and positivity of the function $\SQ^2\ni (\s,\s')\mapsto \abs{\s_1-\s'_1}^p$ (with respect to the product topology) and the Portmanteau theorem
    \begin{align*}
    \int \abs{\s_1-\s'_1}^p\U(d\s,d\s')\leq \liminf_{n\to \infty} \int \abs{\s_1-\s'_1}^p\U_n(d\s,d\s')=\mathscr W^p_{gap,p}(\Pi_0,\Pi_1). 
    \end{align*}
\end{proof}

It will be useful to approximate the cost $\mathscr W^p_{gap,p}$ with finite dimensional classical transport costs. To this end, we introduce the  cost functionals  $\mathcal C_p$.  For two probability measures $\mu,\nu$ on $\IR^d$ we put 
\begin{align}\label{eq:Cp}
    \mathcal C_p(\mu,\nu):=\inf_{\q\in \Cpl(\mu,\nu)}\int \norm{x-y}_p^p\q(dx,dy),
\end{align}
where  $\norm{x}_p^p=\abs{x_1}^p+\dots+\abs{x_d}^p$ and $\Cpl(\mu,\nu)$ denotes the set of all couplings between $\mu$ and $\nu$.

\begin{lem}\label{lem:cost=lim_wasserstein}
    Let $\Pi_i\in \mathcal P_{s,1}(\SQ)$, $i=0,1$. Then 
    \begin{align}
        \mathscr W^p_{gap,p}(\Pi_0,\Pi_1)=\lim_{n\to \infty} n^{-1} \mathcal C_p(\Pi_0^{1,n},\Pi^{1,n}_1)=\sup_{n\in \IN} n^{-1} \mathcal C_p(\Pi_0^{1,n},\Pi^{1,n}_1).
    \end{align}
\end{lem}
\begin{proof}

\textit{Step 1}: We show $\lim_{n\to \infty} n^{-1} \mathcal C_p(\Pi_0^{1,n},\Pi^{1,n}_1)=\sup_{n\in \IN} n^{-1} \mathcal C_p(\Pi_0^{1,n},\Pi^{1,n}_1).$

Let $m,n\in \IN$ and let $\U_{m+n}\in \Cpl(\Pi_0^{1,m+n},\Pi_1^{1,m+n})$ be a $\mathcal C_p-$optimal coupling. Then 
\begin{align*}
    \mathcal C_p(\Pi_0^{1,m+n},\Pi_1^{1,m+n})&=\int \sum_{i=1}^m\abs{\s_i-\s'_i}^p\U_{m+n}(d\s,d\s')+\int \sum_{i=m+1}^{m+n}\abs{\s_i-\s'_i}^p\U_{m+n}(d\s,d\s')\\
    &\geq  \mathcal C_p(\Pi_0^{1,m},\Pi_1^{1,m})+ \mathcal C_p(\Pi_0^{m+1,m+n},\Pi_1^{m+1,m+n})\\
     &=  \mathcal C_p(\Pi_0^{1,m},\Pi_1^{1,m})+ \mathcal C_p(\Pi_0^{1,n},\Pi_1^{1,n}),
\end{align*}
where the last equality follows from stationarity. From Feketes Lemma it then follows that \begin{align*}
    \lim_{n\to \infty} n^{-1} \mathcal C_p(\Pi_0^{1,n},\Pi^{1,n}_1)=\sup_{n\in \IN} n^{-1} \mathcal C_p(\Pi_0^{1,n},\Pi^{1,n}_1).
\end{align*}

\medskip
    
\textit{Step 2}: We show  $\sup_{n\in \IN}n^{-1} \mathcal C_p(\Pi_0^{1,n},\Pi^{1,n}_1)\leq \mathscr W^p_{gap,p}(\Pi_0,\Pi_1)$.

Let $\U\in \Cpl_s(\Pi_0,\Pi_1)$ be $\mathscr W_{gap,p}-$optimal and let $\U^{1,n}\in \Cpl(\Pi_0^{1,n},\Pi^{1,n}_1)$ be the pushforward of $\U$ under the map \begin{align*}
        \SQ^2\ni (\s,\s')\mapsto ((\s_i,\s_i')_{i=1,\dots,n})\in \IR^2.
    \end{align*}
    Then by stationarity 
    \begin{align*}
        n^{-1} \mathcal C_p(\Pi_0^{1,n},\Pi^{1,n}_1)\leq n^{-1}\int\sum_{i=1}^n\abs{\s_i-\s'_i}^p\U(d\s,d\s')=\mathscr W^p_{gap,p}(\Pi_0,\Pi_1),
    \end{align*}
    which shows that $\sup_{n\in \IN}n^{-1} \mathcal C_p(\Pi_0^{1,n},\Pi^{1,n}_1)\leq \mathscr W^p_{gap,p}(\Pi_0,\Pi_1)$.

\medskip

\textit{Step 3}:  We show $\liminf_{n\to \infty} n^{-1} \mathcal C_p(\Pi_0^{1,n},\Pi^{1,n}_1)\geq \mathscr W^p_{gap,p}(\Pi_0,\Pi_1)$.

Let $n\in \IN$ and let  $\U^n$ be a minimizer for 
\[
\inf_{\U\in \Cpl(\Pi^{1,n}_0,\Pi^{1,n}_1)}\int \norm{x-y}_p^pd\U(x,y).
\]
For $i=0,1$ let $\Pi_{i,n}^{\otimes}:=\paste^n_{\#}\left(\bigotimes_{m\in \IZ}\Pi_{i}^{1,n}\right)$, where 
    $\paste^n:([0,\infty)^n)^{\IZ}\to \SQ$ is the
map defined by $\paste^n((\s^m_1,\cdots,\s^m_n)_{m\in \IZ})_{kn+l}:=\s^k_l$, for $l=1,\dots,n$ and $k\in \IZ$.
Set $\bar{\Pi}_{i,n}:=n^{-1}\sum_{j=1}^n\sigma^{j}_{\#}\Pi_{i,n}^{\otimes}$, $i=0,1$.
Similarly, define 
$\U^{\otimes}_n$ as the pushforward of $\bigotimes_{m\in \IZ}\U_n$ under the map $\bar\paste^n$, where $\bar\paste^n:([0,\infty)^n\times [0,\infty)^n)^{\IZ}\to \SQ^2$ is defined by  \begin{align*}&\bar\paste^n(((\s^m_1,\cdots,\s^m_n),(\s'^m_1,\cdots,\s'^m_n))_{m\in \IZ})\\
&\hspace{50mm}:= (\paste^n((\s^m_1,\cdots,\s^m_n)_{m\in \IZ}),\paste^n((\s'^m_1,\cdots,\s'^m_n)_{m\in \IZ})).\end{align*}
Finally, for $j\in \IZ$ let $\bar\sigma^j:\SQ^2\to \SQ^2$ be the diagonal shift defined by $\bar\sigma^j(\s,\s')=(\sigma^j\s,\sigma^j\s')$ and set $\bar{\U}_n=n^{-1}\sum_{j=1}^n\bar\sigma^j_{\#}\U^{\otimes}_n$.
Then $\bar{\U}_n\in \Cpl_s(\bar{\Pi}_{0,n},\bar{\Pi}_{1,n})$ and we claim:
\begin{equation}\label{eq:claim3.1}
\bar{\Pi}_{i,n}\xrightarrow{n\to \infty}\Pi_i \text{ weakly.}
\end{equation}
Admitting the claim, by considering a subsequence, we can assume that there exists $\bar \U\in \mathcal P_{s}(\SQ^2)$ such that $\bar{\U}^n\xrightarrow{n\to \infty}\bar \U\in \Cpl_s(\Pi_0,\Pi_1)$ weakly.
Fatou's Lemma then yields 
\begin{align*}
     \mathscr W^p_{gap,p}(\Pi_0,\Pi_1)&\leq\int \abs{\s_1-\s'_1}^p\bar \U(d\s,d\s') \leq \liminf_{n\to \infty} \int \abs{\s_1-\s'_1}^p\bar \U^n(d\s,d\s')\\
     &= \liminf_{n\to \infty}n^{-1} \sum_{j=1}^n\abs{\s_j-\s'_j}^p\U^n(d\s,d\s')
     =\liminf_{n\to \infty}n^{-1}\mathcal C_p(\Pi_0^{1,n},\Pi_1^{1,n}),
\end{align*}
which proves the statement of this Lemma.

We are left to prove the claim \eqref{eq:claim3.1}: Let $f:\SQ\to \IR$ be continuous (with respect to the infinite product topology) and bounded. By \cite[Theorem 2.1]{Edg}, we can assume that $f$ depends only on finitely many coordinates, that is there exist $k<l$, $k,l\in \IZ$, and a continuous and bounded function $g:\IR^{l-k}\to \IR$ such that $f(\s)=g(\s_k,\dots,\s_l)$. 
Hence, for $n\in \IN$ sufficiently large 
\begin{align}\label{eq:paste}
    &\int fd\bar{\Pi}_{i,n}=n^{-1}\sum_{j=1}^n\int f(\sigma^j\s)\Pi_{i,n}^{\otimes}(d\s)
    =n^{-1}\sum_{j=1}^n\int g(\s_{k+j},\dots,\s_{l+j})\Pi_{i,n}^{\otimes}(d\s)\\
    &=n^{-1}\sum_{j=\max(1,1-k)}^{\min(n,n-l)}\int g(\s_{k+j},\dots,\s_{l+j})\Pi_{i,n}^{\otimes}(d\s)+n^{-1}\sum_{j=1}^{\max(1,1-k)-1} \int g(\s_{k+j},\dots,\s_{l+j})\Pi_{i,n}^{\otimes}(d\s)\nonumber\\
    &+n^{-1}\sum_{j=\min(n,n-l)+1}^{n}\int g(\s_{k+j},\dots,\s_{l+j})\Pi_{i,n}^{\otimes}(d\s)\nonumber
\end{align}
For the first term we get by stationarity
$$n^{-1}\sum_{j=\max(1,1-k)}^{\min(n,n-l)}\int g(\s_{k+j},\dots,\s_{l+j})\Pi_{i,n}^{\otimes}(d\s)=n^{-1}\sum_{j=\max(1,1-k)}^{\min(n,n-l)}\int g(\s_{k},\dots,\s_{l})\Pi_{i}(d\s).$$
From the boundedness of $g$ it follows that  
\begin{align*}
    &\lim_{n\to \infty} n^{-1}\sum_{j=1}^{\max(1,1-k)-1} \int g(\s_{k+j},\dots,\s_{l+j})\bar{\Pi}_{i,n}(d\s)\\
    &=\lim_{n\to \infty}  n^{-1}\sum_{j=\min(n,n-l)+1}^{n}\int g(\s_{k+j},\dots,\s_{l+j})\bar{\Pi}_{i,n}(d\s)=0
\end{align*}
and hence we have \begin{align*}
 \lim_{n\to \infty}  \int f d\bar{\Pi}_{i,n}=\int f d\Pi_i,
\end{align*}
which proves the claim $\bar{\Pi}_{i,n}\xrightarrow{n\to \infty}\Pi_i$ weakly.
\end{proof}


\begin{defi}\label{def:metr_pointpr}
For a monotone configuration $\bar\xi\in\Gamma_m$ recall the enumeration $(\bar\tau_\ell)_{\ell\in\IZ}$ from \eqref{eq:bartau}. For $(x,y)\in \bar\xi$ with $\bar \tau_\ell(\bar\xi)=(x,y)$ some $\ell\in\IZ$ we denote by $(x^+,y^+)\in \bar\xi$ the point $\bar\tau_{\ell+1}(\bar\xi).$
    For $p \geq 1$ and two stationary point processes $\P_i\in \mathcal P_{s,1}(\Gamma)$, $i=0,1$, we let 
    \begin{align}
        \mathscr W_p^p(\P_0,\P_1):=\inf_{\Q\in \Cpl_{s,m}(\P_0,\P_1)}\int\sum_{(x,y)\in \bar\xi\cap \Lambda_1^2}\abs{(x^+-x)-(y^+-y)}^p\Q(d\bar\xi). 
    \end{align}
\end{defi}

The next result shows that the map $\j$ from Theorem \ref{thm:correspondence} is an isometry between the spaces $(\mathcal P_{s,1}(\Gamma),\mathscr W_p)$ and $(\mathcal P_{s,1}(\SQ),\mathscr W_{gap,p}).$ In fact the metric $\mathscr W_p$ is constructed such that this isometry holds. The proof clarifies and relies on the relation between the different notions of stationarity introduced so far in Definition \ref{def:statgapdistribution}, equation \eqref{eq:Ps1} and Theorem \ref{thm:pointstationary}.

\begin{thm}\label{thm:equivalent_cost}
  Let $\P_i\in \mathcal P_{s,1}(\Gamma)$, $i=0,1$. The map \begin{align*}
    \iota: \Cpl_{s,m}(\P_0,\P_1)\to \Cpl_s(\j(\P_0),\j(\P_1)), \Q\mapsto (\sq\circ\PR_0,\sq\circ\PR_1)_{\#}\Q^0
\end{align*}
is a bijection. Furthermore, it holds that 
\begin{align}\label{eq:iota-isom}
    \int \sum_{(x,y)\in \bar\xi\cap \Lambda_1}\abs{(x^+-x)-(y^+-y)}^p\Q(d\bar\xi)=\int \abs{\s_1-\s'_1}^p\iota(\Q)(d\s,d\s').
\end{align}    
In particular
  \begin{align*}
        \mathscr W_p(\P_0,\P_1)=\mathscr W_{gap,p}(\j(\P_0),\j(\P_1)) 
    \end{align*}
    and there exists an $\mathscr W_p-$optimal $\Q\in \Cpl_{s,m}(\P_0,\P_1)$.
\end{thm}
\begin{proof}
    Let $\Q\in \Cpl_{s,m}(\P_0,\P_1)$ be given. 
By definition of $\Cpl_{s,m}(\P_0,\P_1)$    the Palm measure $\Q^0$   of $\Q$ is a probability measure with marginals $(\PR_i)_{\#}\Q^0=\P_i^0$ for $i=0,1$ (cf. \eqref{eq:Proj} for the definition of $\PR_i$).

\medskip

\textit{Step 1}:    The Palm measure $\Q^0$ is stationary w.r.t. the shifts $\theta_{\x_i}$, $i\in \IZ$ (cf.\ \eqref{eq:bartau} and Theorem \ref{thm:pointstationary}).

    Although the proof of the stationarity is the same as in the one dimensional case, see \cite[Theorem 7.3.1]{bremaud}, we repeat it for the convenience of the reader. Fix a measurable set $A\subset \Gamma_m$. Note that for a monotone configuration $\bar \xi\in \Gamma_{\IR^2}$ 
    \begin{align*}
        \eins_{\theta_{\x_1}^{-1}(A)}(\theta_{\x_i}(\bar \xi))=\eins_A(\theta_{\x_{i+1}}(\bar \xi)),
    \end{align*} where
    \begin{align*}
        \theta_{\x_1}^{-1}(A)=\{\theta_{\x_1}^{-1}(\bar \xi):\bar \xi\in A\}.
    \end{align*}
    Hence \begin{align*}
        \abs{ \Q^0(A)-\Q^0(\theta_{\x_1}^{-1}(A))   }&\leq n^{-2}
        \int \abs{\sum_{i\in \IZ} (\eins_A(\theta_{\x_i}\bar\xi)-\eins_A(\theta_{\x_{i+1}}\bar \xi)) \eins_{\Lambda_n^2}(\x_i(\bar \xi)) }         \Q(d\bar\xi)\\
        &\leq 2n^{-2}\xrightarrow{n\to \infty}0,
    \end{align*}
    where the second inequality follows from the monotonicity of $\bar \xi\in \supp{\Q}$.
    This shows the stationarity of $\Q^0$ w.r.t.\ the shifts $\theta_{\x_i}$, $i\in \IZ$.
    
    Put $\U:=(\sq\circ\PR_0,\sq\circ\PR_1)_{\#}\Q^0=\iota(\Q)$. Since $(\PR_0)_{\#}\Q^0=\P^0_0$ and $(\PR_1)_{\#}\Q^0=\P^0_1$, we have $\U\in \Cpl(\sq_{\#}\P_0^0,\sq_{\#}\P_1^0)$. 

\medskip
    
   \textit{Step 2}: The point stationarity of $\Q^0$ implies the stationarity of $\U$. 
    
    Let $f:\SQ^2\to \IR$ be bounded and measurable. Then by \eqref{eq:compatible_shifts}
    \begin{align*}
        \int f(\sigma^k\s,\sigma^k\s')\U(d\s,d\s')&=\int f(\sigma^k(\sq(\PR_0(\bar \xi))),\sigma^k(\sq(\PR_1(\bar \xi))))d\Q^0(\bar \xi)\\
        &=\int f(\sq(\theta_{\tau_{-k}(\PR_0(\bar \xi))}\PR_0(\bar \xi)),\sq(\theta_{\tau_{-k}(\PR_1(\bar \xi))}\PR_1(\bar \xi)))d\Q^0(\bar \xi)\\
        &=\int f(\sq(\PR_0(\theta_{\bar \tau_{-k}}\bar \xi)),\sq(\PR_1(\theta_{\bar \tau_{-k}}\bar \xi)))d\Q^0(\bar \xi),
    \end{align*}
    where the last two equalities follow from the fact that for $\bar \xi\in \Gamma_{m}$ with $(0,0)\in \bar \xi$ we have that $\pr_i(\x_j(\bar\xi))=\tau_j(\PR_i(\bar\xi))$, $i=0,1$ and $j\in \IZ$. The stationarity of $\Q^0$ thus implies 
    \begin{align*}
       \int f(\sigma^k\s,\sigma^k\s')\U(d\s,d\s')
        &=\int f(\sq(\PR_0(\theta_{\bar\tau_{-k}}\bar\xi)),\sq(\PR_1(\theta_{\bar\tau_{-k}}\bar\xi)))d\Q^0(\bar\xi)\\
        &=
        \int f(\sq(\PR_0(\bar\xi)),\sq(\PR_1(\bar\xi)))d\Q^0(\bar\xi)\\
        &=\int f(\s,\s')\U(d\s,d\s').
    \end{align*}
Hence, $\U\in \Cpl_s(\Pi_0,\Pi_1)$.

\medskip

\textit{Step 3}: Equation \eqref{eq:iota-isom} holds.

Since $(0,0)\in \bar\xi $\ $\Q^0-$a.s.\ we obtain
 \begin{align}\label{eq:isom}
     \int\abs{\s_1-\s'_1}^p\U(d\s,d\s')
    &=\int \abs{\pr_0(\x_1(\bar\xi))-\pr_1(\x_1(\bar\xi))}^p\Q^0(d\bar\xi)\\
    &=\int \sum_{(x,y)\in \Lambda_1 \cap \bar\xi} \abs{(x^+-x)-(y^+-y)}^p\Q(d\bar\xi)\nonumber.
\end{align}
which is \eqref{eq:iota-isom}.

\medskip

\textit{Step 4}: $\iota$ is bijective.

Let $\U\in \Cpl_{s}(\sq_{\#}\P^0_0,\sq_{\#}\P^0_1)$ be given. Define $\Q^*$ to be the pushforward of the measure $\U$ under the map \begin{align}\label{eq:bij}
    \SQ^2 \ni(\s,\s')\mapsto 
    \{(\tau_i(\sq^{-1}(\s)),\tau_i(\sq^{-1}(\s'))):i\in \IZ\}
  \in \Gamma_{\IR^2},
\end{align}
which is well-defined, since $\sq:\Gamma_0\to \SQ$ is a bijection. 
We note that by definition $\Q^*$ is concentrated on the subset of monotone configurations, i.e.\ $\Q^*(\Gamma_m)=1$. Furthermore, $(\PR_i)_{\#}\Q^*=\P_i^0$, $i=0,1$.
Let $j\in \IZ$, $(\s,\s')\in \SQ^2$ and set $\bar\xi=\{(\tau_i(\sq^{-1}(\s)),\tau_i(\sq^{-1}(\s'))):i\in \IZ\}$. An easy calculation shows that for $\bar \xi\in \Gamma_m$\begin{align*}
    \theta_{\x_j(\bar\xi)}\bar\xi&=\{(\tau_i(\sq^{-1}(\s))-\tau_j(\sq^{-1}(\s)),\tau_i(\sq^{-1}(\s')))-\tau_j(\sq^{-1}(\s')):i\in \IZ\}\\
    &=
    \{(\tau_i(\sq^{-1}(\sigma^{-j}\s)),\tau_i(\sq^{-1}(\sigma^{-j}\s'))):i\in \IZ\}.
\end{align*}
Hence, for a bounded measurable function $f:\Gamma_{\IR^2}\to \IR$ the stationarity of the measure $\U$ implies\begin{align}\label{eq:event_time_stat_Q^*}
    \int f(\theta_{\x_j}\bar\xi)\Q^*(d\bar\xi)=\int f(\bar \xi)\Q^*(d\bar\xi).
\end{align}

By \cite[Satz 2.4, Satz 2.5]{Mecke1967StationreZM} 
there exists a unique $\sigma$-finite stationary measure $\Q$ on $\Gamma_{\IR^2}$ with $\Q^0=\Q^*$ iff for all measurable and bounded $g:\Gamma_{\IR^2}\times \IR^2\to \IR$ 
\begin{align*}
    \int \int g(\theta_{x}\bar\xi,-x)\bar\xi(dx)\Q^*(d\bar\xi)=
    \int \int g(\bar\xi,x)\bar\xi(dx)\Q^*(d\bar\xi).
\end{align*}
Pick $g$ bounded and measurable. Note that $(0,0)\in \bar\xi$ for all $\bar\xi\in \supp(\Q^*)$. By \eqref{eq:event_time_stat_Q^*}\begin{align*}
    \int \int g(\theta_{x}\bar\xi,-x)\bar\xi(dx)\Q^*(d\bar\xi)&=\sum_{i\in \IZ}\int g(\theta_{\x_i}\bar\xi,-\x_i(\bar\xi))\Q^*(d\bar\xi)
    =\sum_{i\in \IZ}\int 
    g(\theta_{\x_i}\bar\xi,\x_{-i}(\theta_{\x_i}\bar\xi))
    \Q^*(d\bar\xi)\\
    &=\sum_{i\in \IZ}\int
    g(\bar\xi,\x_{-i}(\bar\xi))
    \Q^*(d\bar\xi)= \int\int 
    g(\bar\xi,x)
    \bar\xi(dx)\Q^*(d\bar\xi).
\end{align*}
Let $\Q$ be the stationary measure with $\Q^0=\Q^*$. 
The monotonicity of the elements of $\supp(\Q^*)$ implies the monotonicity of the elements of $\supp(\Q)$ by the inversion formula (cf.\ Remark \ref{rem:inversion}).
The definition of $\Q^0=\Q^*$ (cf. \eqref{eq:bij}) shows that $\iota(\Q)=\U$. Hence, the map $\iota$ is surjective. 
To prove injectivity let $\Q,\Q'\in \Cpl_{s,m}(\P_0,\P_1)$ with $\iota(\Q)=\iota(\Q')$. Let $\Gamma^*:=\{\bar\xi\in \Gamma_{\IR^2}\mid \bar\xi\text{ is monotone and }(0,0)\in \bar\xi)\}$. Note that the map $(\sq\circ\PR_0,\sq\circ\PR_1):\Gamma^*\to \SQ^2$ is bijective. Hence, the definition of $\iota$  
shows that $\Q^0=\Q'^0$ and thus $\Q=\Q'$ (cf. \cite[Satz 2.4]{Mecke1967StationreZM}).
This proves that the map $\iota$ is a bijection.
\end{proof}

\begin{cor}\label{cor:metric space}
The space $\mathcal P_{s,1}(\Gamma)$ equipped with $\mathscr W_p$, $p\geq 1$ as well as $\mathcal P_{s,1}(\SQ)$ equipped with $\mathscr W_{gap,p}$ is an extended metric space.
\end{cor}
\begin{proof}
It is clear that $\mathcal P_{s,1}(\SQ)$ equipped with $\mathscr W_{gap,p}$ is an extended metric space. The triangle inequality follows by a gluing argument (cf.\ \cite[Section 2.4]{erbar2023optimal} for a similar argument in a stationary setup).
    Hence, Theorem \ref{thm:equivalent_cost} yields that also $(\mathcal P_{s,1}(\Gamma),\mathscr W_p)$ is an extended metric space. 
\end{proof}

We will be interested in geodesics w.r.t.\ the gap metric $ \mathscr W_{gap,p}$. Similar to the classical Wasserstein metric, geodesics are obtained by interpolation of optimal couplings. 
The interpolation maps are defined as follows.

\begin{defi}
    We define for $0\leq t\leq 1$ the interpolation maps 
    \begin{align*}
    T_t:\SQ^2\mapsto \SQ, (\s,\s')\mapsto ((1-t)\s_i+t\s'_i)_{i\in \IZ}.
    \end{align*}
\end{defi}
We note that the maps $T_t$ are continuous with respect to the product topology on $\SQ^2$, where $\SQ$ is equipped with the (infinite) product topology. Moreover,
it is immediate to check that for an optimal coupling $\U\in \Cpl_s(\Pi_0,\Pi_1)$ the curve $((T_t)_{\#}\U)_{0\leq t\leq 1}$ is a constant speed geodesic, i.e., that 
\begin{align}\label{eq: geode}
    \mathscr W_{gap,p}((T_s)_{\#}\U,(T_t)_{\#}\U)=(t-s)\mathscr W_{gap,p}(\Pi_0,\Pi_1),\quad \forall 0\leq s<t\leq 1.
\end{align}
Indeed, it is sufficient to consider the couplings $(T_s,T_t)_\#\U$ for an optimal coupling $\U\in\Cpl_s(\Pi_0,\Pi_1)$ to obtain
$$\mathscr W_{gap,p}((T_s)_{\#}\U,(T_t)_{\#}\U)\leq(t-s)\mathscr W_{gap,p}(\Pi_0,\Pi_1)$$
which together with the triangle inequality gives the result.


\begin{thm}\label{thm: ext metric gap}
    The space $\mathcal P_{s,1}(\Gamma)$ equipped with $\mathscr W_p$, $p\geq 1$, is a geodesic extended metric space. 
    For $\P_0,\P_1\in \mathcal P_{s,1}(\Gamma)$ and an $\mathscr W_p-$optimal $\Q\in \Cpl_{s,m}(\P_0,\P_1)$, the curve 
    \begin{align}
        [0,1]\ni t\mapsto \P_t=\j^{-1}((T_t)_{\#}\iota(\Q))
    \end{align}
    defines a constant speed geodesic.
\end{thm}
\begin{proof}
   By Corollary \ref{cor:metric space}, $\mathcal P_{s,1}(\SQ)$ equipped with $\mathscr W_{gap,p}$ is an extended metric space. The existence of constant speed geodesics follows from equation \eqref{eq: geode}.
\end{proof}

The next Lemma shows that interpolations inherit finite second moments for the number statistics. We will need it to apply the representation result Lemma \ref{lem: georgii lem} (which is \cite[Lemma 3.2]{georgii2}) along geodesics. Moreover, it ensures that hyperuniformity is preserved along geodesics, cf.\ Corollary \ref{cor:hyperunif_preserved} which will be important in the applications to long-range interactions.

\begin{lem}\label{lem:second_moment_interpol}
    Let $\P_i\in \mathcal P_{s,1}(\Gamma)$, $i=0,1$ with $\int \xi(\Lambda_1)^2\P_i(d\xi)<\infty$ and  $\Q\in \Cpl_{s,m}(\P_0,\P_1).$ Set $\U:=\iota(\Q)$ and  $\P_t:=\j^{-1}\left((T_t)_{\#}\U\right)$. Then for $0\leq t\leq 1$ and $n\in \IN$ 
    \begin{align*}
        \int \xi(\Lambda_n)^2\P_t(d\xi)\leq (1-t)\int \xi(\Lambda_n)^2\P_0(d\xi)+ t\int \xi(\Lambda_n)^2\P_1(\xi).
    \end{align*}
\end{lem}
\begin{proof}

For $n\in \IN$ by the refined Campbell formula Theorem \ref{thm:refinedCampbell} applied to the function $f(\xi,x)=\xi(\Lambda_n)\eins_{\Lambda_n}(x)$ we obtain  
\begin{align*}
  &  \int \xi(\Lambda_n)^2\P_t(d\xi)=
\int \sum_{x\in \xi\cap \Lambda_n}\xi(\Lambda_n)\P_t(d\xi)
=\int_{\Lambda_n}\int \theta_{-x}\xi(\Lambda_n)
\P^0_t(d\xi)dx\\
&=\int_{\Lambda_n}\int \xi(\Lambda_n-x)
\P^0_t(d\xi)dx
=\int_{\Lambda_n} \int \sum_{(y,z)\in \bar\xi}\eins_{\Lambda_n-x}((1-t)y+tz)\Q^0(d\bar\xi)dx\\
&=\int \sum_{(y,z)\in \bar\xi}\abs{\Lambda_n\cap (\Lambda_n-((1-t)y+tz))}\Q^0(d\bar\xi)=\int \sum_{(y,z)\in \bar\xi}h_n((1-t)y+tz)\Q^0(d\bar\xi),
\end{align*}   
where in the last line $h_n$ is defined by $h_n(x):=\abs{\Lambda_n\cap (\Lambda_n-x)}$, $x\in \IR$. Since the function $h_n$ is convex on $(-\infty,0)$ and on $(0,\infty)$, we have by monotonicity of $\bar\xi$ using that $(0,0)\in\bar\xi$
\begin{align*}
    \int \sum_{(y,z)\in \bar\xi}h_n((1-t)y+tz)\Q^0(d\bar\xi)&\leq 
    (1-t)\int \sum_{(y,z)\in \bar\xi}h_n(y)\Q^0(d\bar\xi)+
    t\int \sum_{(y,z)\in \bar\xi}h_n(z)\Q^0(d\bar\xi)\\
    &=(1-t)\int \sum_{y\in \xi}h_n(y)\P_0^0(d\xi)+
    t\int \sum_{z\in \xi}h_n(z)\P_1^0(d\xi)\\
    &=(1-t)\int \xi(\Lambda_n)^2\P_0(d\xi)+t\int\xi(\Lambda_n)^2\P_1(d\xi),
\end{align*}
where the last line follows by the same calculations carried out at the beginning of this proof.
\end{proof}

As a direct corollary of the preceding Lemma, we obtain that hyperuniformity is preserved along geodesics.

\begin{cor}\label{cor:hyperunif_preserved}
Let $\P_i\in \mathcal P_{s,1}(\Gamma)$, $i=0,1$, be hyperuniform and $\Q\in \Cpl_{s,m}(\P_0,\P_1)$. Set $\U:=\iota(\Q)$ and  $\P_t:=\j^{-1}\left((T_t)_{\#}\U\right)$. Then $\P_t$ is hyperuniform for every $t\in [0,1]$. Moreover, if $\P_0$ and $\P_1$ are class-1-hyperuniform, then also $\P_t$ is class-1-hyperuniform.
\end{cor}
\begin{proof}
    This follows from Lemma \ref{lem:second_moment_interpol} and the fact that \begin{align*}
        \int (\xi(\Lambda_n)-n)^2\P_t(d\xi)=\int \xi(\Lambda_n)^2\P_t(d\xi)-n^2,\quad \forall n\in \IN.
    \end{align*}
\end{proof}

 In the remaining part of this section we show that the specific relative entropy is weakly geodesically convex.
 We start by calculating how the relative entropy $\ent(\cdot\mid \Leb_{\IR^n})$ changes if the reference measure is changed to the product of  exponential distributions $\gamma^{\otimes_{i=1}^n}$. 

\begin{prop}\label{prop:entropy_exp}
    Let $\mu$ be a measure on $[0,\infty)^{n}$ such that $\int x_i\mu(dx)=\int x_1\mu(dx)$ for all $i=1,\dots,n$.
    Then \begin{align*}
        \ent(\mu\mid \Leb_{\IR^{n}})=\ent(\mu\mid\gamma^{\otimes_{i=1}^n} )-n\int x_1\mu(dx).
    \end{align*}
    \end{prop}
\begin{proof}
    Note that $\frac{d\gamma^{\otimes_{i=1}^n}}{d\Leb_{\IR_+^{n}}}(x_{1},\dots,x_n)=\Pi_{i=1}^ne^{-x_i}$ and hence 
    \[
    \frac{d\mu}{d\Leb_{\IR_+^{n}}}(x_{1},\dots,x_n)=\frac{d\mu}{d\gamma^{\otimes_{i=1}^n}}(x_{1},\dots,x_n)\Pi_{i=1}^ne^{-x_i}.
    \]
    This yields 
    \begin{align*}
        \ent\left(\mu\mid \Leb_{\IR_+^{n}}\right)&=
        \int_{\IR_+^{n}} \left( \log\left(\frac{d\mu}{d\gamma^{\otimes_{i=1}^n}}(x)\right)+\sum_{i=1}^n-x_i\right)\frac{d\mu}{d\gamma^{\otimes_{i=1}^n}}(x)d\gamma^{\otimes_{i=1}^n}(x)\\
        &=\ent(\mu\mid\gamma^{\otimes_{i=1}^n} )-\sum_{i=1}^n\int_{\IR_+^{n}} x_id\mu(x)\\
        &=\ent(\mu\mid\gamma^{\otimes_{i=1}^n} )-n\int x_1\mu(dx)\\
    \end{align*}
\end{proof}

Using the classical results for the displacement convexity of the relative entropy and the  representation Lemma \ref{lem:cost=lim_wasserstein} of $\mathscr W_{gap,p}$, we finally  obtain the following weak geodesic convexity result for the specific relative entropy. Let us stress that we show that there exists an optimal coupling such that \eqref{eq: mnb} holds. We do not prove that this holds for every optimal coupling since we do not know whether every optimal coupling is an accumulation point of a sequence of couplings as constructed in the proof below.

\begin{thm}\label{thm:convnex_entropy}
    Let $\Pi_i\in \mathcal P_{s,1}(\SQ)$, $i=0,1$ and assume that $\mathscr  W_{gap,p}(\Pi_0,\Pi_1)<\infty$ for $p>1$. Then, there exists an optimal coupling  $\U\in \Cpl_s(\Pi_0,\Pi_1)$, such that for any $0\leq t\leq 1$ \begin{align}\label{eq: mnb}
       \mathcal E^*((T_t)_{\#}\U) \leq (1-t) \mathcal E^*(\Pi_0)+t\mathcal E^*(\Pi_1).
    \end{align}
    In particular, for $\P_i\in \mathcal P_{s,1}(\Gamma)$ with $\mathscr  W_{p}(\P_0,\P_1)<\infty$ there exists an optimal coupling  $\Q\in \Cpl_{s,m}(\P_0,\P_1)$, such that for $\P_t:=\j^{-1}((T_t)_{\#}\iota(\Q))$ we have \begin{align}
       \mathcal E(\P_t)\leq (1-t)  \mathcal E(\P_0)+t\mathcal E(\P_1)
    \end{align}

\end{thm}
\begin{proof}
Let $n\in \IN$ and  $\U^n$ be a minimizer for \[
\inf_{\tilde \U\in \Cpl(\Pi^{1,n}_0,\Pi^{1,n}_1)}\int \norm{x-y}_p^pd\tilde \U(x,y).
\]
By the displacement convexity of the relative entropy, see \cite[Remark 5.16]{villani2003topics}, \begin{align}\label{eq:displ_convx_n}
    \ent((T_t)_{\#}\U_n\mid \Leb_{\IR^{n}})
        \leq (1-t)\ent(\Pi_0^{1,n}\mid \Leb_{\IR^{n}})
        +t\ent(\Pi_1^{1,n}\mid \Leb_{\IR^{n}}).
\end{align}
Let $\paste^n,\bar\paste^n$ and $\bar\sigma^j$ be the maps defined in the proof of Lemma \ref{lem:cost=lim_wasserstein}.
For $i=0,1$ we set $\Pi_{i,n}^{\otimes}:=\paste^n_{\#}\left(\bigotimes_{m\in \IZ}\Pi_{i}^{1,n}\right)$ and $\bar{\Pi}_{i,n}:=n^{-1}\sum_{j=1}^n\sigma^j_{\#}\Pi_{i,n}^{\otimes}$.
Define \[\U^{\otimes}_n:=\bar\paste^n_{\#}\left(\bigotimes_{m\in \IZ}\U_n\right) \text{ and } \bar{\U}_n:=n^{-1}\sum_{j=1}^n\bar\sigma^j_{\#}\U^{\otimes}_n.\]
Then $\bar{\U}_n\in \Cpl_s(\bar{\Pi}_{0,n},\bar{\Pi}_{1,n})$ and $\bar{\Pi}_{i,n}\xrightarrow{n\to \infty}\Pi_i$ weakly. This follows exactly as in the proof of Lemma \ref{lem:cost=lim_wasserstein}.
Hence, by considering a subsequence, we can assume that $\bar{\U}^n\xrightarrow{n\to \infty}\U\in \Cpl_s(\Pi_0,\Pi_1)$ weakly.
The coupling $\U$ (and any other accumulation point) is $\mathscr W_{gap,p}-$optimal, because for  $\U'\in \Cpl_s(\Pi_0,\Pi_1)$\begin{align*}
    \int \abs{\s_1-\s'_1}^p\U'(d\s,d\s')&\stackrel{\text{stationarity}}{=}\lim_{n\to \infty}n^{-1} \int \sum_{j=1}^n\abs{\s_j-\s'_j}^p\U'(d\s,d\s')\\
    &\stackrel{\text{optimality}}{\geq} \limsup_{n\to \infty}n^{-1} \int \sum_{j=1}^n\abs{\s_j-\s'_j}^p\U_n(d\s,d\s')\\
     &=\limsup_{n\to \infty}\int \abs{\s_1-\s'_1}^p\bar{\U}_n(d\s,d\s')\\
     &\stackrel{\text{Fatou}}{\geq}
     \int \abs{\s_1-\s'_1}^p\U(d\s,d\s').
\end{align*}
Let $\Pi_t:=(T_t)_{\#}\U$. Since $\U$ is stationary, we have $\Pi_t\in \mathcal P_{s,1}(\SQ)$. Then for $n\in \IN$ by convexity of the relative entropy
\begin{align}\label{eq:xyz}  
    \mathcal E^ *\left((T_t)_{\#}\bar{\U}_n \right)&=\lim_{m\to \infty}m^{-1}\ent(\left((T_t)_{\#}\bar{\U}_n\right)^{1,m}|\gamma^{\otimes_{i=1}^m})\\
    &=\lim_{m\to \infty}m^{-1}\ent(n^{-1}\sum_{k=1}^n(\sigma^k)_{\#}\left((T_t)_{\#}\U^{\otimes}_n\right)^{1,m}|\gamma^{\otimes_{i=1}^m})\nonumber \\
    &\leq \lim_{m\to \infty}m^{-1}n^{-1}\sum_{k=1}^n\ent((\sigma^k)_{\#}\left((T_t)_{\#}\U^{\otimes}_n\right)^{1,m}|\gamma^{\otimes_{i=1}^m})
    \nonumber \\
    &= n^{-1} \sum_{k=1}^n\lim_{m\to \infty}
m^{-1}\ent((\sigma^k)_{\#}\left((T_t)_{\#}\U^{\otimes}_n\right)^{1,m}|\gamma^{\otimes_{i=1}^m})\nonumber \\
    &=  n^{-1} \sum_{k=1}^n n^{-1}\ent\left((T_t)_{\#}\U_n\mid \gamma^{\otimes_{i=1}^n}\right) \nonumber  \\
&=n^{-1}\ent\left((T_t)_{\#}\U_n\mid \gamma^{\otimes_{i=1}^n}\right), \nonumber 
\end{align} 
where in the second to last step we used that $\ent\left(m^{\otimes k}|\nu^{\otimes k}\right)=k\ent(m|\nu).$ The continuity of the map $T_t$ implies the weak convergence 
$(T_t)_{\#}\bar{\U}_n\xrightarrow{n\to \infty}\Pi_t$.
By  lower semicontinuity of $\mathcal E^*$ we obtain 
\begin{align*}
    \mathcal{E}^*(\Pi_t)&\leq\liminf_{n\to \infty}\mathcal E^*\left((T_t)_{\#}\bar{\U}_n \right)\\
    &\stackrel{\eqref{eq:xyz}}{\leq} \liminf_{n\to \infty}n^{-1}\ent((T_t)_{\#}\U_n\mid \gamma^{\otimes_{i=1}^n})\\
        &\stackrel{Prop. \ref{prop:entropy_exp}}{=}
        \liminf_{n\to \infty}n^{-1} \ent((T_t)_{\#}\U_n\mid \Leb_{\IR^{n}})
        +\int (1-t)\s_1+t\s_1'\U_n(d\s,d\s')\\
        &=
        \liminf_{n\to \infty}n^{-1} \ent((T_t)_{\#}\U_n\mid \Leb_{\IR^{n}})
        +(1-t)\int \s_1\Pi_0(d\s)+t\int \s_1\Pi_1(d\s)\\
        &\stackrel{\eqref{eq:displ_convx_n}}{\leq} \liminf_{n\to \infty}n^{-1} \left((1-t)\ent(\Pi_0^{1,n}\mid \Leb_{\IR^{n}})
        +t\ent(\Pi_1^{1,n}\mid \Leb_{\IR^{n}})\right)\\
        &+(1-t)\int \s_1\Pi_0(d\s)+t\int \s_1\Pi_1(d\s)\\
        &\leq (1-t) \limsup_{n\to \infty}n^{-1} \ent(\Pi_0^{1,n}\mid \Leb_{\IR^{n}})
       +(1-t)\int \s_1\Pi_0(d\s)\\
       &+ t \limsup_{n\to \infty}n^{-1} \ent(\Pi_1^{1,n}\mid \Leb_{\IR^{n}})
       +t\int \s_1 \Pi_1(d\s)\\
       &\stackrel{Prop. \ref{prop:entropy_exp}}{=}
       (1-t)\limsup_{n\to \infty}n^{-1}\ent(\Pi_0^{1,n}\mid \gamma^{\otimes_{i=1}^n})\\
       &+t\limsup_{n\to \infty}n^{-1}\ent(\Pi_1^{1,n}\mid \gamma^{\otimes_{i=1}^n})\\
       &=(1-t)\mathcal E^*(\Pi_0)+t\mathcal E^*(\Pi_1).
\end{align*}
The second part of the statement immediately follows from Theorem \ref{thm:equivalent_cost} and Theorem \ref{thm:equiv_entropies}.
\end{proof}

\section{Free gaps -- EVI and HWI}\label{sec:evihwi}

The goal of this section is to gain first insights into the geometry induced by $\mathscr{W}_{2}$ by considering the gradient flow of the specific entropy w.r.t.\ $\mathscr{W}_{2}$. It turns out that the so called EVI formulation of the gradient flow can be easily lifted from $\IR^d$. We refer to \cite{AGS08, Villani, Santa}  for an introduction to the gradient flows in metric spaces.
Recall from \eqref{eq:Cp} that we denote  the classical $L^2$ Kantorovich Wasserstein metric by $\mathcal C_2^{\frac 12}$. 
The gradient flow description will rely on the solution to the following SDE that  arises  in the study of the Fokker-Planck equation on $[0,\infty)^n$ with Neumann boundary conditions.

\begin{defi} \label{def:SDE}
For probability measures $\Pi$ on $[0,\infty)^I$, $I\subset \IZ$, we let $\Semi_t^{gap}\Pi:=\mathsf{law}(X_t)$, where $(X_t)_{t\geq 0}$ is  the unique weak  solution (cf.\ \cite[Remark 3.1.2]{pilipenko}) of the SDE \begin{align}\label{eq:SDE}
        dX^i_t=-dt+l^i_t+\sqrt{2}dB^i_t,\quad i\in I,
    \end{align}
    with $\mathsf{law}(X_0)=\Pi$ where
  $(B^i)_{i\in I}$ are independent standard Brownian motions and $(l^i_t)_{t\geq 0}$ is the local time of $X^i$ in $0$ for $i\in I$.
\end{defi}

We now state and prove the integral form of the EVI of the specific relative entropy w.r.t.\ $\mathscr W_2$.  The representation formulas Theorem \ref{thm:equiv_entropies} and Lemma \ref{lem:cost=lim_wasserstein} for  the specific relative entropy  and the extended metric  in terms of the gap distributions allow for a simple lifting procedure. Since the evolution of the coordinates of the $\mathcal C_2^{\frac 12}$-gradient flow of the relative entropy is given by the SDE \eqref{eq:SDE} we obtain that the coordinates of the  gap distributions of the gradient flow of the specific relative entropy w.r.t.\ $\mathscr W_2$ evolve independently and according to  \eqref{eq:SDE}.

\begin{thm}\label{thm:EVI_1_dim}
    Let $\P,\P^* \in \mathcal P_{s,1}(\Gamma)$.  For $\P_t:=\j^{-1}(\mathsf S^{gap}_t\j(\P))$ the integral form EVI holds 
    \begin{align}
        \mathscr W_{2}(\P_t, \P^*)-\mathscr W_{2}(\P,\P^*)\leq 2t (\mathcal E(\P^*)-\mathcal E(\P_t))
    \end{align}
\end{thm}
\begin{proof}
Set $\Pi:=\j(\P)$, $\Pi^*:=\j(\P^*)$ and $\Pi_t:=\mathsf S^{gap}_t\j(\P)$.
    Fix $n\in \IN$ and define the functional 
    \begin{align*}
         F_n: P_2(\IR^n)\to \IR\cup\{+\infty\}
    \end{align*}
    by $ F_n(\mu):=\ent(\mu\mid \Leb_{\IR^n})+\sum_{i=1}^n\int x_i\mu(dx)$. Here $P_2(\IR^n)$ is the set of probability measures on $\IR^n$ with finite second moment.
By \cite[Proposition 8.10]{Santa} the $\mathcal C_2^{\frac 12}$-gradient flow $(\mu_t)_{t\geq 0}$ of $F_n$, started at $\mu$, is given by the distributional solution of the Fokker-Planck equation 
\begin{align*}
    \mu_t\xrightarrow{t\to 0}\mu \text{ and } \partial_t\mu_t=\Delta \mu_t+\sum_{i=1}^n\partial_{x_i}\mu_t
\end{align*}
with Neumann boundary conditions on $[0,\infty)^n$.
The solution to the above equation is given by $\mu_t=\mathsf S_t^{gap}\mu$ , cf.\ \cite[Theorem 3.1.1]{pilipenko}.
Hence,  the integral characterization of the EVI \cite[Proposition 3.1]{Daneri_2008} yields 
\begin{align}\label{eq:evi_fixed_n}
    &\mathcal C_2(\Pi_t^{1,n},(\Pi^*)^{1,n})-\mathcal C_2(\Pi^{1,n},(\Pi^*)^{1,n})\\
    &\leq 2t\left(
     F_n((\Pi^*)^{1,n})- F_n(\Pi_t^{1,n})
    \right)= 2t\left(
    \ent ((\Pi^*)^{1,n}\mid \gamma^{\otimes_{i=1}^n})-\ent (\Pi_t^{1,n}\mid \gamma^{\otimes_{i=1}^n})\right),\nonumber
\end{align}
where the second equality follows from Proposition \ref{prop:entropy_exp}.
 By Lemma \ref{lem:cost=lim_wasserstein} we have 
 \begin{align*}
  n^{-1}   \mathcal C_2(\Pi_t^{1,n},(\Pi^*)^{1,n})\xrightarrow{n\to \infty}\mathscr W^2_{gap,2}(\Pi_t,\Pi^*)
 \end{align*}
and 
    \begin{align*}
  n^{-1}   \mathcal C_2(\Pi^{1,n},(\Pi^*)^{1,n})\xrightarrow{n\to \infty}\mathscr W^2_{gap,2}(\Pi,\Pi^*).
 \end{align*}
 Hence, dividing by $n$ in \eqref{eq:evi_fixed_n} and letting $n\to \infty$ yields the claim.
\end{proof}

We now define the Fisher information $I$ for probability measures on $\IR^n$, $n\geq 1$.
For a probability measure $\mu\in \mathcal{P}(\IR^n)$ with $\mu=\rho\Leb_{\IR^n}$ s.t.\ $\rho\in W^{1,1}_{\sf loc}(\IR^n)$ and such that there is $w\in L^2(\mu)$ with $\rho w= \nabla \rho$  the  Fisher information is defined by 
\begin{equation}\label{eq:Fisher_Rd}
I(\mu|\Leb_{\IR^n}) := \int_{\IR^n}\norm{w}^2d\mu\;.
\end{equation}
Otherwise, we set $I(\mu|\Leb_{\IR^n}):=+\infty$. 
 The following definition is the canonical choice for the Fisher information of the gap distributions.
    For $\Pi\in \mathcal P_{s,1}(\SQ)$ we set \begin{align}
        \mathcal I^*\left(\Pi\right):=\limsup_{n\to \infty} \frac{1}{n}I\left(\Pi^{1,n}\mid \gamma^{\otimes_{i=1}^n}\right).
        \end{align}
The following HWI inequality now follows simply by combining the representation formulas Theorem \ref{thm:equiv_entropies} and Lemma \ref{lem:cost=lim_wasserstein}  with the above definition.

\begin{prop}
    Let $\P_i\in \mathcal P_{s,1}(\Gamma)$, $i=0,1$ with $\mathcal E(\P_1)<\infty$.  Then the HWI inequality holds  \begin{align}
     \mathcal E(\P_0)-\mathcal E(\P_1)\leq \mathscr W_{2}\left(\P_0,\P_1\right) \sqrt{\mathcal I^*\left(\j(\P_0)\right)}.
    \end{align}
\end{prop}
\begin{proof}
Let $\Pi_i:=\j(\P_i)$ for $i=0,1$.
We have by Theorem \ref{thm:equiv_entropies}
    \begin{align}\label{eq:HWI_calcul}
        \mathcal E^*(\Pi_0)-\mathcal E^*(\Pi_1)&=\lim_{n\to \infty} n^{-1}\left(\ent(\Pi^{1,n}_0\mid \gamma^{\otimes_{i=1}^n})-\ent(\Pi^{1,n}_1\mid \gamma^{\otimes_{i=1}^n}) \right)
    \end{align}
The HWI inequality on $\IR^n$, see \cite[Corollary 20.13]{Villani}, yields \begin{align*}
    \ent(\Pi^{1,n}_0\mid \gamma^{\otimes_{i=1}^n})-\ent(\Pi^{1,n}_1\mid \gamma^{\otimes_{i=1}^n})&\leq \mathcal C_2^{\frac12}\left(\Pi^{1,n}_0,\Pi^{1,n}_1\right)\sqrt{I(\Pi_0^{1,n}\mid \gamma^{\otimes_{i=1}^n})}.
\end{align*}
Thus, Lemma \ref{lem:cost=lim_wasserstein} yields \begin{align*}
     \mathcal E^*(\Pi_0)-\mathcal E^*(\Pi_1)\leq \mathscr W_{2,gap}\left(\Pi_0,\Pi_1\right) \sqrt{\limsup_{n\to \infty}n^{-1}I(\Pi_0^{1,n}\mid \gamma^{\otimes_{i=1}^n})}.
\end{align*}

\end{proof}

\section{Short range interactions and convexity of the free energy}\label{sec:Interaction_energy}

In this section we will consider particular interaction energies of point processes and show that the corresponding free energy, entropy plus energy, is convex w.r.t.\ $\mathscr{W}_{gap,p}$ for $p>1.$

 We will restrict to interaction energies $H_n$ induced by  pair potentials $\phi$, that is 
 $$H_n(\xi)=\frac{1}{2}\sum_{x,y\in \Lambda_n\cap \xi,x\neq y}\phi(x-y)$$
for $\xi\in\Gamma$. An even potential $\phi:\IR\to \IR\cup\{\infty\}$ is called stable if for all $\xi\in \Gamma$ and $n\in \IN$
 \begin{align}\label{eq:stable}
    H_n(\xi):=\frac{1}{2}\sum_{x,y\in \Lambda_n\cap \xi}\phi(x-y)\geq -b\xi(\Lambda_n),
\end{align}
for some constant $b=b(\phi)\in \IR$.
Let 
\begin{equation}\label{eq:gn}
h_n(x)=\abs{\Lambda_n\cap (\Lambda_n-x)} \ \text{ and }\ g_n(x):= \phi(x)h_n(x).
\end{equation}
Note that $h_n(x)=0$ for $\abs{x}\geq n$. 
The following formula (cf.\ \cite[Lemma 3.2]{georgii2}) allows to rewrite the finite box interaction energy of a stationary point process in terms of its Palm measure and the function $g_n$.

\begin{lem}\label{lem: georgii lem}
    Let $\phi$ be a stable potential and  $\P\in \mathcal P_{s,1}(\Gamma)$ with $\int \xi(\Lambda_1)^2\P(d\xi)<\infty$. Then for $n\in \IN$
    \begin{align}\label{eq:formula_hamiltonian}
       \int H_n(\xi)\P(d\xi)= \int \frac{1}{2}\sum_{0\neq x\in \xi}g_n(x)\P^0(d\xi).
    \end{align}
\end{lem}
\begin{rem}
    The proof of  \cite[Lemma 3.2]{georgii2} shows that the assumption $\int \xi(\Lambda_1)^2\P(d\xi)<\infty$ can be dropped if $\phi\geq 0$.
\end{rem}

We give a direct proof of the above lemma for the special case of the logarithmic potential $\phi(x)=-\log\abs{x}$, which is not stable. The proof is based on the proof of \cite[Lemma 3.2]{georgii2}.
\begin{lem}\label{lem:superstable_log}
    Let $\phi(x)=-\log\abs{x}$  and  $\P\in \mathcal P_{s,1}(\Gamma)$ with $\int \xi(\Lambda_1)^2\P(d\xi)<\infty$. Then for $n\in \IN$
    \begin{align}
       \int H_n(\xi)\P(d\xi)= \int \frac{1}{2}\sum_{0\neq x\in \xi}g_n(x)\P^0(d\xi).
    \end{align}
\end{lem}
\begin{proof}
    Set $f_n(\xi,x)=\eins_{\Lambda_n}(x)\frac{1}{2}\sum_{0\neq y\in \xi}\phi(y)\eins_{\Lambda_n-x}(y)$ for $x\in \IR$ and $\xi \in \Gamma$. 
    We decompose into negative and positive parts as
    \begin{align}\label{eq:logpalmenergy}
        &2f_n(x,\xi)\nonumber \\&=\eins_{\Lambda_n}(x)\xi((\Lambda_n-x)\cap(0,\infty)) \sum_{ y\in \xi\cap (\Lambda_n-x)\cap(0,\infty)}\xi((\Lambda_n-x)\cap(0,\infty))^{-1}\phi(y)\nonumber \\
        &+\eins_{\Lambda_n}(x)\xi((\Lambda_n-x)\cap(-\infty,0)) \sum_{ y\in \xi\cap (\Lambda_n-x)\cap(-\infty,0)}\xi((\Lambda_n-x)\cap(-\infty,0))^{-1}\phi(y)
    \end{align}
     By  convexity of $\phi$ on   $(0,\infty)$ for the first term we have  the lower bound 
    \begin{align*}
        &\eins_{\Lambda_n}(x)\xi((\Lambda_n-x)\cap(0,\infty)) \sum_{ y\in \xi\cap (\Lambda_n-x)\cap(0,\infty)}\xi(\Lambda_n-x\cap(0,\infty))^{-1}\phi(y)\\
          &\geq \eins_{\Lambda_n}(x)\xi((\Lambda_n-x)\cap(0,\infty)) \phi\left(\sum_{ y\in \xi\cap (\Lambda_n-x)\cap(0,\infty)}\xi((\Lambda_n-x)\cap(0,\infty))^{-1}y\right)\\
          &\geq \eins_{\Lambda_n}(x)\xi((\Lambda_n-x)\cap(0,\infty)) \phi\left(\sum_{ y\in \xi\cap (\Lambda_n-x)\cap(0,\infty)}\xi((\Lambda_n-x)\cap(0,\infty))^{-1}n\right)\\
          &= \eins_{\Lambda_n}(x)\xi((\Lambda_n-x)\cap(0,\infty)) \phi\left(n\right)
    \end{align*}
    Analogously, we can argue for the second term in \eqref{eq:logpalmenergy} to get 
    $$2f_n(x,\xi)\geq \eins_{\Lambda_n}(x)\xi(\Lambda_n-x) \phi\left(n\right).$$
    By the refined Campbell theorem  applied to the function $\eins_{\Lambda_n}(x)\xi(\Lambda_n-x)$  we have 
    \begin{align*}
        \int_{\IR}\int \eins_{\Lambda_n}(x)\xi(\Lambda_n-x)\P^0(d\xi)dx&=\int\sum_{x\in \xi}\eins_{\Lambda_n}(x)\xi(\Lambda_n)\P(d\xi)=
       \int\xi^2(\Lambda_{n})\P(d\xi)<\infty.
    \end{align*}
     By  the last two estimates, we can apply the refined Campbell theorem to the function $f_n$, which yields (recall $g_n$ from \eqref{eq:gn})
    \begin{align*}
     \int \frac{1}{2}\sum_{0\neq x\in \xi}g_n(x)\P^0(d\xi)&=\int_{\IR}\int f_n(x,\xi) \P^0(d\xi)dx=\int\sum_{x\in \xi}f_n(x,\theta_x\xi)\P(d\xi)
        =\int H_n(\xi)\P(d\xi).
    \end{align*}
\end{proof}

Our goal is to establish weak geodesic convexity  of the interaction energy (cf.\ equation \eqref{def:interac}) which will be defined as a volume normalised limit of the expectations $\int H_n(\xi)\P(d\xi)$. 
In the light of Lemma \ref{lem: georgii lem}, one way to guarantee this convexity, is to assume the function $g_n$ to be convex. 
We make this ansatz precise in the following assumption, where we impose a quantitative lower bound on the second derivative of $g_n$.

\begin{assumption}\label{assumption0}
Let the even potential $\phi$  be twice differentiable on $(-\infty,0)$ and $(0,\infty)$.
Let $h_n(x):=\abs{\Lambda_n\cap (\Lambda_n-x)}$. Note that $h_n(x)=0$ for $\abs{x}\geq n$. Set $g_n(x):= \phi(x)h_n(x)$ and assume that \begin{align}\label{eq:function_f}
    g_n''(x)\geq n f(\abs{x}),\quad \forall x\in \Lambda_n\setminus \{0\}
\end{align}
for some continuous function $f:(0,\infty)\to (0,\infty)$. In particular, $g_n$ is convex on $(-n/2,0)$ and $(0,n/2)$.
\end{assumption}

Under this assumption, we immediately obtain the following strict convexity property of the finite box energies  $\int H_n(\xi)\P(d\xi)$. Moreover, we obtain an explicit formula for the gain which (under more assumptions) will later yield weak $\lambda$-geodesic convexity of the free energy (cf.\ Proposition \ref{prop:lambda_conv}).

\begin{lem}\label{lem:convex_fix_n}
    Let $\phi$ satisfy assumption \ref{assumption0} and assume that equation \eqref{eq:formula_hamiltonian} holds for any $n\in \IN$ and  $\P\in \mathcal P_{s,1}(\Gamma)$ with $\int \xi(\Lambda_1)^2\P(d\xi)<\infty$. Let $\P_i\in \mathcal P_{s,1}(\Gamma)$, $i=0,1$ with $\int \xi(\Lambda_1)^2\P_i(d\xi)<\infty$ and  $\Q\in \Cpl_{s,m}(\P_0,\P_1)$. 
    Set $\U:=\iota(\Q)$ and  $\P_t:=\j^{-1}\left((T_t)_{\#}\U\right)$.
    Then for $0\leq t\leq 1$ \begin{align}
        n^{-1}\int H_n(\xi) \P_{t}(d\xi)&\leq
        (1-t)n^{-1}\int H_n(\xi)\P_{0}(d\xi)+tn^{-1}\int H_n(\xi)\P_{1}(d\xi)\\
        &-\frac{(1-t)t}{2}\int \sum_{(0,0)\neq (x,y)\in \bar\xi\cap \Lambda_n^2} (x-y)^2\inf_{z\in [x,y]}f(\abs{z})\Q^0(d\bar\xi)\nonumber
    \end{align}
\end{lem}
\begin{proof}By convexity of $g_n$, we have
    \begin{align*}
       &2n^{-1}\int H_n(\xi)\P_{t}(d\xi)\\&\stackrel{\eqref{eq:formula_hamiltonian}}{=} n^{-1}\int \sum_{0\neq x\in \xi}g_n(x)\P_{t}^0(d\xi)\\&= n^{-1}
        \sum_{(0,0)\neq(x,y)\in \bar\xi}g_n((1-t)x+ty)\Q^0(d\bar\xi)\\
         &\leq  n^{-1} \int \sum_{(0,0)\neq (x,y)\in \bar\xi} \left((1-t)g_n(x)+tg_n(y)-\frac{t(1-t)}{2}(x-y)^2\inf_{z\in [x,y]}g_n''(z)\right)\Q^0(d\bar\xi)\\
         &\leq 2n^{-1}(1-t)\int H_n(\xi)\P_{0}(d\xi)+2n^{-1}t\int H_n(\xi)\P_{1}(d\xi)\\
         &-\frac{(1-t)t}{2n}\int \sum_{(0,0)\neq (x,y)\in \bar\xi\cap \Lambda_n^2} (x-y)^2\inf_{z\in [x,y]}g_n''(z)\Q^0(d\bar\xi) \\
          &\leq 2n^{-1}(1-t)\int H_n(\xi)\P_{0}(d\xi)+2n^{-1}t\int H_n(\xi)\P_{1}(d\xi)\\
          &-\frac{(1-t)t}{2}\int \sum_{(0,0)\neq (x,y)\in \bar\xi\cap \Lambda_n^2} (x-y)^2\inf_{z\in [x,y]}f(\abs{z})\Q^0(d\bar\xi),
    \end{align*}
    where in the third line we used the convexity of $g_n$.
\end{proof}

We define the gain 
\begin{align}\label{eq:gain}
\g(n,t,\Q^0):=\frac{(1-t)t}{2}\int \sum_{(0,0)\neq (x,y)\in \bar\xi\cap \Lambda_n^2} (x-y)^2\inf_{z\in [x,y]}f(\abs{z})\Q^0(d\bar\xi)
\end{align}
which quantifies the gain over plain convexity.
Furthermore, for $\P\in \mathcal P_{s,1}(\Gamma)$ set 
\begin{align}
    \mathcal H(\P,n):= \int H_n(\xi)\P(d\xi).
\end{align}

In order to make use of the results of \cite{georgii2} we need to impose some more assumptions on the pair potential $\phi$.
A potential $\phi$ is called purely repulsive if $\phi$ is
nonnegative and bounded away from zero near the origin, i.e., if there exists some positive $\delta=\delta(\phi)>0$ such that \begin{align*}
\phi(x)\geq \delta \eins_{\norm{x}\leq \delta}.
\end{align*}
A potential $\phi$ is called superstable if $\phi=\phi^s+\phi^r$ for stable $\phi^s$ and purely repulsive $\phi^r$. Finally, $\phi$ is called lower regular, if there exists a decreasing function $\psi:[0,\infty)\to [0,\infty)$ such that \begin{align*}
    \phi(x)\geq -\psi(\abs{x}),\quad \forall x\in \IR
\end{align*}and
\begin{align*}
    \int_0^{\infty}\psi(s)ds<\infty.
\end{align*}
A potential is called regular, if it is lower regular and if there exists $r(\phi)>0$ such that $\phi(x)\leq \psi(\abs{x})$ for $x\geq r(\phi)$.
For the rest of this section, we make the following assumptions.

\begin{assumption}\label{assumption1}
For the rest of this section, we fix    a  superstable and  regular potential $\phi$. Furthermore, we assume $\phi$ to satisfy assumption \ref{assumption0} and let $f$ be the function from assumption \ref{assumption0}.
\end{assumption}

We give two examples of pair potentials satisfying assumption \ref{assumption1}.

\begin{ex}[Hypersingular Riesz Gas]\label{example:riesz}
    Consider $\phi(x)=\abs{x}^{-s}$ for $s>1$. Then $g_n$ is twice differentiable on the intervals $(-n/2,0)$ and $(0,n/2)$  with \begin{align*}g_n''(x)=s(s+1)\abs{x}^{-s-2}h_n(x)+2s\abs{x}^{-s-1}.
    \end{align*}
   Since $h_n(x)\geq n/2$ for $x\in \Lambda_n\setminus \{0\}$,   we obtain the lower bound \begin{align*}
       g_n''(x)\geq \frac{s(s+1)n}{2}\abs{x}^{-s-2}.
   \end{align*}
  Hence, a possible choice is the convex function $f(\abs{x})=\frac{s(s+1)}{2}\abs{x}^{-s-2}$.
\end{ex}
\begin{ex}[Yukawa]\label{ex:yukawa}
    Let $\phi(x)=\frac{1}{\abs{x}}e^{-\abs{x}}$ and $\varepsilon>0$. The choice of the function $f$ will depend on $\varepsilon$. For $x\in (-n/2,0)\cup (0,n/2)$ and $n\geq 1$ large enough (dependent on $\varepsilon$) \begin{align*}
        g_n''(x)=\phi''(x)h_n(x)+2\phi'(x)h_n'(x)\geq \phi''(x)h_n(x)-2\abs{\phi'(x)}\geq \frac{n}{2}\left(\phi''(x)-\varepsilon\abs{\phi'(x)}\right),
    \end{align*}
    since $\abs{h_n'(x)}=1$ and $h_n(x)\geq n/2$ for $x\in \Lambda_n\setminus \{0\}$.
We have for $x\in \Lambda_n\setminus\{0\}$ \begin{align*}
    \phi''(x)-\varepsilon \abs{\phi'(x)}=e^{-\abs{x}}\left(2\abs{x}^{-3}+(2-\varepsilon)x^{-2}+(1-\varepsilon)\abs{x}^{-1}\right)
\end{align*}

    Hence, a possible choice is $f(x)=e^{-x}x^{-3}$ which is convex (on the positive axis).
\end{ex}

 So far, we considered the weak topology on $\mathcal P_{s,1}(\Gamma)$ (induced by the vague topology on $\Gamma$).  
In order to use results from \cite{georgii2}, such as lower semicontinuity of the free energy, we have to introduce another topology on $\mathcal P_{s,1}(\Gamma)$ (cf. \cite[Page 2]{georgii2}). 

\begin{defi}\label{def:local_top}
    Let $\mathcal L$  denote the class of all
measurable functions $f:\Gamma\to \IR$, which are local, in that $\exists n\in \IN$ such that $f(\xi)=f(\xi\cap \Lambda_n)$, $\forall \xi\in \Gamma$, and tame, in that $\abs{f(\xi)}\leq c(1+\xi(\Lambda_n))$ for $n\in \IN$, a constant $c\in \IR$ and any $\xi\in \Gamma$. The topology $\mathcal T_{\mathcal L}$ is defined as the coarsest topology on $\mathcal P_{s,1}(\Gamma)$ such that the mappings $\P\mapsto \IE_{\P}[f]$ are continuous for any $f\in L$. We note that the topology $\mathcal T_{\mathcal L}$ is finer than the weak topology on $\mathcal P_{s,1}(\Gamma)$ which is induced by the vague topology on $\Gamma$. 
\end{defi}

\begin{thm}[{\cite[Theorem 1 and Lemma 3.4]{georgii2}}]
   For every $\P\in \mathcal P_{s,1}(\Gamma)$, the limit
   \begin{align}\label{def:interac}
    \mathcal W^{int}(\P):=\lim_{n\to \infty}\frac{1}{n}\int H_n(\xi)\P(d\xi)
\end{align}  
exists. For   $\P\in \mathcal P_{s,1}(\Gamma)$ with $\int \xi(\Lambda_1)^2\P(d\xi)<\infty$ the interaction energy admits the  representation 
\begin{align}
     \wint(\P)=\frac{1}{2}\int \sum_{0 \neq x\in \xi}\phi(x)\P^0(d\xi).
\end{align}
For $\P\in \mathcal P_{s,1}(\Gamma)$ with $\int \xi(\Lambda_1)^2\P(d\xi)=\infty$ the limit is given by $\wint(\P)=\infty$.
For $\beta>0$ the free energy  $\mathcal F_{\beta}(\P):=\beta \wint(\P)+\mathcal E(\P)$
is lower semicontinuous w.r.t.\ the topology $\mathcal T_{\mathcal L}$. Moreover, by assumption \ref{assumption0}, the functional $\mathcal F_{\beta}$ is bounded from below by a finite constant.
\end{thm}


Combining the weak geodesic convexity of the specific relative entropy Theorem \ref{thm:convnex_entropy} and the convexity property of the finite box energies Lemma \ref{lem:convex_fix_n} we obtain the weak strict geodesic convexity of the free energy.

\begin{thm}\label{thm:interaction_gain}
  Let $\P_i\in  \mathcal P_{s,1}(\Gamma)$, $i=0,1$ with  finite distance $\mathscr W_p(\P_0,\P_1)<\infty$ and finite free energy $\mathcal F_{\beta}(\P_i)<\infty$, $i=0,1$. Let $\Q$ be the $\mathscr W_p-$optimal coupling from Theorem \ref{thm:convnex_entropy}. Then for $0\leq t\leq 1$\begin{align}\label{eq:lkj}
    \mathcal F_{\beta}(\P_{t})\leq (1-t)\mathcal F_{\beta}(\P_{0})+t\mathcal F_{\beta}(\P_{1}) -
        \frac{\beta(1-t)t}{2}\int \sum_{(0,0)\neq (x,y)\in \bar\xi} (x-y)^2\inf_{z\in [x,y]}f(\abs{z})\Q^0(d\bar\xi),
    \end{align}
    where $\U:=\iota(\Q)$,  $\P_t:=\j^{-1}\left((T_t)_{\#}\U\right)$ and $f$ is chosen as in \eqref{eq:function_f}.
\end{thm}
\begin{proof}
    By Theorem \ref{thm:convnex_entropy} and Lemma \ref{lem:convex_fix_n} \begin{align*}
        \mathcal F_{\beta}(\P_{t})&= \beta \lim_{n\to \infty}n^{-1}\int H_n \P_t(d\xi)+\mathcal E(\P_{t})\\
        &\leq (1-t)\mathcal F_{\beta}(\P_{0})+t\mathcal F_{\beta}(\P_{1}) -
        \frac{\beta(1-t)t}{2}\int \sum_{(0,0)\neq (x,y)\in \bar\xi} (x-y)^2\inf_{z\in [x,y]}f(\abs{z})\Q^0(d\bar\xi).
    \end{align*}
\end{proof}

The following approximation result is proved in \cite[Lemma 5.1]{georgii2}. It allows us to
approximate two point processes $\P_0,\P_1\in \mathcal P_{s,1}(\Gamma)$ with finite free energy but with $\mathscr W_p(\P_0,\P_1)=\infty$ by point processes which are at finite distance. To these approximations, we can then apply the  strategy of the proof of Theorem \ref{thm:interaction_gain} and thus 
prove the existence of a coupling $\Q\in \Cpl_{s,m}(\P_0,\P_1)$ such that \eqref{eq:lkj} holds.

\begin{lem}\label{lem:finite_energy_approx}
    For $\P\in  \mathcal P_{s,1}(\Gamma)$ with  $\mathcal F_{\beta}({P})<\infty$  there exists a sequence $\P^n\in \mathcal P_{s,1}(\Gamma)$ with $\wint(\P^n)\leq \wint(\P)+1/n$, $\mathcal E(\P^n)\leq \mathcal E(\P)+1/n$ and 
    \begin{align}
        \int \sum_{x\in \xi\cap \Lambda_1}\abs{x^+-x}^p\P^n(d\xi)<\infty.
    \end{align}
    Furthermore, $\P^n\xrightarrow{n\to \infty}\P$  w.r.t. $\mathcal T_{\mathcal L}$ (in particular weakly) and for $n\in \IN$ there exists $q_n>0$  such that the measure $\P^n$ is concentrated on the set \begin{align}\label{eq:sets_gamma_q}
        \Gamma_{q_n}=\{\xi\in \Gamma: \phi(x-y)\leq q_n,1/q_n\leq \abs{x-y} \text{ for any two distinct }x,y\in \xi \}.
    \end{align}
\end{lem}
\begin{proof}
This is the content of \cite[Lemma 5.1]{georgii2}.
   The only thing left to show is  that for $n\in \IN$\begin{align}
        \int \sum_{x\in \xi\cap \Lambda_1}\abs{x^+-x}^p\P^{n}(d\xi)<\infty.
    \end{align}
    Since \begin{align}
        \int \sum_{x\in \xi\cap \Lambda_1}\abs{x^+-x}^p\P^{n}(d\xi)\leq \int \xi(\Lambda_1)\P^{n}(d\xi)+
         \int \sum_{x\in \xi\cap \Lambda_1, x^+\notin \Lambda_1}\abs{x^+-x}^p\P^{n}(d\xi),
    \end{align}
    it is sufficient to show that $\int \sum_{x\in \xi\cap \Lambda_1, x^+\notin \Lambda_1}\abs{x^+-x}^p\P^{n}(d\xi)<\infty$.
    Since the measures $\P^n$ are obtained by stationarizing an infinite product measure (cf. the proof \cite[Lemma 5.1]{georgii2} for the precise construction, a similar construction is done in the proof of Lemma \ref{lem:cost=lim_wasserstein}),  it follows, that there exist $n-$dependent constants $a_n,b_n>0$ such that 
    \begin{align*}
        \P^n(\xi([1,r))=0)\leq a_n e^{-b_nr}.
    \end{align*}Hence,
    \begin{align*}
       \int \sum_{x\in \xi\cap \Lambda_1, x^+\notin \Lambda_1}\abs{x^+-x}^p\P^{n}(d\xi)
       &=\int \P^n\left(\sum_{x\in \xi\cap \Lambda_1, x^+\notin \Lambda_1}\abs{x^+-x}^p>r\right)dr\\
       &\leq \int \P^n\left(\xi([1,r^{1/p}))=0\right)dr\\
       &<\infty.
    \end{align*}
\end{proof}

We leverage Lemma \ref{lem:finite_energy_approx} to prove an extension of Theorem \ref{thm:interaction_gain} to the case of stationary point processes which are at potentially infinite distance. 

 \begin{thm}\label{thm:strictly_smaller_free_energy}
    For $\P_i\in  \mathcal P_{s,1}(\Gamma)$ with  $\mathcal F_{\beta}({\P_i})<\infty$, $i=0,1$, there exists $\Q\in \Cpl_{s,m}(\P_0,\P_1)$ with \begin{align}  
    \mathcal F_{\beta}(\P_{t})\leq  (1-t)\mathcal F_{\beta}(\P_{0})+t\mathcal F_{\beta}(\P_{1})-\frac{(1-t)t\beta}{2}\int \sum_{(0,0)\neq (x,y)\in \bar\xi} (x-y)^2\inf_{z\in [x,y]}f(\abs{z})\Q^0(d\bar\xi),
    \end{align}
     where  $\U:=\iota(\Q)$,  $\P_t:=\j^{-1}\left((T_t)_{\#}\U\right)$ and $f$ is chosen as in \eqref{eq:function_f}.
\end{thm}
\begin{proof}
Apply Lemma \ref{lem:finite_energy_approx} to obtain 
$\P^n_i\in \mathcal P_{s,1}(\Gamma)$ with $\wint(\P_i^n)\leq \wint(\P_i)+1/n$ and $\mathcal E(\P^n_i)\leq \mathcal E(\P_i)+1/n$ and 
    \begin{align}
        \int \sum_{x\in \xi\cap \Lambda_1}\abs{x^+-x}^p\P^n_i(d\xi)<\infty.
    \end{align}
Let $\Pi^n_i\in \mathcal P_{s,1}(\SQ)$ be the gap distribution of $\P^n_i$, i.e.\ $\Pi^n_i:=\j(\P^n_i)$. Then \begin{align*}
    \int \s_1^p\Pi^n_i(d\s)=\int \sum_{x\in \xi\cap \Lambda_1}\abs{x^+-x}^p\P^n_i(d\xi)<\infty.
\end{align*}
Hence,
\begin{align}\label{eq:fin_dist}
    \mathscr W^p_{gap,p}(\Pi^n_0,\Pi^n_1)\leq \int \abs{\s_1-\s'_1}^pd\Pi^n_0\otimes \Pi^n_1(\s,\s')\leq 2^{p-1}\left(\int \s_1^p \Pi^n_0(d\s)+\int \s_1^p \Pi^n_1(d\s)\right)<\infty
\end{align}
and thus $\mathscr W_p(\P^n_0,\P^n_1)<\infty$.
Let $\Q_n\in \Cpl_{s,m}(\P^n_0,\P^n_1)$ be the $\mathscr W_p-$optimal coupling from Theorem \ref{thm:convnex_entropy}, $\U_n:=\iota(\Q_n)\in \Cpl_s(\Pi_0^n,\Pi_1^n)$ and $\P^n_t:=\j^{-1}\left((T_t)_{\#}\U_n\right)$. Then by   Theorem \ref{thm:interaction_gain} 
\begin{align}\label{eq:gain_free_energy_n} 
    \mathcal F_{\beta}(\P^n_{t})&\leq (1-t)\mathcal F_{\beta}(\P^n_{0})+t\mathcal F_{\beta}(\P^n_{1})\\
    &-   \frac{(1-t)t\beta}{2}\int \sum_{(0,0)\neq (x,y)\in \bar\xi} (x-y)^2\inf_{z\in [x,y]}f(\abs{z})\Q_n^0(d\bar\xi)\nonumber\\
        &\leq (1-t)\mathcal F_{\beta}(\P_{0})+t\mathcal F_{\beta}(\P_{1})+\frac{2}{n}\nonumber\\
        &-\frac{(1-t)t\beta}{2}\int \sum_{(0,0)\neq (x,y)\in \bar\xi} (x-y)^2\inf_{z\in [x,y]}f(\abs{z})\Q_n^0(d\bar\xi)\nonumber.
\end{align}
 Noting that 
 \begin{align*}
    \sup_{n\geq 1}\mathcal F_{\beta}(\P_i^n)<\infty,
\end{align*}
\cite[Lemma 3.1]{georgii2} shows that 
\[
\sup_{n\geq 1}\int \xi(\Lambda_1)^2\P^n_i(d\xi)<\infty.
\]
Combining this with the weak convergence $\P^n_i\xrightarrow{n\to \infty}\P_i$, Lemma \ref{lem:conv-palm} yields the weak convergence $\Pi^n_i\xrightarrow{n\to \infty}\Pi_i$, for $i=0,1$.
Hence, $(\Pi^n_i)_{n\geq 1}$ is tight and therefore also $\bigcup_n\Cpl_s(\Pi^n_0,\Pi_1^n)$
such that there exists a subsequence $(\U_{n_k})_{k\geq 1}$ such that $\U_{n_k}\xrightarrow{k\to \infty}\U\in \Cpl_s(\Pi_0,\Pi_1)$ weakly. 
By continuity of the maps $T_t$, $0\leq t\leq 1$, we have 
\begin{align}\label{eq:123}
\j(\P^{n_k}_t)=(T_t)_{\#}\U_{n_k}\xrightarrow{k\to \infty}(T_t)_{\#}\U \text{ weakly}.
\end{align}
Our goal it to show that (up to a subsequence)
\[
    \P^{n_k}_t\xrightarrow{l\to \infty}\P_t:=\j^{-1}((T_t)_{\#}\U )\text{ in }\mathcal T_{\mathcal L}.
    \]
    
Note that there exists $c\in \IR$ such that for any $k\in \IN$ we have $\P^{n_k}_t\in \{\mathcal F_{\beta}\leq c\}$. Hence, by the $\mathcal T_{\mathcal L}-$compactness of this sublevel set (cf. \cite[Lemma 3.4]{georgii2}), there exists a further subsequence of $(\P^{n_k}_t)_{k\in \IN}$, which converges in $\mathcal T_{\mathcal L}$ to some $\P'\in \mathcal P_{s,1}(\Gamma)$. 
For ease of notation, we denote this subsequence also by $(\P^{n_k}_t)_{k\in \IN}$, i.e., we assume that $\P^{n_k}_t\xrightarrow{k\to \infty} \P'$ in $\mathcal T_{\mathcal L}$.
Since $\P^{n_k}_t\in \{\mathcal F_{\beta}\leq c\}$ for any $k\in \IN$, \cite[Lemma 3.1]{georgii2} shows that \[
    \sup_k\int \xi(\Lambda_1)^2 \P^{n_k}_t(d\xi)<\infty
    \] 
    and hence Lemma \ref{lem:conv-palm} shows that $\j(\P^{n_k}_t)\xrightarrow{l\to \infty}\j(\P')$ weakly.
    Combining this with \eqref{eq:123}, we obtain that $\j(\P')=(T_t)_{\#}\U$ and hence 
    \[
    \P^{n_k}_t\xrightarrow{l\to \infty}\P_t\text{ in }\mathcal T_{\mathcal L}.
    \]
    The lower semicontinuity of $\mathcal F_{\beta}$ then implies\begin{align*}
      \mathcal F_{\beta}(\P_{t})\leq \liminf_{k\to \infty}\mathcal F_{\beta}(\P^{n_k}_{t}).
    \end{align*}
     Set $\Q:=\iota^{-1}(\U)$.
    For $l\geq 1$ set $a_l=\sum_{i=1}^l\s_i$ and for $l\leq 0$ let $a_l=\sum_{i=l}^0\s_{i}$. Similarly,  for $l\geq 1$ set $a'_l=\sum_{i=1}^l\s'_i$ and for $l\leq 0$ let $a'_l=\sum_{i=l}^0\s'_{i}$.
The weak convergence $\U_{n_k}\xrightarrow{l\to \infty}\U$ and  positivity of $f$ combined with  the Portmanteau theorem implies 
\begin{align}\label{eq:gain_fatou}
    &\liminf_{k\to \infty}\int \sum_{(0,0)\neq (x,y)\in \bar\xi} (x-y)^2\inf_{z\in [x,y]}f(\abs{z})\Q_{n_k}^0(d\bar\xi)\\
    &=
    \liminf_{k\to \infty}\int \sum_{l\in \IZ}(a_l-a'_l)^2\inf_{z\in [a_l,a'_l]} f(\abs{z})d\U_{n_k}(\s,\s')\nonumber\\
    &\geq 
    \int \sum_{l\in \IZ}(a_l-a'_l)^2\inf_{z\in [a_l,a'_l]} f(\abs{z})d\U(\s,\s')\nonumber\\
  &= \int \sum_{(0,0)\neq (x,y)\in \bar\xi} (x-y)^2\inf_{z\in [x,y]}f(\abs{z})\Q^0(d\bar\xi).\nonumber
\end{align} Finally, we obtain \begin{align*} 
    \mathcal F_{\beta}(\P_{t})\leq   (1-t)\mathcal F_{\beta}(\P_{0})+t\mathcal F_{\beta}(\P_{1})-\frac{(1-t)t\beta}{2}\int \sum_{(0,0)\neq (x,y)\in \bar\xi} (x-y)^2\inf_{z\in [x,y]}f(\abs{z})\Q^0(d\bar\xi).
\end{align*}

\end{proof}

In the following we let $D(\mathcal F_{\beta}):=\{\P\in \mathcal P_{s,1}(\Gamma)\mid \mathcal F_{\beta}(\P)<\infty\}$ be the domain of the free energy. A direct consequence of the preceding convexity statement is the uniqueness of minimizers of the free energy functional.

\begin{cor}\label{cor:uniq_min}
   If  $D(\mathcal F_{\beta})\neq \emptyset$, the functional $\mathcal F_{\beta}$ has a unique minimizer in $\mathcal P_{s,1}(\Gamma)$.
\end{cor}

\begin{proof}
There exists some $\tilde\P\in \mathcal P_{s,1}(\Gamma)$ with $\mathcal F_{\beta}(\tilde\P)<\infty$.
    Let $(\P_n)_{n\geq 1}$ be a minimizing sequence. Then, we can assume that $\sup_n \mathcal F_\beta(\P_n)<\infty$. Since the sublevel sets of $\mathcal F_{\beta}$ are compact w.r.t.\ $\mathcal T_{\mathcal L}$ (cf. \cite[Lemma 3.4]{georgii2}), we can assume that $\P_n\xrightarrow{n\to \infty}\P\in \mathcal P_{s,1}(\Gamma)$ w.r.t.\ $\mathcal T_{\mathcal L}$. By the lower semicontinuity of $\mathcal F_{\beta}$ 
    \begin{align*}
        \mathcal F_{\beta}(\P)\leq 
          \liminf_{n\to \infty}\mathcal F_{\beta}(\P_n).
    \end{align*}
    Hence $\P$ is a minimizer. By Theorem \ref{thm:strictly_smaller_free_energy} the minimizer is unique.
\end{proof}

\begin{rem}
   It follows from \cite[Theorem 3]{georgii2} that for superstable, regular, non-integrably divergent potentials there is a stationary point process with finite free enery, in particular, $D(\mathcal F_{\beta})\neq \emptyset$. Examples for such  potentials are for instance the hypersingular Riesz-gases from Example \ref{example:riesz} or the Yukawa potential from Example \ref{ex:yukawa}.
\end{rem}

\subsection{$\lambda$--convexity of $\mathcal F_\beta$ and curves of maximal slope}\label{sec:lambda}

We turn to more involved consequences of Theorem \ref{thm:strictly_smaller_free_energy}. 
Our goal it to leverage Theorem \ref{thm:strictly_smaller_free_energy} to obtain weak $\lambda-$geodesic convexity of the free energy. To this end, we have to compare the quantitative gain in convexity in Theorem \ref{thm:strictly_smaller_free_energy} to the distance $\mathscr W_p(\P_0,\P_1)$. 
For this comparison to hold, we have to consider a subspace of $\mathcal P_{s,1}(\Gamma)$. Moreover, this restriction 
guarantees the finiteness of the extended metric $\mathscr W_p$. 

For the rest of this section, we let $1<p<2$.  Let  $\P_i\in  \mathcal P_{s,1}(\Gamma)$, $i=0,1$ with finite free energy $\mathcal F_{\beta}(\P_i)<\infty$. Let $\Q\in \Cpl_{s,m}(\P_0,\P_1)$ be the  coupling given by Theorem \ref{thm:strictly_smaller_free_energy} and $\U:=\iota(\Q)$. 
Then,  by Theorem \ref{thm:strictly_smaller_free_energy}, for $0\leq t\leq 1$\begin{align*}  
    \mathcal F_{\beta}(\P_{t})\leq (1-t)\mathcal F_{\beta}(\P_{0})+t\mathcal F_{\beta}(\P_{1}) -
        \frac{\beta(1-t)t}{2}\int \sum_{(0,0)\neq (x,y)\in \bar\xi} (x-y)^2\inf_{z\in [x,y]}f(\abs{z})\Q^0(d\bar\xi),
    \end{align*}
 where $\P_t:=\j^{-1}\left((T_t)_{\#}\U\right)$. 
Hölder's inequality yields
\begin{align*}
&\mathscr W_p^2(\P_0,\P_1) \leq \left[\int \abs{\s_1-\s'_1}^p\U(d\s,d\s')\right]^{2/p} \\
& = \left[\int \abs{\s_1-\s'_1}^p \left(\sup_{z\in [\s_1,\s'_1]}f^{p/(p-2)}(\abs{z})\right)^{p/2-1}\left(\sup_{z\in [\s_1,\s'_1]}f^{p/(p-2)}(\abs{z})\right)^{1-p/2} \U(d\s,d\s')\right]^{2/p} \\
&\leq \int \abs{\s_1-\s'_1}^2 \left(\sup_{z\in [\s_1,\s'_1]}f^{p/(p-2)}(\abs{z})\right)^{1-2/p} \U(d\s,d\s')\\
&\hspace{57mm}\times\left[\int \sup_{z\in [\s_1,\s'_1]}f^{p/(p-2)}(\abs{z}) \U(d\s,d\s') \right]^{2/p-1}\\
&= \left[\int \abs{\s_1-\s'_1}^2 \inf_{z\in [\s_1,\s'_1]}f(\abs{z}) \U(d\s,d\s')\right]\left[\int \sup_{z\in [\s_1,\s'_1]}f^{p/(p-2)}(\abs{z}) \U(d\s,d\s') \right]^{2/p-1}\\
&\leq \int \sum_{(0,0)\neq (x,y)\in \bar\xi} (x-y)^2\inf_{z\in [x,y]}f(\abs{z})\Q^0(d\bar\xi)\left[\int \sup_{z\in [\s_1,\s'_1]}f^{p/(p-2)}(\abs{z}) \U(d\s,d\s') \right]^{2/p-1}.
\end{align*}
Note that, if the function $f^{p/(p-2)}$ is convex,  we have 
\begin{align}\label{eq:finite_dist}
    \int \sup_{z\in [\s_1,\s'_1]}f^{p/(p-2)}(\abs{z}) \U(d\s,d\s')\leq \max\left(\int f^{p/(p-2)}(\s_1)\Pi_0(d\s),\int f^{p/(p-2)}(\s_1)\Pi_1(d\s)\right),
\end{align}
where $\Pi_i:=\j(\P_i)$ for $i=0,1$,
and hence 
\begin{align*}
     \mathcal F_{\beta}(\P_{t})&\leq (1-t)\mathcal F_{\beta}(\P_{0})+t\mathcal F_{\beta}(\P_{1})\\
     &-\frac{\beta(1-t)t}{2} \mathscr W_p^2(\P_0,\P_1) \left[ \max\left(\int f^{p/(p-2)}(\s_1)\Pi_0(d\s),\int f^{p/(p-2)}(\s_1)\Pi_1(d\s)\right)\right]^{\frac{p-2}{p}}
\end{align*}

Therefore, if we assume the RHS in \eqref{eq:finite_dist} to be finite, then we have $ \mathscr W_p(\P_0,\P_1)<\infty$ since the free energy is bounded below.
We summarize the preceding calculations in the following corollary. 

\begin{cor}\label{cor:lambda_conv}
     Let $1<p<2$  and let  $\P_i\in  \mathcal P_{s,1}(\Gamma)$, $i=0,1$ with  finite free energy $\mathcal F_{\beta}(\P_i)<\infty$, and $\Pi_i:=\j(\P_i)$ for  $i=0,1$. Assume that the function $f^{p/(p-2)}$  (cf. assumption \ref{assumption1}) is convex on $[0,\infty)$ and  $\int f^{p/(p-2)}(\s_1)\Pi_i(d\s)<\infty$ for $i=0,1$. Then $ \mathscr W_p(\P_0,\P_1)<\infty$ and 
     \begin{align*}
     \mathcal F_{\beta}(\P_{t})&\leq (1-t)\mathcal F_{\beta}(\P_{0})+t\mathcal F_{\beta}(\P_{1})\\
     &-\frac{\beta(1-t)t}{2} \mathscr W_p^2(\P_0,\P_1)\left[ \max\left(\int f^{p/(p-2)}(\s_1)\Pi_0(d\s),\int f^{p/(p-2)}(\s_1)\Pi_1(d\s)\right)\right]^{\frac{p-2}{p}},
\end{align*}
 where  $\Q\in \Cpl_{s,m}(\P_0,\P_1)$ is the $\mathscr W_p-$optimal coupling from  Theorem \ref{thm:interaction_gain}, $\U:=\iota(\Q)$ and  $\P_t:=\j^{-1}\left((T_t)_{\#}\U\right)$.
\end{cor}
\begin{proof}
    Since the free energy $\mathcal F_{\beta}$ is bounded from below, the preceding calculation shows that $\mathscr W_p(\P_0,\P_1)<\infty$. Then apply Theorem \ref{thm:interaction_gain} and repeat the calculation for the $\mathscr W_p-$optimal  coupling $\Q\in \Cpl_{s,m}(\P_0,\P_1)$ from Theorem \ref{thm:interaction_gain}.
\end{proof}

We note that the convexity assumption on $f^{p/(p-2)}$ holds for the pair potentials considered in Examples \ref{example:riesz} and \ref{ex:yukawa}. 

\begin{assumption}
    For the rest of this section, we assume that the function  $f^{p/(p-2)}$ is convex.
\end{assumption}

For $p'\geq p$ and $a,b>0$ we define the subspace 
\begin{align*}
X^{a,b}_{p,p'}:=\left\{\P\in \mathcal P_{s,1}(\Gamma)\mid \int f^{p/(p-2)}(\s_1)\Pi(d\s)\leq a, \int \s_1^{p'}\Pi(d\s)\leq b\text{ where }\Pi:=\j(\P)\right\}.
\end{align*}

The following example shows that the space $X^{a,b}_{p,p'}$ essentially controls the moments of the gaps. Hence, it is a natural replacement for the classical Wasserstein spaces $\mathcal P_p$.

\begin{ex}[Hypersingular Riesz Gas]
    This is a continuation of Example \ref{example:riesz}. In this case, the space $X^{a,b}_{p,p'}$ is given by 
    \begin{align*}
        \left\{\P\in \mathcal P_{s,1}(\Gamma)\mid \int \s_1^{\frac{p(s+2)}{(2-p)}}\Pi(d\s)\leq a\left(\frac{s(s+1)}{2}\right)^{p/(2-p)} , \int \s_1^{p'}\Pi(d\s)\leq b\right\}
    \end{align*}
For fixed  $a\in \IR$ and $1<p<2$ choose $p'=p(s+2)/(2-p)$ and $b=a\left(\frac{s(s+1)}{2}\right)^{p/(2-p)}$. Then we have 
\begin{align*}
X^{a,b}_{p,p'}=\left\{\P\in \mathcal P_{s,1}(\Gamma)\mid \int \s_1^{\frac{p(s+2)}{(2-p)}}\Pi(d\s)\leq a\left(\frac{s(s+1)}{2}\right)^{p/(2-p)}\right\}.
\end{align*}
\end{ex}

The space $X^{a,b}_{p,p'}$ gives us enough regularity to obtain weak $\lambda$ geodesic convexity of the free energy:

\begin{prop}\label{prop:lambda_conv}
    Let $1<p<2$. The space $(X^{a,b}_{p,p'},\mathscr W_p)$ is a complete geodesic metric space. Moreover,  the functional $\mathcal F_{\beta}$ is weakly $\lambda$-geodesically convex on  $(X^{a,b}_{p,p'},\mathscr W_p)$ with $\lambda=\beta a^{-\frac{2-p}{p}}$, i.e., for $\P_0,\P_1\in X^{a,b}_{p,p'}$ there exists a constant speed geodesic $(\P_t)_{0\leq t\leq 1}$ such that for any $0\leq t\leq 1$
    \begin{align*}
        \mathcal F_\beta(\P_t)\leq (1-t) \mathcal F_\beta(\P_0)+t \mathcal F_\beta(\P_1)-\lambda \frac{t(1-t)}{2} \mathscr W_p^2(\P_0,\P_1).
    \end{align*}
  
\end{prop}
\begin{proof}
    Let $\P_i\in X^{a,b}_{p,p'}$, $i=0,1$, and $\Pi_i:=\j(\P_i)$. By definition of $X^{a,b}_{p,p'}$ we have $\int \s_1^{p'}\Pi_i(d\s)<\infty$ for $i=0,1$  and thus $\mathscr W_p(\P_0,\P_1)<\infty$ by considering the product coupling.
    Let $\Q\in \Cpl_{s,m}(\P_0,\P_1)$ be the optimal coupling from Theorem  \ref{thm:interaction_gain} and 
   $(\P_t)_{0\leq t\leq 1}$ the induced geodesic. Note that the convexity of $f^{p/(p-2)}(\cdot)$ and $(\cdot)^{p'}$ ensure that $\P_t\in X^{a,b}_{p,p'}$.  Hence, $X^{a,b}_{p,p'}$ is a geodesic space. 
    Let $(\P_n)_{n{\geq 1}}$ be a Cauchy sequence. Denote by $\Pi_n\in \mathcal P_{s,1}(\SQ)$ the corresponding gap-measures. 
    By definition of the space $X^{a,b}_{p,p'}$ the measures $\Pi_n$ have  finite $p'$-th moment.
    Note that Lemma \ref{lem:cost=lim_wasserstein} yields that for every $m\in \IN$ the sequence $(\Pi_n^{1,m})_{n\geq 1}$ is Cauchy in $\mathcal  C_p^{1/p}$ (cf.\ \eqref{eq:Cp} for the definition of $\mathcal C_p$). Hence, by \cite[Theorem 6.18]{Villani}, the sequence  $(\Pi^{1,m}_n)_{n\geq 1}$ converges in $\mathcal C_p^{1/p}$. In particular, the sequence $(\Pi_n)$ converges weakly (in the product topology) to some $\Pi\in \mathcal P_{s,1}(\SQ)$. The weak convergence and the Portmanteau theorem imply 
    \begin{align*}
        \int \s_1^{p'}\Pi(d\s)\leq \liminf_{n\to \infty}\int \s_1^{p'}\Pi_n(d\s)\text{ and }\int f^{p/(p-2)}(\s_1)\Pi(d\s)\leq \liminf_{n\to \infty}\int f^{p/(p-2)}(\s_1)\Pi_n(d\s).
    \end{align*}
    Hence, $\P \in X^{a,b}_{p,p'}$, where $\P=\j^{-1}(\Pi)\in \mathcal P_{s,1}(\Gamma)$. Let $\U_{l,n}\in \Cpl_s(\Pi_{l},\Pi_{n})$ be a $\mathscr W_{gap,p}$ optimal coupling. By tightness of the marginals, we may assume that $\U_{l,n}\xrightarrow{n\to\infty}\U_l\in \Cpl_s(\Pi_l,\Pi)$ weakly. Then 
    \begin{align*}
        \int \abs{\s_1-\s_1'}^pd\U_l(\s,\s')\leq \liminf_{n\to \infty} \int \abs{\s_1-\s_1'}^pd\U_{l,n}(\s,\s')=\liminf_{n\to \infty} \mathscr W_{gap,p}^p(\Pi_l,\Pi_n).
    \end{align*}
    Note that  the RHS  in the above line converges to $0$ as $l\to \infty$ since the sequence $(\Pi_n)_{n\geq 1}$ is Cauchy in $\mathscr W_{gap,p}$.
    Hence, $\Pi_l\xrightarrow{l\to \infty}\Pi$ in $\mathscr W_{gap,p}$ and  $(X^{a,b}_{p,p'},\mathscr W_p)$ is complete.
    Finally, we showed at the beginning of the proof that for  $\P_i\in  \mathcal P_{s,1}(\Gamma)$, $i=0,1$ with  finite free energy $\mathcal F_{\beta}(\P_i)<\infty$, the geodesic from Corollary \ref{cor:lambda_conv} lies in the space $X^{a,b}_{p,p'}$. This shows the weak $\lambda-$geodesic convexity of $\mathcal F_{\beta}$.
\end{proof}

We want to apply the theory for curves of maximal slopes in metric spaces as outlined in \cite[Chapter 2]{AGS08}. To that end, we have to establish certain  topological properties of the metric $\mathscr W_p$ and of the functional $\mathcal F_\beta$ on the space $X^{a,b}_{p,p'}$.
 Our calculations mainly take place in the space of sequences $\SQ$.
We start with two technical lemmas. The first lemma shows that convergence in $\mathscr W_p$ implies weak convergence. The second lemma shows that the free energy $\mathcal F_\beta$ is lower semicontinuous w.r.t.\ the metric $\mathscr W_p$ on $X^{a,b}_{p,p'}$. As a consequence, we obtain uniqueness of minimizers of $\mathcal F_\beta$ in the space $X^{a,b}_{p,p'}$.

\begin{lem}\label{lem:dist->weak}
    Let $\mathscr W_p(\P_n,\P)\xrightarrow{n\to \infty}0$ for $\P_n,\P\in X^{a,b}_{p,p'}$. Then $\P_n\xrightarrow{n\to \infty}\P$ weakly.
\end{lem}
\begin{proof}
      Let $\Pi_n:=\j(\P_n)$. 
    By definition of the space $X^{a,b}_{p,p'}$ the marginals of $\Pi_n$ have  finite $p'$-th moment.
    Note that Lemma \ref{lem:cost=lim_wasserstein} yields that for every $m\in \IN$ the sequence $(\Pi_n^{1,m})_{n\geq 1}$ is Cauchy in $\mathcal  C_p^{1/p}$. Hence, by \cite[Theorem 6.18]{Villani}, the sequence  $(\Pi^{1,m}_n)_{n\geq 1}$ converges in $\mathcal C_p^{1/p}$. In particular, the sequence $(\Pi_n)_{n\geq 1}$ converges weakly (in the product topology) to some $\Pi\in \mathcal P_{s,1}(\SQ)$. The weak convergence and the Portmanteau theorem imply 
    \begin{align*}
        \int \s_1^{p'}\Pi(d\s)\leq \liminf_{n\to \infty}\int \s_1^{p'}\Pi_n(d\s)\text{ and }\int f^{p/(p-2)}(\s_1)\Pi(d\s)\leq \liminf_{n\to \infty}\int f^{p/(p-2)}(\s_1)\Pi_n(d\s).
    \end{align*}
    Hence, $\P \in X^{a,b}_{p,p'}$, where $\P:=\j^{-1}(\Pi)\in \mathcal P_{s,1}(\Gamma)$. Lemma \ref{lem:conv-palm} shows that $\P_n\xrightarrow{n\to \infty}\P$ weakly.
\end{proof}

\begin{lem}\label{lem:lsc+min}
    The functional $\mathcal F_{\beta}$ is lower semicontinuous w.r.t.\ $\mathscr W_p$ on $X^{a,b}_{p,p'}$. Moreover, if $D(\mathcal F_{\beta})\cap X^{a,b}_{p,p'}\neq \emptyset$, then $\mathcal F_{\beta}$ has a unique minimizer in $X^{a,b}_{p,p'}$.
\end{lem}
\begin{proof}
    Let $X^{a,b}_{p,p'}\ni\P_n\xrightarrow{n\to \infty} \P\in X^{a,b}_{p,p'}$ w.r.t.\ $\mathscr W_p$. Then by Lemma \ref{lem:dist->weak} we have that $\P_n\xrightarrow{n\to \infty}\P$ weakly.
We can assume that $\liminf \mathcal F_\beta (\P_n)<\infty$, because otherwise we are done.     
Since the sublevel sets of $\mathcal F_{\beta}$ are compact w.r.t.\ $\mathcal T_{\mathcal L}$ (cf. \cite[Lemma 3.4]{georgii2}), we actually have that $\P_n\xrightarrow{n\to \infty}\P$ in $\mathcal T_{\mathcal L}$ (maybe along a subsequence). The lower semicontinuity of $\mathcal F_{\beta}$ w.r.t.\ $\mathcal T_{\mathcal L}$    then shows 
\begin{align*}
   \mathcal F_{\beta}(\P)\leq \liminf_{n\to \infty} \mathcal F_{\beta}(\P_n). 
\end{align*}
In order to show existence of a unique minimizer, let $\P_n\in X^{a,b}_{p,p'}$ with $\mathcal F_{\beta}(\P_n)\xrightarrow{n\to \infty} \inf_{\P\in X^{a,b}_{p,p'}}\mathcal F_{\beta}(\P)$.
By compactness of sublevel sets, there exists $\P^*\in \mathcal P_{s,1}(\Gamma)$ such that $\P_n\xrightarrow{n\to \infty} \P^*$ w.r.t.\ $\mathcal T_{\mathcal L}$.
 Moreover, \cite[Lemma 3.7]{erbar2023optimal} shows that 
 \[
    \sup_n\int \xi(\Lambda_1)\log(\xi(\Lambda_1))\P_n(d\xi)<\infty
    \]
Hence, by Lemma \ref{lem:conv-palm} we have $\Pi_n \xrightarrow{n\to \infty} \Pi^*$ weakly, where $\Pi_n:=\j(\P_n)$ and $\Pi^*:=\j(\P^*)$. The weak convergence and the Portmanteau theorem show that $\P^*\in X^{a,b}_{p,p'}$. The lower semicontinuity of $\mathcal F_{\beta}$ w.r.t.\ $\mathcal T_{\mathcal L}$ shows
\begin{align*}
    \mathcal F_{\beta}(\P^*)\leq \liminf_{n\to \infty}\mathcal F_{\beta}(\P_n)=  \inf_{\P\in X^{a,b}_{p,p'}}\mathcal F_{\beta}(\P).
\end{align*}
Hence, $\P^*$ is a minimizer. Uniqueness follows from the weak $\lambda-$geodesic convexity of $\mathcal F_\beta$, cf.\ Proposition \ref{prop:lambda_conv}.
\end{proof}

Finally, we are in the position to prove the existence of curves of maximal slope in Theorem \ref{thm:exist max slope}. 
To be able to state and prove this result, we have to recall some definitions from analysis in metric spaces. Let $(S,d)$ be a metric space and  $v\in AC(a,b;S)$ be an absolutely continuous curve, i.e., there exists an integrable function $m:(a,b)\to \IR$ such that 
\[
d(v(s),v(t))\leq \int_s^t m(r)dr,\quad \forall a< s\leq t<b.
\]
The metric derivative of $v$ is defined by \begin{align*}
    \abs{v'(t)}:=\lim_{s\to t}\frac{d(v(s),v(t))}{\abs{s-t}},
\end{align*}
which exists for $\Leb_{\IR}-$a.e. $t\in (a,b)$ (cf.\ \cite[Theorem 1.1.2]{AGS08}).
Next, we present the generalization of gradients for functions defined on metric spaces.

\begin{defi}\cite[Definition 1.2.1]{AGS08}
    Let $(S,d)$ be a metric space and $F:S\to (-\infty,\infty]$ measurable. A function $g:S\to [0,\infty]$ is a strong upper gradient for $F$ if for every absolutely continuous curve $v\in AC(a,b;S)$ the function $g\circ v$ is Borel and \begin{align*}
        \abs{F(v(t))-F(v(s))}\leq \int_s^tg(v(r))\abs{v'}(r)dr \quad \forall a< s\leq t< b.
    \end{align*}
    In particular, if $g\circ v \abs{v'}\in L^1(a,b)$ then $F\circ v$ is absolutely continuous and \begin{align*}
        \abs{(F\circ v)'(t)}\leq g(v(t))\abs{v'(t)}\quad \text{ for a.e. } t\in (a,b).
    \end{align*}
\end{defi}

The notion of curves of maximal slope, like the EVI in Section \ref{sec:evihwi}, is a generalization of gradient flows to the setting of metric spaces. 

\begin{defi}\label{def:curve of maximal slope}\cite[Def 1.3.2]{AGS08}
    Let $(S,d)$ be a metric space and $F:S\to (-\infty,\infty]$ measurable. A locally absolutely continuous curve $v:(a,b)\to S$ is a curve of maximal slope for the functional $F$ with respect to its strong upper gradient $g$, if $F\circ v$ is $\Leb_{\IR}-$a.e. equal to  a non-increasing map $f$ and \begin{align*}
        f'(t)\leq -\frac{1}{2}\abs{v'}^2(t)-\frac{1}{2}g^2(v(t)) \text{ in } (0,\infty).
    \end{align*} 
\end{defi}

Under suitable assumptions,
a strong upper gradient of a functional $F$ is given by its
 local slope (cf. \cite[Corollary 2.4.10]{AGS08}).

\begin{defi}\cite[Def 1.2.4]{AGS08}
    Let $(S,d)$ be a metric space and $F:S\to (-\infty,\infty]$ measurable. Let $D(F):=\{s\in S\mid F(s)<\infty\}$ be the domain of $F$. The local slope $ \abs{\partial  F}$ of $F$ is defined by \begin{align*}
        \abs{\partial  F}(s):=\limsup_{s'\to s}\frac{(F(s)-F(s'))_+}{d(s,s')},\quad \forall s\in D(F).
    \end{align*}
\end{defi}

Finally, we are in a position to prove the existence of curves of maximal slope for the free energy functional $\mathcal F_\beta$. The proof consists in verifying the assumptions of \cite[Corollary 2.4.12]{AGS08}.
 
\begin{thm}\label{thm:exist max slope}
Consider the metric space $( X^{a,b}_{p,p'},\mathscr W_p)$.
Every $\P\in X^{a,b}_{p,p'}$ with  $\mathcal F_{\beta}(\P)<\infty$    is the starting point of a curve of maximal slope
for $\mathcal F_{\beta}$ with respect to the local slope $\abs{\partial \mathcal F_{\beta}}$, given by  \[
    \abs{\partial \mathcal F_{\beta}}(\P):=\limsup_{X^{a,b}_{p,p'}\ni \P'\to\P}\frac{(\mathcal F_{\beta}(\P)-\mathcal F_{\beta}(\P'))^+}{\mathscr W_p(\P',\P)}.
    \] Moreover, such a curve $(\P_t)_{t>0}$ satisfies the energy identity\begin{align}
    \frac{1}{2}\int_0^T\abs{\P'}^2(t)+\frac{1}{2}\int_0^T\abs{\partial \mathcal F_{\beta}}^2(\P(t))dt+\mathcal F_{\beta}(\P(T))=\mathcal F_{\beta}(\P(0)),\quad T>0.
\end{align}
\end{thm}
\begin{proof}
    We want to apply \cite[Corollary 2.4.12]{AGS08} in the setting $\phi=\mathcal F_{\beta}$, $(S,d)=(X^{a,b}_{p,p'},\mathscr W_p)$, $\lambda=\beta a^{-\frac{2-p}{p}}$ and $\sigma$ equal to the weak topology. Since $\mathcal F_{\beta}$ is weakly $\lambda-$geodesically convex (cf. Proposition \ref{prop:lambda_conv}), we have to prove the following four statements. 
    \begin{enumerate}
    \item \label{item:0} Let $c\in \IR$. Every sequence $(\P_n)_{n\in \IN}\subset  X^{a,b}_{p,p'}\cap \{\mathcal F_{\beta}\leq c\}$ admits a convergent subsequence (in $X^{a,b}_{p,p'}$) w.r.t.\ the weak topology.
        \item \label{item:1} Let $\P^n_i\xrightarrow{n\to \infty}\P_i$ weakly in $X^{a,b}_{p,p'}$, $i=0,1$, then 
        \begin{align*}
        \mathscr W^p_p(\P_0,\P_1)\leq \liminf_{n\to \infty}
        \mathscr W^p_p(\P^n_0,\P^n_1).
    \end{align*}
\item \label{item:2} $\mathcal F_{\beta}$ is lower semicontinuous with respect to the weak topology on $X^{a,b}_{p,p'}$.
    \item \label{item:3} The map \begin{align*}
         X^{a,b}_{p,p'}\ni\P\mapsto \abs{\partial \mathcal F_{\beta}}(\P)
    \end{align*}
is lower semicontinuous with respect to the weak topology on $X^{a,b}_{p,p'}\cap \{\mathcal F_{\beta}\leq c\}.$
\end{enumerate}

{ \eqref{item:0}} For $n\in \IN$ let $\P_n\in \{\P\in X^{a,b}_{p,p'}\mid \mathcal F_{\beta}(\P)\leq c\}$. By compactness of the sublevel sets of the specific relative entropy (cf. \cite[Proposition 6.8]{RAS}),  there exists a subsequence $(\P_{n_k})_{k\geq 1}$ with $\P_{n_k}\xrightarrow{k\to \infty}\P\in \mathcal P_{s,1}(\Gamma)$ weakly. 
 Moreover, \cite[Lemma 3.7]{erbar2023optimal} shows that \[
    \sup_n\int \xi(\Lambda_1)\log(\xi(\Lambda_1))\P_n(d\xi)<\infty.
    \]
Applying Lemma \ref{lem:conv-palm} and the Portmanteau Theorem shows that $\P\in X^{a,b}_{p,p'}$.

{ \eqref{item:1}} Let $\P^n_i\xrightarrow{n\to \infty}\P_i$ weakly in $\mathcal P_{s,1}(\Gamma)$, $i=0,1$,  let $\Pi^n_i:=\j(\P^n_i)\in \mathcal P_{s,1}(\SQ)$ be the corresponding gap distributions and $\U^n\in \Cpl_s(\Pi^n_0,\Pi_1^n)$ a $\mathscr W_{gap,p}-$optimal coupling. The uniform bound \begin{align}\label{eq:unif}
       \sup_{n\geq 1} \int \s_1^{p'}\Pi^n_i(d\s)<\infty
    \end{align}
shows that for $m\in \IN$ the sequences $((\Pi^n_i)^{-m,m})_{n\geq 1}$ are tight (cf. \cite[Remark 5.1.5]{AGS08}). A diagonalization argument shows that the two sequences $(\Pi^n_i)_{n\geq 1}$, $i=0,1$, are tight with respect to the infinite product topology on $\SQ$. Hence, there exists  a subsequence $(\U_{n_k})_{k\geq 1}$ such that $\U_{n_k}\xrightarrow{k\to \infty}\U\in \mathcal P_{s,1}(\SQ^2)$ weakly.
In particular, $\Pi_i^{n_k}\xrightarrow{k\to \infty}\Pi_i'$ weakly, where $\Pi'_i$, $i=0,1$, are the marginals of $\U\in \Cpl_s(\Pi_0',\Pi_1')$.
Hence, by the Portmanteau theorem\begin{align}\label{eq:1234}
    \int (\s_1-\s'_1)^p \U(d\s,d\s')\leq \liminf_{k\to \infty } \int (\s_1-\s'_1)^p \U_{n_k}(d\s,d\s')=\liminf_{k\to \infty }\mathscr W^p_{gap,p}(\Pi_0^{n_k},\Pi^{n_k}_1).
\end{align}
     If we can show that $\Pi_i=\Pi_i'$, the claim follows from inequality \eqref{eq:1234}. But the bound \eqref{eq:unif} combined with Lemma \ref{lem:conv-palm} shows that $\P^{n_k}_i\xrightarrow{k\to \infty}\P_i'$ weakly, where $\P_i':=\j^{-1}(\Pi_i')$. Hence, $\P_i=\P'_i$ and thus $\Pi_i=\Pi_i'$.

{\eqref{item:2}} Let $\P_n\xrightarrow{n\to \infty} \P$ weakly in $X^{a,b}_{p,p'}$.  Assume that $\liminf_{n\to \infty}\mathcal F_{\beta}(\P_n)<\infty$. Otherwise there is nothing to show. Let $(\P_{n_k})_{k\geq 1}$ be a subsequence attaining the limit inferior, i.e.\ let \begin{align*}
    \liminf_{n\to \infty}\mathcal F_{\beta}(\P_n)=\lim_{k\to \infty}\mathcal F_{\beta}(\P_{n_k}).
\end{align*}
We can thus assume that for some $c\in \IR$ we have $(\P_{n_k})\subset \{\P'\in\mathcal P_{s,1}(\Gamma)\mid \mathcal F_{\beta}(\P')\leq c\}$. This sublevel set is compact with respect to the topology $\mathcal T_{\mathcal L}$ (cf. \cite[Lemma 3.4]{georgii2}). Hence, we have $\P_{n_k}\xrightarrow{k\to \infty}\P$ in $\mathcal T_{\mathcal L}$. 
The claim now follows from the lower  semicontinuity of $\mathcal F_{\beta}$ with respect to $\mathcal T_{\mathcal L}$ (proven in \cite[Lemma 3.4]{georgii2}).

{\eqref{item:3}} We have to prove the  lower semicontinuity of the map \begin{align*}
         X^{a,b}_{p,p'}\ni\P\mapsto \abs{\partial \mathcal F_{\beta}}(\P)
    \end{align*}
on sublevel sets of $\mathcal F_{\beta}$    with respect to the weak topology. By \cite[Theorem 2.4.9]{AGS08} the local slope admits the following representation \begin{align*}
        \abs{\partial F_{\beta}}(\P)=\sup_{X^{a,b}_{p,p'} \ni\P'\neq \P}\left(\frac{\mathcal F_{\beta}(\P)-\mathcal F_{\beta}(\P')}{\mathscr W_p(\P,\P')}\right)^+,\quad \forall \P\in D(\mathcal F_{\beta})\cap X^{a,b}_{p,p'}.
    \end{align*}
    Lemma \ref{lem:cost=lim_wasserstein} yields the following representation \begin{align*}
        \abs{\partial F_{\beta}}(\P)=\sup_{X^{a,b}_{p,p'} \ni \P'\neq \P}\sup_{m\in \IN}m^{1/p}\left(\frac{\mathcal F_{\beta}(\P)-\mathcal F_{\beta}(\P')}{\mathcal C^{1/p}_p(\Pi^{1,m},\Pi'^{1,m})}\right)^+,\quad \forall \P\in D(\mathcal F_{\beta})\cap X^{a,b}_{p,p'}.
    \end{align*}
    Fix $\P'\in X^{a,b}_{p,p'}$, $m\in \IN$ and let $\P_n\in X^{a,b}_{p,p'}\cap  D(\mathcal F_{\beta})$ with  $\P_n\xrightarrow{n\to \infty}\P\in X^{a,b}_{p,p'}\cap D(\mathcal F_{\beta})$ weakly. Then \eqref{item:2} shows that \begin{align*}
    \mathcal F_{\beta}(\P)\leq \liminf_{n\to \infty} \mathcal F_{\beta}(\P_n).
\end{align*}   
     Moreover, \cite[Lemma 3.7]{erbar2023optimal} shows that \[
    \sup_n\int \xi(\Lambda_1)\log(\xi(\Lambda_1))\P_n(d\xi)<\infty
    \]
     and we can apply Lemma \ref{lem:conv-palm} to obtain  that $\Pi_n\xrightarrow{n\to \infty}\Pi$ weakly.
Moreover, by definition of $X^{a,b}_{p,p'}$ \begin{align*}
        \sup_{n\in \IN}\int_{\IR^m} \norm{\s}^{p'}d\Pi^{1,m}_n=\sup_{n\in \IN}\int_{\IR^m} \left(\sum_{i=1}^m\s_i^2\right)^{p'/2}d\Pi^{1,m}_n<\infty.
    \end{align*}
    Combining this uniform integrability with the weak convergence $\Pi_n\xrightarrow{n\to \infty}\Pi$ we obtain that $\Pi_n^{1,m}\xrightarrow{n\to \infty}\Pi^{1,m}$ in $\mathcal C^{1/p}_p$ (cf. \cite[Theorem 6.9]{Villani}). Hence, the map \begin{align*}
    X^{a,b}_{p,p'}\cap D(\mathcal F_{\beta})\ni \P\mapsto\left(\frac{\mathcal F_{\beta}(\P)-\mathcal F_{\beta}(\P')}{\mathcal C^{1/p}_p(\Pi^{1,m},\Pi'^{1,m})}\right)^+
    \end{align*}
    is lower semicontinuous with respect to the weak topology on sublevel sets of $\mathcal F_{\beta}$.
    Finally, $\abs{\partial\mathcal F_{\beta}}$ is lower semicontinuous on $ X^{a,b}_{p,p'}\cap D(\mathcal F_{\beta})$ with respect to the weak topology, since it is the supremum of lower semicontinuous maps.
\end{proof}

\begin{rem}
    We do not know whether there is a stochastic representation for the curve of maximal slope $t\mapsto\P_t$, mostly because we work in the setup of $1<p<2.$ If we could show the same result for $p=2$ we expect to see an interacting gap process which one then would have to translate into a particle process on the level of points where the Markov property might be challenging. (For the non-interacting case see Section \ref{sec:evihwi}.)
\end{rem}

By applying \cite[Lemma 2.4.13 and Theorem 2.4.14]{AGS08}, we obtain 
as immediate consequences an inequality bounding the distance of a point process to the minimizer of $\mathcal F_\beta$ from above by the difference of the free energies (equation \eqref{eq:uhb}) and an exponential  convergence of curves of maximal slope to the minimizer (equation \eqref{eq:edc}).

\begin{cor}\label{cor:consequences}
   Assume that $D(\mathcal F_{\beta})\cap X^{a,b}_{p,p'}\neq \emptyset$ and  let $\P^*\in X^{a,b}_{p,p'}$ be the unique minimizer of $\mathcal F_{\beta}$ on $X^{a,b}_{p,p'}$. Then for any $\P\in X^{a,b}_{p,p'}\cap D(\mathcal F_{\beta})$
   \begin{align}\label{eq:uhb}
       \frac{\beta }{2a^{\frac{2-p}{p}}} \mathscr W_p^2(\P,\P^*)\leq \mathcal F_{\beta}(\P)-\mathcal F_{\beta}(\P^*)\leq \frac{a^{\frac{2-p}{p}}}{2\beta } \abs{\partial \mathcal F_{\beta}}^2(\P).
    \end{align}
    Moreover, for any curve of maximal slope $(\P_t)_{t>0}$ w.r.t. $\abs{\partial \mathcal F_{\beta}}$ in $X^{a,b}_{p,p'}$ satisfies for every $t\geq t_0>0$
    \begin{align}\label{eq:edc}
        \frac{\beta }{2a^{\frac{2-p}{p}}}\mathscr W_p^2(\P_t,\P^*)\leq \mathcal F_{\beta}(\P_t)-\mathcal F_{\beta}(\P^*)\leq (\mathcal F_{\beta}(\P_{t_0})-\mathcal F_{\beta}(\P^*))e^{-2\beta a^{\frac{p-2}{p}}(t-t_0)}.
    \end{align}
\end{cor}
\begin{rem}
    The first inequality in \eqref{eq:uhb} can be interpreted as a Talagrand type inequality, the second as a log-Sobolev inequality, cf.\ e.g.\ \cite[Section 4.4]{figalli2021invitation}.
\end{rem}

We can leverage the inequality \eqref{eq:uhb} to show that the map $\beta\mapsto \mbox{argmin}_\P\mathcal F_\beta(\P)$ is locally $\frac12$-Hölder.

\begin{cor}\label{cor:temper}
Assume that $D(\mathcal F_{\beta})\neq \emptyset$.
    For $0<\beta<\beta'$ let $\P_{\beta}$ and $\P_{\beta'}$ be the unique minimizers of $\mathcal F_{\beta}$ and $\mathcal F_{\beta'}$ on $\mathcal P_{s,1}(\Gamma)$ respectively. Set $\Pi_\beta:=\j(\P_\beta)$ and $\Pi_{\beta'}:=\j(\P_{\beta'})$ and assume that there exists $p'\geq p$ such that $\int \s_0^{p'}\Pi_\beta(d\s),\int \s_0^{p'}\Pi_{\beta'}(d\s)<\infty$. Then \begin{align}
       \mathscr W_p^2(\P_{\beta},\P_{\beta'})\leq \frac{2(\beta'-\beta)}{\beta}\max\left(\int f(\s_0)^{\frac{p}{p-2}}\Pi_\beta(d\s) ,\int f(\s_0)^{\frac{p}{p-2}}\Pi_{\beta'}(d\s) \right)^{\frac{2-p}{p}}\mathcal W^{int}(\P_{\beta})
    \end{align}
\end{cor}
\begin{proof}
    By assumption there exists $b\in \IR$ such that $\P_{\beta},\P_{\beta'}\in X^{a,b}_{p,p'}$ for 
    \begin{align*}
        a:=\max\left(\int f(\s_0)^{p/(p-2)}\Pi_\beta(d\s) ,\int f(\s_0)^{p/(p-2)}\Pi_{\beta'}(d\s) \right).
    \end{align*}
    Hence, by inequality \eqref{eq:uhb}  
    \begin{align*}
        \frac{\beta}{2a^{\frac{2-p}{p}}} \mathscr W_p^2(\P_\beta,\P_{\beta'})\leq \mathcal F_{\beta}(\P_{\beta'})-\mathcal F_{\beta}(\P_\beta).
    \end{align*}
    The inequality 
    \begin{align*}
        \mathcal F_{\beta}(\P_{\beta'})\leq \mathcal F_{\beta'}(\P_{\beta'})\leq \mathcal F_{\beta'}(\P_{\beta})
    \end{align*}
    yields the claim, since 
    \begin{align*}
        \mathcal F_{\beta'}(\P_{\beta})-\mathcal F_{\beta}(\P_\beta)=(\beta'-\beta)\mathcal W^{int}(\P_{\beta}).
    \end{align*}
\end{proof}

\section{Long-range interactions and convexity of the free energy}\label{sec:longrange}

The goal of this section is to show that the first part of Section \ref{sec:Interaction_energy} carries over to long-range interactions. However, at the moment, we cannot show the weak $\lambda-$geodesic convexity, see Remark \ref{rem:why}.
We restrict to the case $p>1$ and consider (long-range) Riesz and logarithmic interactions given by 
\begin{equation*} \phi(x) = \begin{cases}
                    \abs{x}^{-s} & \text{ if }  0<s<1  \\
                     -\log\abs{x} & \text{ if } s=0.
                 \end{cases} 
                 \end{equation*}
For these potentials the interaction energy $\mathcal W^{int}$, defined in \eqref{def:interac}, in general is infinite. Hence, the results from the previous section do not apply. Instead, the interaction energy has to be replaced by the electric energy $\mathcal W^{elec}$ which is defined  in Definition \ref{def:elec_energy} below.
Nevertheless, the electric energy $\mathcal W^{elec}$ can be approximated by a renormalized interaction energy (cf.\eqref{eq:wsx}).
Hence, our ansatz to obtain convexity of the free energy along curves of stationary point processes is the same as in the previous section, i.e., combining the weak geodesic convexity of the specific relative entropy Theorem \ref{thm:convnex_entropy} and the convexity property of the finite box energies Lemma \ref{lem:convex_fix_n}. 

We note that for $s>0$ the potential $\phi$ is superstable and hence equation \eqref{eq:formula_hamiltonian} holds. For $s=0$ equation \eqref{eq:formula_hamiltonian} holds by Lemma \ref{lem:superstable_log}.
If $0<s<1$ the calculation  in Example \ref{example:riesz} shows that $\phi$ satisfies assumption \ref{assumption0} for the function 
$f(x)=\frac{s(s+1)}{2}\abs{x}^{-s-2}$.
If $s=0$ it can be checked that $\phi$ satisfies assumption \ref{assumption0} for the function $f(x)=\frac{1}{2}x^{-2}$. We summarize this observation.

\begin{lem}
    The potential $\phi$ satisfies equation \eqref{eq:formula_hamiltonian} for any $n\in \IN$ and  $\P\in \mathcal P_{s,1}(\Gamma)$ with $\int\xi(\Lambda_1)^2\P(d\xi)<\infty$. Moreover, $\phi$ satisfies assumption \ref{assumption0} for the function
    \[ f(x) = \begin{cases}
                    \frac{s(s+1)}{2}\abs{x}^{-s-2} & \text{ if }  0<s<1  \\
                     \frac{1}{2}x^{-2} & \text{ if } s=0
                 \end{cases}. \]
                 In particular, Lemma \ref{lem:convex_fix_n} is applicable in all cases.
\end{lem}

We now turn to the definition of the free energy functional.
For $n\in \IN$ we define the background energy by \begin{align}
    B_n:=\frac{1}{2}\int_{-n/2}^{n/2} \int_{-n/2}^{n/2}\eins_{x\neq y} \phi(x-y)dxdy. 
\end{align}
The modified interaction energy $ \mathcal W^{mod}(\P)$,  $\P\in \mathcal P_{s,1}(\Gamma)$, was introduced in \cite{Lebl__2016}. It is defined in a similar way to the interaction energy considered in Section \ref{sec:Interaction_energy}. However, the background energy has to be subtracted in order to renormalize the energy.  For  $\P\in \mathcal P_{s,1}(\Gamma)$ we set\begin{align}\label{eq:wsx}
    \mathcal W^{mod}(\P):=\liminf_{n\to \infty}n^{-1}\left(\int H_n d\P-B_n\right)=\liminf_{n\to \infty}n^{-1}\left(\mathcal H(\P,n)-B_n\right)
\end{align}
We define the free energy functional as usual by $\mathcal F^{mod}(\P):=\beta\mathcal W^{mod}(\P)+\mathcal E(\P)$, $\beta>0$.

\medskip

We now recall the definition of the electric energy $\mathcal W^{elec}$ of a point process $\P\in \mathcal P_{s,1}(\Gamma)$ which was introduced in \cite{sandier2012ginzburg}. It will replace the interaction energy $\wint$.
We let 
\begin{eqnarray*}
g(x,y):=\begin{cases}
\norm{(x,y)}^{-s} \text{ for  } (x,y)\in \IR^2 & s>0 \\
-\log|x| \text{ for } x\in\IR & s=0.
\end{cases}
\end{eqnarray*}
Then (see \cite[Section 1.2]{petrache2017next})
\begin{align*}
    -\div\left(\abs{y}^{s}\nabla g(x,y)\right)=c_s \delta_0\quad \forall (x,y)\in \IR^2
\end{align*}
where $c_s\in \IR$ is a constant.
\begin{defi}\label{def:elec_energy}[Electric energy]
Let $I\subset \IR$ be an interval.
\begin{enumerate}
    \item We call electric fields on $I$ the set of all vector fields $E$ in $L^p_{loc}(I\times \IR,\IR\times \IR)$, for some $1<p<2/(s+1)$ fixed
\item Let $\xi$ be a finite point configuration in $I$ and $E$ an electric field on $I$. We say that $E$ is compatible with $\xi$ in $I$ provided 
\begin{align}\label{eq:hypersurf}
    -\div\left(\abs{y}^sE\right)=c_s\left(\xi-\delta_{\IR}\right) \text{ in }I\times \IR,
\end{align}
in the sense of distributions. Here  $\delta_{\IR}$ is the Radon measure on $\IR\times \IR$ defined by \begin{align*}
    \int_{\IR^2} f(x,y) d\delta_{\IR}(x,y)=\int_{\IR} f(x,0)dx,\quad f\in C_c(\IR^2,\IR).
\end{align*}
Similarly, in \eqref{eq:hypersurf} the point configuration $\xi$ is interpreted as a measure on $\IR^2$ via  
\begin{align*}
    \int_{\IR^2} f(x,y) d\xi(x,y)=\int_{\IR} f(x,0)d\xi(x)\quad f\in C_c(\IR^2,\IR).
\end{align*}
\item If $E$ is compatible with $\xi$ in $I$, for $\eta\in (0,1)$  we define the $\eta$-truncation of the electric field $E$ as
\begin{align*}
    E_{\eta}(x,y):=E(x,y)-\sum_{p\in \xi\cap I}\nabla f_{\eta}((x,y)-(p,0)),
\end{align*}
    where $f_{\eta}$ is the function
    \begin{align*}
        f_{\eta}(x,y):=\left(g(x,y)-g(\eta)\right)_+
    \end{align*}
     and $g(\eta)$ is defined by $g(\eta)=g(\eta,0)$.
\item Let $\xi$ be a point configuration on $\IR$. We define the global electric energy of $\xi$ as
\begin{align*}
     \tilde{\mathcal W}^{elec}(\xi):=\inf_E \left(
    \lim_{\eta\to 0}\left( 
    \limsup_{n\to \infty}\frac{1}{\abs{\Lambda_n}}\frac{1}{c_s}\int_{\Lambda_n\times \IR}\abs{y}^{s}\norm{E_{\eta}}^2-g(\eta)
    \right)
    \right),
\end{align*}
where the $\inf$ is taken over electric fields $E$ that are compatible with $\xi$ in $\IR$.
\item For $\P\in \mathcal P_{s,1}(\Gamma)$ we set \begin{align*}
    \mathcal W^{elec}(\P):=\int \tilde{\mathcal W}^{elec}(\xi)\P(d\xi).
\end{align*}

\end{enumerate}
    
\end{defi}

We are interested in the free energy functional $\mathcal F_{\beta}^{elec}(\P):=\beta\mathcal W^{elec}(\P)+\mathcal E(\P)$, $\P\in \mathcal P_{s,1}(\Gamma)$. Our goal is to prove the following convexity statement, whose content and proof is very similar to those of  Theorems  \ref{thm:interaction_gain} and \ref{thm:strictly_smaller_free_energy}.

\begin{thm}\label{Theorem 1}
    Let $\P_i\in \mathcal P_{s,1}(\Gamma)$ with $\mathcal F^{elec}_{\beta}(\P_i)<\infty$, $i=0,1$. Then there exists a coupling $\Q\in \Cpl_{s,m}(\P_0,\P_1)$ such that for $0\leq t\leq 1$ \begin{align}
        \mathcal F^{elec}_{\beta}(\P_{t})&\leq  (1-t) \mathcal F^{elec}_{\beta}(\P_{0})+t\mathcal F^{elec}_{\beta}(\P_{1})\\
        &-\frac{(\eins_0(s)+s(s+1))(1-t)t\beta}{4}\int \sum_{(0,0)\neq (x,y)\in \bar\xi} (x-y)^2\inf_{z\in [x,y]}\abs{z}^{-s-2}\Q^0(d\bar\xi),\nonumber
    \end{align}
    where  $\U:=\iota(\Q)$ and  $\P_t:=\j^{-1}\left((T_t)_{\#}\U\right)$.
\end{thm}

\begin{rem}\label{rem:why}
    We do not know whether the coupling in Theorem \ref{Theorem 1} can be chosen to be optimal, even if we assume finite distance of the point processes. This is the reason why we cannot say anything on $\lambda-$geodesic convexity as in Section \ref{sec:lambda}
\end{rem}
 Similar to the proof of Theorem \ref{thm:strictly_smaller_free_energy} we make use of   approximating point processes. For these approximations the electric energy $\mathcal W^{elec}$ is well approximated by the modified interaction energies $\mathcal W^{mod}$ that fit better to our approach. This is made precise in the following proposition, see \cite[Proposition 5.3]{Lebl__2016} and its   proof.
 
\begin{prop}\label{prop:approx_hyper}
    For $\P\in \mathcal P_{s,1}(\Gamma)$ with $\mathcal W^{elec}(\P)<\infty$ there exists  $(\P^N)_{N\geq 1}\in \mathcal P_{s,1}(\Gamma)$ with
    $\P^N\xrightarrow{N\to \infty}\P$ weakly and 
    \begin{align*}
        \mathcal E(\P^N)\xrightarrow{N\to \infty}\mathcal E(\P)
    \end{align*}and \begin{align*}
       \limsup_{N\to \infty}\limsup_{n\to \infty} (nN)^{-1}\left(\mathcal H(\P^N,nN)-B_{nN}\right)\leq \mathcal W^{elec}(\P). 
    \end{align*}
    Moreover, $\mathcal W^{mod}(\P^N)\to \mathcal W^{elec}(\P)$ as $N\to \infty$.
    For $\xi \in \supp(\P^N)$ and $x\in \xi$ we have
    \begin{align}\label{eq:hyperuni}
    \inf\{y\in \supp(\xi):y>x\}-x\leq 2N.
    \end{align}
    Finally, it holds that $\int \xi(\Lambda_1)^2\P^N(d\xi)<\infty$.
\end{prop}

The proof of  Theorem \ref{Theorem 1} is split into two parts.
The case $s=0$ has to be treated separately.

 \begin{proof}[Proof of Theorem \ref{Theorem 1} for $0<s<1$]
     Applying Proposition \ref{prop:approx_hyper} to the processes $\P_0$ and  $\P_1$ yields two sequences $(\P^N_0)_{N\geq 1}$ and $(\P^N_1)_{N\geq 1}$. 
     Let $\Pi_i^N:=\j(\P^N_i)$, $i=0,1$, be the gap distribution of $\P^N_i$.
     Property \eqref{eq:hyperuni} implies that for every $N$ we have $\Pi_0^N(\s_1\leq 2N)=1$ and hence $\mathscr W_p(\P_0^N,\P^N_1)<\infty$. 
     Let $\U_N\in \Cpl_s(\Pi_0^N,\Pi_1^N)$ be a $\mathscr W_{gap,p}-$ optimal coupling from Theorem \ref{thm:convnex_entropy}. 
     We have the uniform bound 
     \begin{align*}
         \sup_{N\geq 1}\mathcal E(\P^N_i)<\infty.
     \end{align*} 
Then \cite[Lemma 3.7]{erbar2023optimal} implies the uniform integrability assumptions made in Lemma \ref{lem:conv-palm}, and we obtain the weak convergence
\begin{align*}
    \Pi^N_i\xrightarrow{N\to \infty}\Pi_i.
\end{align*}
In particular, for a subsequence $(\U_{N_k})_{k\geq 1}$ we have $\U_{N_k}\xrightarrow{k\to \infty}\U\in \Cpl_s(\Pi_0,\Pi_1)$ weakly.
Let $\Pi^N_t:=(T_t)_{\#}\U_N$ and $\Pi_t:=(T_t)_{\#}\U$, then the weak convergence holds 
$\Pi^{N_k}_t\xrightarrow{k\to \infty}\Pi_t$.
Let $\P_t^N:=\j^{-1}(\Pi^N_t)$  and $\P_t:=\j^{-1}(\Pi_t)$.
We want to show that (up to a subsequence)
\begin{align}\label{eq:xyz1}
    \P_{t}^{N_k}\xrightarrow{k\to \infty}\P_t.
\end{align}
Theorem \ref{thm:convnex_entropy} implies the uniform bound 
\begin{align*}
    \sup_{k\geq 1} \mathcal E(\P^{N_k}_t)<\infty,
\end{align*}
from which we obtain a subsequence  (by weak compactness of sublevel sets of the specific entropy), which for ease of notation we still denote by  $(\P_t^{N_k})_{k\geq 1}$, and a process $\P'\in \mathcal P_{s,1}(\Gamma)$ such that $\P^{N_k}_t\xrightarrow{k\to \infty}\P'$ weakly.
 Moreover, \cite[Lemma 3.7]{erbar2023optimal} shows that 
 \[
    \sup_{k\geq 1}\int \xi(\Lambda_1)\log(\xi(\Lambda_1))\P_t^{N_k}(d\xi)<\infty
    \]
     and we can apply Lemma \ref{lem:conv-palm} to obtain
the weak convergence $\Pi^{N_k}_t=\j(\P^{N_k}_t)\xrightarrow{k\to \infty}\j(\P')$.
Since we already showed $\Pi^{N_k}_t\xrightarrow{k\to \infty}\Pi_t$ weakly, we obtain that $\j(\P')=\Pi_t$ and hence $\P'=\P_t$. This shows \eqref{eq:xyz1}.
To keep notation simple, in the following we will assume the weak convergence $\P_t^{N}\xrightarrow{N\to \infty}\P_t$ and omit the subsequence $(N_k)_{l\geq 1}$. Set $\Q:=\iota^{-1}(\U)$. Recall the notation $\g$ from \eqref{eq:gain}.
By Theorem \ref{thm:convnex_entropy} and Lemma \ref{lem:convex_fix_n}, 
     \begin{align}\label{eq:convex_free_energy}
         &\mathcal F_{\beta}^{mod}(\P^N_{t}) \\&=
       \beta \liminf_{n\to \infty}n^{-1}\left(\mathcal H(\P^N_{t},n)-B_n\right)+ \mathcal E(\P^N_{t}) \nonumber \\
       &\leq \beta \liminf_{n\to \infty}(nN)^{-1}\left(\mathcal H(\P^N_{t},nN)-B_{nN}\right)+ \mathcal E(\P^N_{t})\nonumber\\
       &\leq 
       \beta \liminf_{n\to \infty}(nN)^{-1}\left(\mathcal H(\P^N_{t},nN)-B_{nN}\right)+  (1-t)\mathcal E(\P^N_{0})+t\mathcal E(\P^N_{1})\nonumber\\
       &\leq 
       \beta \liminf_{n\to \infty}(nN)^{-1}\left((1-t)\mathcal H(\P^N_{0},nN)+t\mathcal H(\P^N_{1},nN)-B_{nN}-(nN)\g(nN,t,\Q_N^0)\right)\nonumber\\
       &+(1-t)\mathcal E(\P^N_{0})+t\mathcal E(\P^N_{1})\nonumber\\
       &\leq (1-t)\beta\limsup_{n\to \infty}(nN)^{-1}\left(\mathcal H(\P^N_{0},nN)-B_{nN}\right)+t\beta\limsup_{n\to \infty}(nN)^{-1}\left(\mathcal H(\P^N_{1},nN)-B_{nN}\right)\nonumber\\
       &+(1-t)\mathcal E(\P^N_{0})+t\mathcal E(\P^N_{1})-\beta\liminf_{n\to \infty}\g(nN,t,\Q_N^0)\nonumber\\
       &\leq (1-t)\beta\limsup_{n\to \infty}(nN)^{-1}\left(\mathcal H(\P^N_{0},nN)-B_{nN}\right)+t\beta\limsup_{n\to \infty}(nN)^{-1}\left(\mathcal H(\P^N_{1},nN)-B_{nN}\right)\nonumber\\
       &+(1-t)\mathcal E(\P^N_{0})+t\mathcal E(\P^N_{1})
       -\frac{s(s+1)(1-t)t\beta}{4}\int \sum_{(0,0)\neq (x,y)\in \bar\xi} (x-y)^2\inf_{z\in [x,y]}\abs{z}^{-s-2}\Q_N^0(d\bar\xi),
     \end{align}
where the last line follows by taking $f(\abs{x})=\frac{s(s+1)}{2}\abs{x}^{-s-2}$ (see Example \ref{example:riesz}) and 
\begin{align*}
    \liminf_{n\to \infty}\g(nN,t,\Q_N^0) &= \liminf_{n\to \infty}
    \frac{(1-t)t}{2}\int\sum_{(0,0)\neq (x,y)\in \bar\xi\cap \Lambda_{nN}^2} (x-y)^2\inf_{z\in [x,y]}f(\abs{z})\Q_N^0(d\bar\xi)\\
    &\geq\frac{s(s+1)(1-t)t}{4}\int \sum_{(0,0)\neq (x,y)\in \bar\xi} (x-y)^2\inf_{z\in [x,y]}\abs{z}^{-s-2}\Q_N^0(d\bar\xi).
\end{align*}
A calculation analogous to \eqref{eq:gain_fatou} shows that \begin{align}
    &\lim_{N\to \infty}\int \sum_{(0,0)\neq (x,y)\in \bar\xi} (x-y)^2\inf_{z\in [x,y]}\abs{z}^{-s-2}\Q_N^0(d\bar\xi)\\
    &\geq \int \sum_{(0,0)\neq (x,y)\in \bar \xi} (x-y)^2\inf_{z\in [x,y]}\abs{z}^{-s-2}\Q^0(d\bar\xi)\nonumber
\end{align}
Taking $N\to \infty$ in \eqref{eq:convex_free_energy},  Proposition \ref{prop:approx_hyper}  yields
\begin{align*}
\liminf_{N\to \infty}\mathcal F^{mod}_{\beta}(\P^N_{t})
    &\leq (1-t) \mathcal F^{elec}_{\beta}(\P_{0})+t\mathcal F^{elec}_{\beta}(\P_{1})\\
    &-\frac{s(s+1)(1-t)t\beta}{4}\int \sum_{(0,0)\neq (x,y)\in \bar\xi} (x-y)^2\inf_{z\in [x,y]}\abs{z}^{-s-2}\Q^0(d\bar\xi).
\end{align*}
Moreover, by \cite[Theorem 1]{Lebl__2016} the functional $\mathcal F_{\beta}^{elec}$ is the lower semicontinuous regularization of $\mathcal F^{mod}_{\beta}$ (w.r.t. weak convergence), giving
 \begin{align*}
    \mathcal F_{\beta}^{elec}(\P_{t})\leq \liminf_{N\to \infty}\mathcal F^{mod}_{\beta}(\P^N_{t})
\end{align*}
proving the result.
\end{proof}

\begin{proof}[Proof of Theorem \ref{Theorem 1} for $s=0$]
Using the same notation, the exact same argument as in the case $0<s<1$ yields 
\begin{align*}
\liminf_{N\to \infty}\mathcal F^{mod}_{\beta}(\P^N_{t})
    &\leq (1-t) \mathcal F^{elec}_{\beta}(\P_{0})+t\mathcal F^{elec}_{\beta}(\P_{1})\\
    &-\frac{(1-t)t\beta}{4}
    \int \sum_{(0,0)\neq (x,y)\in \bar \xi} (x-y)^2\inf_{z\in [x,y]}\abs{z}^{-2}\Q^0(d\bar\xi).\nonumber
\end{align*}
However, $\mathcal F_{\beta}^{elec}$ is not the lower semicontinuous regularization of $\mathcal F^{mod}_{\beta}$. From \cite[Theorem 1]{Lebl__2016} it follows that 
$\mathcal F_{\beta}^{elec}$ is the  lower semicontinuous regularization of $\mathcal E+\beta (\mathcal D^{\log}+\mathcal W^{mod})$, with \begin{align*}
        \mathcal D^{\log}(\P):=C^{\log}\limsup_{n\to \infty}\left(n^{-d}\int_{\Lambda_n}\int_{\Lambda_n}(\rho_{2,\P}(x,y)-1)dxdy+1\right)\log(n), \quad \forall \P\in \mathcal P_{s,1}(\Gamma),
    \end{align*}
    where $C^{\log}$ is a constant and $\rho_{2,\P}(x,y)$ is the two-point correlation function of $\P$.
    Since the point processes $\P^N_i$, $i=0,1$, are class-1-hyperuniform (cf. \cite[Proposition 5.3]{Lebl__2016}), 
    Corollary \ref{cor:hyperunif_preserved} shows that also the processes $\P^N_t$, $N\in\IN, t\in [0,1]$, are class-1-hyperuniform. Thus, by \cite[Proposition 5.3]{Lebl__2016} we have that $\mathcal D^{\log}(\P_t^N)=0$.
 Hence, 
\begin{align*}
    \liminf_{N\to \infty}\mathcal F^{mod}_{\beta}(\P^N_{t})&= \liminf_{N\to \infty}\mathcal F^{mod}_{\beta}(\P^N_{t})+\beta D^{\log}(\P_t^N)\\
    &\geq \mathcal F_{\beta}^{elec}(\P_t),
\end{align*}
which proves the claim.
\end{proof}

Finally, as an immediate consequence, we obtain the uniqueness of the minimizer of $\mathcal F_\beta^{elec}$.

\begin{cor}\label{cor:uniquePelec}
    For every $\beta>0$ the functional
    \begin{align}
        \mathcal F_{\beta}^{elec}: \mathcal P_{s,1}(\Gamma)\to \IR \cup \{+\infty\}, \P\mapsto \beta \mathcal W^{elec}(\P)+\mathcal E(\P)
    \end{align}
    has a unique minimizer. 
\end{cor}
\begin{proof}
Since the sublevel sets of the specific relative entropy $\mathcal E$ are compact with respect to the weak topology (cf. \cite[Proposition 6.8]{RAS}) and since $\mathcal F_{\beta}^{elec}$ is lower semicontinuous (see \cite[Theorem 1]{Lebl__2016}), there exists a minimizer.
    Let $\P_i\in \mathcal P_{s,1}(\Gamma)$ be two distinct minimizers with $\mathcal F^{elec}_{\beta}(\P_i)<\infty$, $i=0,1$. Then Theorem \ref{Theorem 1} yields $\P\in \mathcal P_{s,1}(\Gamma)$ with $\mathcal F^{elec}_{\beta}(\P)<\mathcal F^{elec}_{\beta}(\P_0)=\mathcal F^{elec}_{\beta}(\P_1)$,  a contradiction.
\end{proof}
\begin{rem}
   As a consequence of \cite{LeSe17}, for the cases covered in this section we know that $D(\mathcal F_{\beta}^{elec})\neq \emptyset.$ Hence, Corollary \ref{cor:uniquePelec} covers the results of \cite{EHL21} and extends them to the case of long-range Riesz interaction. In particular, we obtain a variational characterisation of log- and Riesz-gases in $d=1$ as unique minimizers of $\mathcal F_\beta^{elec}.$
\end{rem}



\begin{thebibliography}{BDGZ24}

\bibitem[AGS08]{AGS08}
L.~Ambrosio, N.~Gigli, and G.~Savar{\'e}.
\newblock {\em Gradient flows in metric spaces and in the space of probability
  measures}.
\newblock Lectures in Mathematics ETH Z\"urich. Birkh\"auser Verlag, Basel,
  second edition, 2008.

\bibitem[AKR98a]{AKR98}
S.~Albeverio, Y.~G. Kondratiev, and M.~R{\"o}ckner.
\newblock Analysis and geometry on configuration spaces.
\newblock {\em J.~Funct.~Anal.}, 154(2):444--500, 1998.

\bibitem[AKR98b]{AKR98b}
S.~Albeverio, Y.~G. Kondratiev, and M.~R{\"o}ckner.
\newblock Analysis and geometry on configuration spaces: The {G}ibbsian case.
\newblock {\em J.~Funct.~Anal.}, 157(1):242--291, 1998.

\bibitem[AKT84]{AKT84}
M.~Ajtai, J.~Koml{\'o}s, and G{\'a}bor Tusn{\'a}dy.
\newblock On optimal matchings.
\newblock {\em Combinatorica}, 4:259--264, 1984.

\bibitem[AST19]{AmStTr16}
L.~Ambrosio, F.~Stra, and D.~Trevisan.
\newblock A {PDE} approach to a 2-dimensional matching problem.
\newblock {\em {Probab. Theory Relat. Fields}}, 173(1-2):433--477, 2019.

\bibitem[BDGZ24]{BuDaGaZe24}
R.~Butez, S.~Dallaporta, and D.~Garc{\'{\i}}a-Zelada.
\newblock On the {Wasserstein} distance between a hyperuniform point process
  and its mean.
\newblock Preprint, {arXiv}:2404.09549 (2024), 2024.

\bibitem[Bou22]{Bo22}
J.~Boursier.
\newblock Decay of correlations and thermodynamic limit for the circular
  {Riesz} gas.
\newblock Preprint, {arXiv}:2209.00396 (2022), 2022.

\bibitem[Br{\'e}20]{bremaud}
P.~Br{\'e}maud.
\newblock {\em Point process calculus in time and space. {An} introduction with
  applications}, volume~98 of {\em Probab. Theory Stoch. Model.}
\newblock Cham: Springer, 2020.

\bibitem[CLPS14]{CaLuPaSi14}
S.~Caracciolo, C.~Lucibello, G.~Parisi, and G.~Sicuro.
\newblock Scaling hypothesis for the {E}uclidean bipartite matching problem.
\newblock {\em Physical Review E}, 90(1), 2014.

\bibitem[CM24]{ClMa23}
N.~Clozeau and F.~Mattesini.
\newblock Annealed quantitative estimates for the quadratic 2d-discrete random
  matching problem.
\newblock {\em Probab. Theory Relat. Fields}, 190(1-2):485--541, 2024.

\bibitem[CPPR10]{ChPePeRo10}
S.~Chatterjee, R.~Peled, Y.~Peres, and D.~Romik.
\newblock Gravitational allocation to {P}oisson points.
\newblock {\em Ann. of Math. (2)}, 172(1):617--671, 2010.

\bibitem[Der16]{De16}
D.~Dereudre.
\newblock Variational principle for {Gibbs} point processes with finite range
  interaction.
\newblock {\em Electron. Commun. Probab.}, 21:11, 2016.
\newblock Id/No 10.

\bibitem[Der19]{Der19}
D.~Dereudre.
\newblock Introduction to the theory of {G}ibbs point processes.
\newblock {\em Stochastic Geometry: Modern Research Frontiers}, pages 181--229,
  2019.

\bibitem[DFHL25]{DeFlHuLe25}
D.~Dereudre, D.~Flimmel, M.~Huesmann, and T.~Lebl{\'e}.
\newblock ({N}on)-hyperuniformity of perturbed lattices.
\newblock Preprint, {arXiv}:2405.19881 (2025), 2025.

\bibitem[DG09]{DeGe09}
D.~Dereudre and H.-O. Georgii.
\newblock Variational characterisation of {Gibbs} measures with {Delaunay}
  triangle interaction.
\newblock {\em Electron. J. Probab.}, 14:2438--2462, 2009.

\bibitem[DS08]{Daneri_2008}
S.~Daneri and G.~Savar{\'{e} }.
\newblock Eulerian calculus for the displacement convexity in the wasserstein
  distance.
\newblock {\em {SIAM} Journal on Mathematical Analysis}, 40(3):1104--1122, jan
  2008.

\bibitem[DSHS24]{DSHeSu24}
Lorenzo Dello~Schiavo, Ronan Herry, and Kohei Suzuki.
\newblock Wasserstein geometry and {Ricci} curvature bounds for {Poisson}
  spaces.
\newblock {\em J. {\'E}c. Polytech., Math.}, 11:957--1010, 2024.

\bibitem[DVJ08]{Daley}
D.~J. Daley and D.~Vere-Jones.
\newblock {\em An Introduction to the Theory of Point Processes. Volume II:
  General Theory and Structure}.
\newblock Springer, 2008.

\bibitem[Edg77]{Edg}
G.~Edgar.
\newblock Measurability in a {B}anach space.
\newblock {\em Indiana Univ. Math. J.}, 26:663--677, 1977.

\bibitem[EHJM25]{erbar2023optimal}
M.~Erbar, M.~Huesmann, J.~Jalowy, and B.~M{\"u}ller.
\newblock Optimal transport of stationary point processes: metric structure,
  gradient flow and convexity of the specific entropy.
\newblock {\em J. Funct. Anal.}, 289(4):67, 2025.
\newblock Id/No 110974.

\bibitem[EHL21]{EHL21}
M.~Erbar, M.~Huesmann, and T.~Lebl{\'e}.
\newblock The one-dimensional log-gas free energy has a unique minimizer.
\newblock {\em Commun. Pure Appl. Math.}, 74(3):615--675, 2021.

\bibitem[FG21]{figalli2021invitation}
A.~Figalli and F.~Glaudo.
\newblock {\em An invitation to optimal transport, {Wasserstein} distances, and
  gradient flows}.
\newblock EMS Textb. Math. Berlin: European Mathematical Society (EMS), 2021.

\bibitem[FV18]{FrVe18}
S.~Friedli and Y.~Velenik.
\newblock {\em Statistical mechanics of lattice systems. {A} concrete
  mathematical introduction}.
\newblock Cambridge: Cambridge University Press, 2018.

\bibitem[Geo94]{georgii2}
H.-O. Georgii.
\newblock Large deviations and the equivalence of ensembles for {Gibbsian}
  particle systems with superstable interaction.
\newblock {\em Probab. Theory Relat. Fields}, 99(2):171--195, 1994.

\bibitem[Geo11]{Georgii}
H.-O. Georgii.
\newblock {\em Gibbs measures and phase transitions}, volume~9.
\newblock Walter de Gruyter, 2011.

\bibitem[GH22]{GoHu22}
M.~Goldman and M.~Huesmann.
\newblock A fluctuation result for the displacement in the optimal matching
  problem.
\newblock {\em Ann. Probab.}, 50(4):1446--1477, 2022.

\bibitem[GT21]{GoTr21}
M.~Goldman and D.~Trevisan.
\newblock Convergence of asymptotic costs for random {Euclidean} matching
  problems.
\newblock {\em Probab. Math. Phys.}, 2(2):121--142, 2021.

\bibitem[GZ93]{georgii1}
H.-O. Georgii and H.~Zessin.
\newblock Large deviations and the maximum entropy principle for marked point
  random fields.
\newblock {\em Probab. Theory Relat. Fields}, 96(2):177--204, 1993.

\bibitem[HJW22]{HoJaWa22}
A.~E. Holroyd, S.~Janson, and J.~W\"{a}stlund.
\newblock Minimal matchings of point processes.
\newblock {\em Probab. Theory Related Fields}, 184(1-2):571--611, 2022.

\bibitem[HL25]{HuLe25}
M.~Huesmann and T.~Lebl{\'e}.
\newblock The link between hyperuniformity, {Coulomb} energy, and {Wasserstein}
  distance to {Lebesgue} for two-dimensional point processes.
\newblock Preprint, {arXiv}:2404.18588 (2025), 2025.

\bibitem[HM25]{huesmann2024benamoubrenier}
M.~Huesmann and B.~M{\"u}ller.
\newblock A {Benamou}-{Brenier} formula for transport distances between
  stationary random measures.
\newblock {\em Stochastic Processes Appl.}, 185:21, 2025.
\newblock Id/No 104633.

\bibitem[HP05]{HoPe05}
A.~E. Holroyd and Y.~Peres.
\newblock Extra heads and invariant allocations.
\newblock {\em Ann. Probab.}, 33(1):31--52, 2005.

\bibitem[HPPS09]{HPPS09}
A.~E. Holroyd, R.~Pemantle, Y.~Peres, and O.~Schramm.
\newblock {P}oisson matching.
\newblock {\em Ann. Inst. Henri Poincar\'e Probab. Stat.}, 45(1):266--287,
  2009.

\bibitem[HS13]{HS13}
M.~Huesmann and K.-T. Sturm.
\newblock Optimal transport from {Lebesgue} to {Poisson}.
\newblock {\em Ann. Probab.}, 41(4):2426--2478, 2013.

\bibitem[HS25]{HuSt25}
M.~Huesmann and H.~Stange.
\newblock Non-local {Wasserstein} {Geometry}, {Gradient} {Flows}, and
  {Functional} {Inequalities} for {Stationary} {Point} {Processes}.
\newblock Preprint, {arXiv}:2504.12047 (2025), 2025.

\bibitem[Jan18]{Ja18}
S.~Jansen.
\newblock Gibbsian point processes.
\newblock {\em https://www.mathematik.uni-muenchen.de/~jansen/gibbspp.pdf},
  2018.

\bibitem[JKSZ24]{JaKoStZa24}
B.\ Jahnel, J.\ K{\"o}ppl, Y.\ Steenbeck, and A.\ Zass.
\newblock The variational principle for a marked {Gibbs} point process with
  infinite-range multibody interactions.
\newblock Preprint, {arXiv}:2408.17170 (2024), 2024.

\bibitem[Kal97]{Kallenberg}
O.~Kallenberg.
\newblock {\em Foundations of modern probability}, volume~2.
\newblock Springer, 1997.

\bibitem[Leb16]{Lebl__2016}
T.~Lebl{\'e}.
\newblock Logarithmic, {Coulomb} and {Riesz} energy of point processes.
\newblock {\em J. Stat. Phys.}, 162(4):887--923, 2016.

\bibitem[Led17]{Le17}
M.~Ledoux.
\newblock On optimal matching of {G}aussian samples.
\newblock {\em Zap. Nauchn. Sem. S.-Peterburg. Otdel. Mat. Inst. Steklov.
  (POMI)}, 457(Veroyatnost' \ i Statistika. 25):226--264, 2017.

\bibitem[LP18]{Last}
G.~{Last} and M.~{Penrose}.
\newblock {\em {Lectures on the Poisson process.}}, volume~7.
\newblock Cambridge: Cambridge University Press, 2018.

\bibitem[LRY24]{LRY24}
R.~Lachi{\`e}ze-Rey and D.~Yogeshwaran.
\newblock Hyperuniformity and optimal transport of point processes.
\newblock {arXiv}:2402.13705 (2024), 2024.

\bibitem[LS17]{LeSe17}
T.~Lebl{\'e} and S.~Serfaty.
\newblock Large deviation principle for empirical fields of {L}og and {R}iesz
  gases.
\newblock {\em Invent.~math.}, 210(3):645--757, 2017.

\bibitem[LT09]{LaTh09}
G.~{Last} and H.~{Thorisson}.
\newblock {Invariant transports of stationary random measures and
  mass-stationarity.}
\newblock {\em {Ann. Probab.}}, 37(2):790--813, 2009.

\bibitem[Mec67]{Mecke1967StationreZM}
J.~Mecke.
\newblock Station{\"a}re zuf{\"a}llige {Ma{{\ss}}e} auf lokalkompakten
  {Abelschen} {Gruppen}.
\newblock {\em Z. Wahrscheinlichkeitstheor. Verw. Geb.}, 9:36--58, 1967.

\bibitem[Osa12]{Os12}
H.~Osada.
\newblock Infinite-dimensional stochastic differential equations related to
  random matrices.
\newblock {\em Probab.~Theory Related Fields}, 153(3-4):471--509, 2012.

\bibitem[Osa13]{Os13}
H.~Osada.
\newblock Interacting {B}rownian motions in infinite dimensions with
  logarithmic interaction potentials.
\newblock {\em Ann.~Probab.}, 41(1):1--49, 2013.

\bibitem[Pil14]{pilipenko}
A.~Pilipenko.
\newblock {\em An Introduction to {S}tochastic {D}ifferential {E}quations with
  Reflection}.
\newblock Lectures in pure and applied mathematics. University Press Potsdam,
  2014.

\bibitem[PS17]{petrache2017next}
M.~Petrache and S.~Serfaty.
\newblock Next order asymptotics and renormalized energy for {Riesz}
  interactions.
\newblock {\em J. Inst. Math. Jussieu}, 16(3):501--569, 2017.

\bibitem[RAS15]{RAS}
F.~Rassoul-Agha and T.~Sepp{\"a}l{\"a}inen.
\newblock {\em A course on large deviations with an introduction to Gibbs
  measures}, volume 162.
\newblock American Mathematical Soc., 2015.

\bibitem[San15]{Santa}
F.~Santambrogio.
\newblock Optimal transport for applied mathematicians.
\newblock {\em Birk{\"a}user, NY}, 55(58-63):94, 2015.

\bibitem[Ser19]{Serfaty}
S.~Serfaty.
\newblock Microscopic description of log and {Coulomb} gases.
\newblock In {\em Random matrices}, pages 341--387. Providence, RI: American
  Mathematical Society (AMS); Princeton, NJ: Institute for Advanced Study
  (IAS), 2019.

\bibitem[SS12]{sandier2012ginzburg}
{\'E}.~Sandier and S.~Serfaty.
\newblock From the {G}inzburg-{L}andau model to vortex lattice problems.
\newblock {\em Comm.~Math.~Phys.}, 313(3):635--743, 2012.

\bibitem[Suz25]{Su23}
K.~Suzuki.
\newblock Curvature bound of {Dyson} {Brownian} motion.
\newblock {\em Commun. Math. Phys.}, 406(7):56, 2025.
\newblock Id/No 154.

\bibitem[Tsa16]{Ts16}
L.-C. Tsai.
\newblock Infinite dimensional stochastic differential equations for {D}yson's
  model.
\newblock {\em Probab. Theory Related Fields}, 166(3-4):801--850, 2016.

\bibitem[Vil03]{villani2003topics}
C.~Villani.
\newblock {\em Topics in Optimal Transportation}.
\newblock Graduate studies in mathematics. American Mathematical Society, 2003.

\bibitem[Vil09]{Villani}
C.~Villani.
\newblock {\em Optimal transport: old and new}, volume 338.
\newblock Springer, 2009.

\end{thebibliography}

\end{document}